\documentclass[12pt]{amsart}
\usepackage[a4paper,margin=1in]{geometry}

\usepackage{amsmath,amssymb,amsthm,bm}
\usepackage{microtype}
\usepackage[cal=euler]{mathalfa}
\usepackage{tikz,tikz-cd}
\usetikzlibrary{patterns}
\usetikzlibrary{arrows}
\usetikzlibrary{calc,positioning}

\input{commands}
\defcite\GS{goncharov_donaldson} %goncarov-shen
\defcite\KS{kontsevich_stability}
\defcite\KSW{kontsevich_wall}
\defcite\ful{fulton_intersection}
\defcite\bri{bridgeland_stabilityb}
\defcite\bris{bridgeland_stability}
\defcite\wo{woolf_stability}
\defcite\nag{nagao_donaldson}
\defcite\KY{keller_deriveda}
\defcite\joyb{joyce_configurations2}
\defcite\joyc{joyce_configurations3}
\defcite\joym{joyce_motivic}

\newtheorem{assumption}{Assumption}
\forcsvlist\oper{Div,Rat,Alg,Num,Td,NE,Amp,WCS,Slice,
Stab,conv,IH,IE,sgr,Arg}
\forcsvlist\oper{SF,Re,gldim,Stab,Nef,Grpd}
\opr\Sta{St^{\mathrm a}}
\opr\lIC{\lb{IC}}

\def\NSR{N^1_\bR}
\def\bNE{\ub\NE}
\def\toh{\tfrac12}
\def\R{\bfK}
\def\ip{\bi\pi}

\def\num{\mathrm{num}}

\def\st{\mathrm{st}}
\def\cf{cf.~} %confer
\def\Cf{Cf.~} %confer

\usepackage[colorlinks,allcolors=blue]{hyperref}

\begin{document}
\title{Wall-crossing structures on surfaces}

\author{Sergey Mozgovoy}

\address{School of Mathematics, Trinity College Dublin, Dublin 2, Ireland
\newline\indent
Hamilton Mathematics Institute, Dublin 2, Ireland}

\email{mozgovoy@maths.tcd.ie}

\begin{abstract}
Families of Bridgeland stability conditions induce families of stability data, wall-crossing structures and scattering diagrams on the motivic Hall algebra.
These structures can be transferred to the quantum torus if the stability conditions of the family have global dimension at most 2.
We show that geometric stability conditions on surfaces with nef anticanonical bundle have global dimension 2 and we study the resulting family of stability data.
We formulate a conjecture relating this family to the family of stability data associated with a quiver with potential and we verify this conjecture for the projective plane.
\end{abstract}

\maketitle
\tableofcontents
\section{Introduction}
Bridgeland stability conditions on triangulated categories were introduced in \bris as a mathematical incarnation of the notion of $\Pi$-stability in string theory \cite{douglas_geometry,aspinwall_da}.
One can define the notion of a continuous family of stability conditions and consider the universal object in the category of all such families.
This universal object is given by the set of all stability conditions equipped with a natural topology.
It has a structure of a complex manifold \bris. 
\medskip

On the other hand, one can also define the notion of a stability condition on an abelian category \cite{joyce_configurations1,bridgeland_stability}.
Under appropriate assumptions, one can define moduli spaces and moduli stacks of semistable objects of the category  \cite{joyce_configurations1,joyce_configurations2,joyce_configurations3,joyce_configurations4,joyce_motivic,joyce_theory} and consider their invariants, known as Donaldson-Thomas (DT) invariants or BPS invariants.
The behavior of these invariants is controlled by a formula, called the wall-crossing formula \cite{joyce_configurations2,joyce_theory}, reflecting the fact that DT invariants change their values when stability condition crosses certain codimension one subspaces of the space of all stability conditions.
\medskip

The above two approaches were combined in \cite{kontsevich_stability}, where one constructed DT invariants associated to stability conditions on appropriate triangulated 3CY categories.
Moreover, in \cite{kontsevich_stability} one conceptualized the structure formed by the collection of all DT invariants associated to a given stability condition by introducing the framework of stability data on graded Lie algebras
(stability data $(Z,a)$ consists of a central charge $Z:\Ga\to\bC$, where ~$\Ga$ is the grading lattice, and a collection $a=(a_\ga)_{\ga\in\Ga}$ of elements of the Lie algebra).
In the case of unrefined DT invariants, the corresponding Lie algebra is the Lie algebra of a Poisson torus,
and in the case of motivic DT invariants, the corresponding Lie algebra is the Lie algebra of a quantum torus.
One can think about stability data on a graded Lie algebra 
as an analogue of the notion of a stability condition on a triangulated category.
As in the case of stability conditions, one can introduce the notion of a continuous family of stability data on a graded Lie algebra \cite{kontsevich_stability}
(also called a variation of BPS structures \cite{bridgeland_riemann,bridgeland_geometry}).
%The elegance of the above framework is based on the fact that 
Importantly, the wall-crossing formula is intrinsically embedded into the notion of continuity of a family of stability data
(the Kontsevich-Soibelman wall-crossing formula).
Given a continuous family of stability conditions on an appropriate triangulated 3CY category, one can define a continuous family of stability data on the corresponding quantum torus \cite{kontsevich_stability}.
\medskip

There exist alternative ways to encode a continuous family
%$(Z_x,a_x)_{x\in M}$ (here $Z_x$ is a central charge and $a_x$ is a collection of elements of a Lie algebra)
of stability data on a graded Lie algebra.
%(here $Z_x:\Ga\to\bC$ is a linear map on the grading lattice $\Ga$ and $a_x=(a_{x,\ga})_{\ga\in\Ga}$ is a collection of elements of the graded Lie algebra).
One of them is a wall-crossing structure \cite{kontsevich_wall} which captures information of stability data along a fixed ray in $\bC$ (usually $\bi\bR_{>0}$ or $-\bR_{>0}$). 
This new structure may carry less information than the original family of stability data unless the central charge rotates in the family.
In the case of a wall-crossing structure on a vector space, it can be effectively and intuitively represented by a scattering diagram \cite{kontsevich_wall,kontsevich_affine,gross_real}.
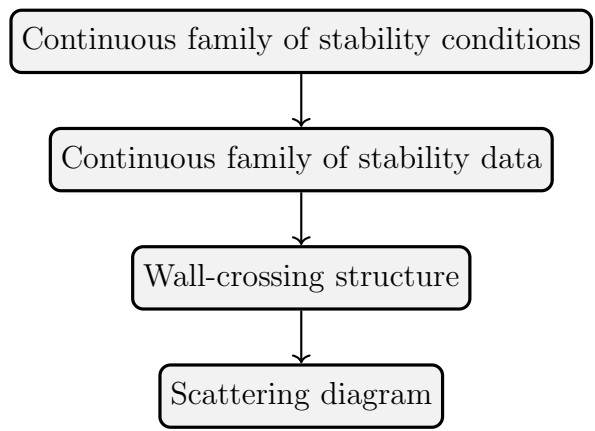
\vspace{.5cm}

\tikzset{
  ar1/.style={->, thick},
  rct1/.style={
	  rectangle, rounded corners, 
	  draw=black, very thick,
	  fill=black!5,
	  minimum height=2em, 
	},
}
\begin{ctikz}
\node[rct1](A){Continuous family of stability conditions};
\node[rct1,below=.7cm of A](B)
{Continuous family of stability data};
\node[rct1,below=.7cm of B](C){Wall-crossing structure};
\node[rct1,below=.7cm of C](D){Scattering diagram};
\draw[ar1](A)--(B);
\draw[ar1](B)--(C);
\draw[ar1](C)--(D);
\end{ctikz}
\vspace{.5cm}

While the above hierarchy seems to be rather natural, in some situations (for example, in the case of quivers with potential) there is no need to go through all of the above steps and one can go directly from a continuous family of stability conditions (parametrized by a vector space) to a scattering diagram \cite{gross_canonical,bridgeland_scattering,mou_scattering}.
The applications of this approach to the theory of cluster algebras are ubiquitous.
The theory of generalized scattering diagrams, associated to wall-crossing structures on manifolds, has yet to be developed (\cf \cite{bousseau_scattering}).
\medskip

In this paper we will give a detailed introduction to stability conditions, stability data and relationship between them.
We will develop some tools to treat both of the above notions in a unified manner.
For example, it was observed in \cite{bayer_space,bayer_short} that the support property with respect to a quadratic form can be extended from a stability condition to its \nbd.
The same property for stability data was made one of the axioms in the original definition of a continuous family of stability data \KS[\S2.3].
We will show that this axiom follows automatically from a weaker support property. 
In doing so, we introduce the notion of a continuous family of stable supports and prove that the support property for such family with respect to a fixed quadratic form is an open condition.
The corresponding statements for stability conditions and stability data then follow immediately.

\medskip
It was mentioned above that one can associate stability data with a stability condition on a triangulated 3CY category \KS. 
On the other hand, one can associate stability data with a stability condition on an abelian category of homological dimension one \cite{joyce_configurations2}.
The reason for these restrictions on a category is that the precursor of the wall-crossing formula, the Harder-Narasimhan identity in the Hall algebra of the corresponding triangulated or abelian category, can not be transferred to the quantum torus in general.
In this paper we will show that this still can be done if a stability condition $\si=(Z,\cP)$ has global dimension ~$\le 2$.
The global dimension is a real invariant of a slicing $\cP$ \cite{ikeda_q,qiu_global} which measures the distance between interacting slices ~\S\ref{sec:global dim}.
We will show that stability data associated with such a stability condition satisfies the wall-crossing formula
(Corollary ~\ref{cor:WC2}).
We will also show that a continuous family of such stability conditions induces a continuous family of stability data (Theorem \ref{th:fam sd from sc}).
Morally, these results are similar to \cite[\S6.4]{joyce_configurations4}, where one proves a wall-crossing formula for Gieseker-semistable sheaves on a surface having a nef anticanonical bundle.
\medskip

Our next goal is to find meaningful examples in which the above results can be applied.
Given a smooth projective surface $X$, we say that a stability condition $\si=(Z,\cP)$ 
on $\cD=D^b(\Coh X)$ is special geometric if the sheaves $\cO_x$, for $x\in X$, are $\si$-stable of the same phase and if the discriminant~ \eqref{disc1} is negative-definite on $\Ker Z$.
For any pair of real divisors $\be,\om$ with ample~ $\om$, one can define some special geometric stability condition ~$\si_{\be,\om}$~ \cite{arcara_bridgeland,bridgeland_stabilityb}. 
We will prove in Theorem \ref{th:unique shift} that every special geometric stability condition is equal to some $\si_{\be,\om}$ up to the action of $\wtl\GL_2^+(\bR)$ on the space of stability conditions
(\cf \cite{bridgeland_stabilityb}).
The space of special geometric stability conditions is an open subset of the space of all stability conditions.
\medskip

We will prove that if $X$ is a surface having a nef anticanonical divisor, then the global dimensional of any special geometric stability condition is equal to $2$ (Theorem \ref{th:vanish3}).
In the case of $X=\bP^2$ this result was proved in \cite{fan_contractibility} based on the method developed in \cite{li_smoothness}.
This method, for a stability condition $\si_{\be,\om}$, relies on the fact that $K_X$ is proportional to $\om$ and it doesn't seem to generalize to other surfaces.
In our approach we develop a method to control the behavior of mini/max phases of objects under deformations of stability conditions (Theorem \ref{th:phase behave}) and we show that this method can be applied in our situation by proving a new Bogomolov-type inequality for $\si_{\be,\om}$-semistable objects (Theorem \ref{th:BTI}).
In Theorem \ref{th:vanish1} we prove that given two semistable objects $E,\,F$ with $\vi(E)<\vi(F)$ and a nef line bundle $L$, we have
\begin{equation}
\Hom(F\ts L,E)=0.
\end{equation}
This statement is trivial for Gieseker-semistable sheaves, but it becomes more involved in the case of $\si_{\be,\om}$-semistable objects as the object $F\ts L$ is not necessarily $\si_{\be,\om}$-semistable.
\medskip 

The above results imply that for any smooth projective surface $X$ with a nef anticanonical bundle we have a continuous family of stability data parametrized by the space $\Stab^+(\cD)$ of special geometric stability conditions.
Stability conditions contained in the same $\wtl\GL_2^+(\bR)$-orbit produce equivalent stability data, hence it is enough to consider the family of stability data parametrized by pairs of divisors $(\be,\om)$ with ample $\om$.
In the case of $X=\bP^2$, one defined in~ \cite{bousseau_scattering} the corresponding scattering diagram by using the interpretation of DT invariants through intersection cohomology (see \S\ref{sec:inters}).
This approach relies on the equation $\Ext^2(E,E)=0$ for semistable objects $E$, hence one needs to restrict the possible phases (the equation is not true, for example, for $E=\cO_x$).
As we mentioned earlier, the wall-crossing structure (and the corresponding scattering diagram) associated to a continuous family of stability data does not capture in general the whole information of the family as it only preserves stability data along one ray, for every point of the family.
On the other hand, this scattering diagram captures a significant part of the family of stability data and it is fully described in \cite{bousseau_scattering}.
\medskip

Another topic that we address in this paper is the relationship of the above stability data for surfaces and stability data for quivers with potentials.
There exist many examples (see \eg \cite{beaujard_vafa})
of smooth projective surfaces $X$ with a nef anticanonical bundle such that $D^b(\Coh X)$ admits a tilting objects $T'$ that can be lifted to a tilting object $T$ on the canonical bundle $Y=\om_X$ considered as a non-compact 3CY variety (see \S\ref{sec:except}).
In these examples, the algebra $\End(T)^\op$ can be interpreted as the Jacobian algebra $J_W$ of some quiver with potential $(Q,W)$,
and the algebra $\End(T')^\op$ can be interpreted as the partial Jacobian algebra $J_{W,I}$ for a certain cut $I\sbs Q_1$ (see \S\ref{sec:Jacobian}).
By the tilting theory we have
\begin{equation}
D^b(\Coh X)\iso D^b(\mmod J_{W,I}),\qquad D^b_c(\Coh Y)\iso
D^b(\mmod J_{W}),
\end{equation}
where $D^b_c(\Coh Y)$ denotes the derived category of sheaves with compact support.
The first equivalence implies that stability conditions and stability data that we constructed earlier, can as well be associated with the algebra $J_{W,I}$.
\medskip

For a del Pezzo surface, under appropriate assumptions on stability conditions on $D^b_c(\Coh Y)$ and the phases of semistable objects, one can expect that these semistable objects are push-forwards of semistable objects in $D^b(\Coh X)$.
Then one can expect that some components of stability data for $D^b_c(\Coh Y)$ coincide with components of stability data for $D^b(\Coh X)$.
This is not what we are going to do.
\medskip

Given a quiver with potential $(Q,W)$ and a stability function $Z:\bZ Q_0\to\bC$, one can define the corresponding stability data $A_Z^{Q,W}$ in the quantum torus with coefficients in the (localized) Grothendieck ring of varieties with exponentials \S\ref{sec:Jacobian}.
One could invoke instead monodromic mixed Hodge structures \cite{kontsevich_cohomological}, but this would be unnecessary.
In the presence of a cut, the coefficients are contained in the usual (localized) Grothendieck ring of varieties.
Morally, stability data $A_Z^{Q,W}$ should correspond to some stability condition on the derived 3CY category of DG modules over the Ginzburg DG algebra associated to $(Q,W)$.
The above stability data enjoys the wall-crossing formula (Theorem \ref{WC-potential}) and the family of such stability data (parametrized by stability functions $Z:\bZ Q_0\to\bC$) is continuous.
\medskip

Let us assume now that $(Q,W)$ and a cut $I\sbs Q_1$ arise from the tilting theory on a surface~$X$ with a nef anticanonical bundle, as we discussed earlier.
Then we have a family of stability data for $D^b(\Coh X)$ (parametrized by special geometric stability conditions) and a family of stability data for $(Q,W)$ (parametrized by stability functions $Z:\bZ Q_0\to\bC$).
The quantum tori, where these families live, turn out to be the same.
We conjecture that the above families are compatible, meaning that there exists a special geometric stability condition $\si=(Z,\cP)$ on $D^b(\Coh X)\iso D^b(\mmod J_{W,I})$ such that its heart is equal to $\mmod J_{W,I}$ (so that we have a stability function $Z:\bZ Q_0\to \bC$) and such that stability data $A_\si$ corresponding to $\si$ is equal to $A_Z^{Q,W}$.
This conjecture looks somewhat surprising as it gives a strong compatibility of DT invariants on a surface and on its canonical bundle.
We verify this conjecture for $X=\bP^2$ in ~\S\ref{sec:gluing}.
In other examples from ~\cite{beaujard_vafa} we checked that, as long as there exists a special geometric stability condition with the heart equal to $\mmod J_{W,I}$, the conjecture is true (see Remark \ref{phenom}).
This is a manageable problem which will be addressed elsewhere.
The above compatibility means that it is enough to know stability data $A_Z^{Q,W}$ for some $Z$ in order to determine stability data $A_\si$ for all special geometric stability conditions~ $\si$.
A complete (conjectural) description of stability data $A_Z^{Q,W}$ exists for $\bP^2$ and other surfaces 
\cite{beaujard_vafa,mozgovoy_attractor}.
It is fully verified for small dimension vectors.

%On the other hand, it seems that computations in \cite{beaujard_vafa} were implicitly based on such compatibility. 

\subsection*{Plan of the paper}
%This paper is organized as follows.
In \S\ref{sec:cones} and \S\ref{sec:str semigr} we introduce strict cones and strict semigroups and study their properties.
This will be needed later in the definition of stability data and wall-crossing formulas.
In \S\ref{sec:sgr-fam} we introduce families of strict semigroups and study the behavior of their lower sets under deformations.
In \S\ref{secCQ} we study the relationship between strict cones and quadratic forms.
In \S\ref{sec:supp famil} we introduce the notion of a continuous family of stable supports and prove a result about the behavior of such families with respect to quadratic forms.
This result has an important applications in the case of stability conditions (Theorem~\ref{th:Qpos2}) and in the case of stability data (Theorem~\ref{th:Qpos3}).

In \S\ref{sec:stab cond} we give a detailed introduction to Bridgeland stability conditions.
In particular, in Theorem \ref{th:Qpos2} we give a proof of the fact that the support property with respect to a fixed quadratic form is an open condition
and in \S\ref{sec:phase behavior} we introduce a method to control the behavior of mini/max phases under deformations.

In \S\ref{sec:stab data} we give a detailed introduction to stability data.
In particular, in \S\ref{sec:fam stab data} we introduce continuous families of stability data with a weakened support axiom and we prove in Theorem \ref{th:Qpos3} that our definition is equivalent to the original definition of \KS[\S2.3].
In ~\S\ref{sec:WCS} we introduce wall-crossing structures and discuss their relation to continuous families of stability data.

In \S\ref{sec:WCF} we discuss various wall-crossing formulas.
In \S\ref{sec:Gr ring of stacks} we introduce Grothendieck rings of algebraic varieties and stacks
and in \S\ref{sec:motivic Hall} we introduce motivic Hall algebras of exact categories.
In \S\ref{sec:WC hall} we prove the usual wall-crossing formula in the motivic Hall algebra for an appropriate stability condition on a triangulated category.
In \S\ref{sec:wc qt} we prove the wall-crossing formula in the quantum torus under the assumption that our stability condition has global dimension $\le2$.
In \S\ref{sec:sd for sc} and \S\ref{sec:fam of SD for SC} we construct continuous families of stability data associated with continuous families of stability conditions having global dimension $\le2$.

In \S\ref{sec:stab cond surf} we study geometric stability conditions on surfaces.
In particular, in \S\ref{sec:GSC} we show that every
special geometric stability condition is equal to some $\si_{\be,\om}$ up to the action of $\wtl\GL_2^+(\bR)$ on the space of stability conditions.
In \S\ref{sec:large volume} we discuss the relationship of the large volume limit with Gieseker and twisted stability.
In \S\ref{sec:bog-type} we prove a new Bogomolov-type inequality.
In \S\ref{altern param} we discuss alternative ways to parametrize geometric stability conditions.

In \S\ref{sec:SD for GSC} we study stability data associated to special geometric stability conditions.
In ~\S\ref{vanishing} we prove various vanishing results for semistable objects \wrt geometric stability conditions and we prove that special geometric stability conditions on surface with a nef anticanonical bundle have global dimension $2$.
In \S\ref{sec:GSD} we construct a continuous family of geometric stability data.
In \S\ref{sec:inters} we discuss a relation of geometric stability data to intersection cohomology of moduli spaces of semistable objects.

In \S\ref{sec:rel QP} we discuss a relation of geometric stability data to stability data associated to quivers with potentials.
In \S\ref{sec:except} we briefly discuss basic tilting theory for smooth projective algebraic varieties and their canonical bundles.
In \S\ref{sec:rel QP plane} we illustrate this approach on the example of $\bP^2$.
In \S\ref{sec:Jacobian} we introduce quivers with potentials, their Jacobian algebras, partial Jacobian algebras, and their motivic invariants (stability data).
In \S\ref{relating SD} and \S\ref{sec:gluing} we prove that the family of geometric stability data on $\bP^2$ and the family of stability data associated to the corresponding quiver with potential are compatible.

\subsection*{Acknowledgments}
I would like to thank
Pierrick Bousseau,
Chunyi Li,
Emanuele Macr\`i,
Jan Manschot,
Boris Pioline,
Yu Qiu,
and Yan Soibelman
for useful discussions.
I would like to thank 
Boris Pioline
for helpful feedback on a draft version of the paper.

\section{Cones and semigroups}
\label{sec:cone-semigr}

Given an abelian group $M$ and a commutative ring $R$, we define $M_R=M\ts_\bZ R$.
Given a module $M$ over a commutative ring $R$, we define  $M\dual=\Hom_R(M,R)$.
%We define $\bR_{>0}=\sets{x\in\bR}{x>0}$.

\subsection{Cones}
\label{sec:cones}
Let $E$ be a real vector space of dimension $n>0$.
We equip it with a norm $\nn\cdot$.
For any subset $S\sbs E$, let $\conv(S)$ denote its 
\idef{convex hull}.
By Carath\'eodory's theorem, every element of $\conv(S)$ can be written as a convex combination of $n+1$ points in $S$.
This implies that if $S\sbs E$ is compact, then $\conv(S)$ is also compact.

A subset $C\sbs E$ is called a \idef{cone} if $\bR_{>0}C\sbs C$ and it is called a \idef{convex cone} if also $C+C\sbs C$.
We will say that a cone $C$ is \idef{blunt} if $0\notin C$.
We define a \idef{ray} to be a blunt cone of dimension~$1$.
%We define a \idef{sector} to be a convex cone in $\bC$ having angle $<\pi$.
%We will say that a sector is \idef{acute}
%if its angle is $<\pi/2$.
%Our sectors will be blunt, unless stated otherwise.
For any subset $S\sbs E$, we define the 
\idef{conical hull} $\cone(S)$ of $S$
(or the  convex cone generated by ~$S$)
to be the minimal convex cone that contains~ $S$.
Explicitly,
\begin{equation}
\cone(S)=\conv(\bR_{>0}S)
%=\bR_{>0}\conv(S)
=\sets{\sum_{i=1}^k a_ix_i}{a_i\in\bR_{>0},\, x_i\in S,\,k\ge1}.
\end{equation}
Note that traditionally one also includes $0$ in $\cone(S)$.
%We define the \idef{closed conical hull} $\ccone(S)$ to be the closure of $\cone(S)$.
We will say that a convex cone $C\sbs E$ is \idef{strict} if its closure $\bar C$ does not contain a line 
or, equivalently, if $\bar C\ms\set0$ is convex.

\begin{lemma}\label{lm:strict}
A convex cone $C\sbs E$ is strict if and only if there exists $u\in E\dual$ such that
\begin{equation}
C\sbs C_u:=\sets{x\in E}{u(x)\ge \nn x}.
\end{equation}
\end{lemma}
\begin{proof}
It is clear that $C_u$ is a strict closed cone.
Conversely, let $C$ be a strict cone.
We can assume that it is closed.
Then the set $C_1=\sets{x\in C}{\nn x=1}\sbs C\ms\set0$ is compact,
hence the convex hull $\conv(C_1)\sbs C\ms\set0$ is also compact.
By the Hahn-Banach separation theorem, there exists $u\in E\dual$ and $\eps>0$ such that $u(x)>\eps$ for all $x\in \conv(C_1)$.
This implies that $\frac1\eps u(x)\ge\nn x$ for all $x\in C$.
\end{proof}

\begin{lemma}\label{lm:C VZe}
Let $Z:E\to F$ be a linear map between two finite-dimensional normed vector spaces.
For any strict cone $V\sbs F$ and $\eps>0$,
consider the convex cone
\begin{equation}\label{S VZe}
C(V,Z,\eps)=\cone\sets{x\in E}{Z(x)\in V,\, \nn{Z(x)}\ge\eps\nn x}.
\end{equation}
Then $C(V,Z,\eps)$ is strict and there exists $\eps'>0$ (independent of $Z$) such that
\begin{equation}
\nn{Z(x)}\ge\eps'\nn x\qquad \forall x\in C(V,Z,\eps).
\end{equation}
\end{lemma}
\begin{proof}
There exists $0\ne u\in F\dual$ such that 
$V\sbs C_u=\sets{y\in F}{u(y)\ge\nn y}$.
If $Z(x)\in V$ and $\nn{Z(x)}\ge\eps\nn x$,
then $uZ(x)\ge\nn {Z(x)}\ge \eps\nn x$.
This implies that $C(V,Z,\eps)$ is contained in the strict cone $C'=\sets{x\in E}{uZ(x)\ge \eps\nn x}$.
Let $\nn u>0$ be the operator norm of the bounded operator $u:F\to\bR$.
For any $x\in C'$, we have 
$\eps \nn x\le\n{u Z(x)}\le \nn u\cdot\nn{Z(x)}$
and we choose $\eps'=\eps\nn u\inv$.
\end{proof}

Note that for a ray $\ell\sbs F$, we have
\begin{equation}
C(\ell,Z,\eps)=\sets{x\in E}{Z(x)\in \ell,\, \nn{Z(x)}\ge\eps\nn x}
\end{equation}
as the right hand side is automatically convex.

\begin{lemma}\label{strict sums}
Let $C\sbs E$ be a strict cone.
Then there exists $c>0$ such that
$$\sum\nn {x_i}\le c\nn {\sum x_i}\qquad \forall x_1,\dots,x_k\in C.$$
\end{lemma}
\begin{proof}
Let $u\in E\dual$ be such that $C\sbs C_u$.
For any $x_1,\dots,x_k\in C$, we have
$$\sum\nn {x_i}\le \sum u(x_i)
%= u\rbr{\sum x_i}
\le \nn u\cdot \nn{\sum x_i},$$
where $\nn u$ is the operator norm of $u:E\to\bR$.
\end{proof}

\begin{example}\label{ex:norm sum}
Let $V\sbs \bC$ be a strict cone of angle $\theta<\tfrac\pi2$.
Then
$$\sum\n {x_i}\le \frac1{\cos(\theta)}\n {\sum x_i}\qquad \forall x_1,\dots,x_k\in V.$$
Note also that $\n{x_j}\le\n{\sum_i x_i}$ for any $1\le j\le k$.
\end{example}

\subsection{Strict semigroups}
\label{sec:str semigr}
Let $S$ be a commutative semigroup, written additively.
We define an \idef{ideal} of  $S$ to be a subset $I\sbs S$ such that $I+S\sbs I$.
For any subset $X\sbs S$, we define the semigroup generated by $X$ to be
\begin{equation}
\sgr(X)=\sets{a_1+\dots+a_k}{a_i\in X,\, k\ge1}\sbs S.
\end{equation}
If $S$ has \idef{cancellation property}, meaning that $a+b=a+c$ implies $b=c$,
then we can embed ~$S$ into an abelian group generated by $S$ (the Grothendieck group of $S$).
In what follows we will always assume that $S$ has cancellation property and satisfies $S\cap(-S)\sbs\set0$.
%We will say that a commutative semigroup $S$ with cancellation is salient if $S\cap(-S)\sbs\set{0}$.
In this case we define the partial order on $S$
\begin{equation}
a\le b\iff b-a\in S\cup\set0.
\end{equation}

A subset $I\sbs S$ is an ideal if and only if
it is an \idef{upper set} of the poset $S$, meaning that $a\le b$ and $a\in I$ imply $b\in I$.
Similarly, a subset $I\sbs S$ is the complement of an ideal if and only if it is a \idef{lower set} of the poset $S$, meaning that $a\le b$ and $b\in I$ imply $a\in I$. 

Let $\Ga$ be a free abelian group of finite rank and $S\sbs\Ga$ be a semigroup.
We will say that $S$ is \idef{strict} if it generates a strict convex cone in $\Ga_\bR$.
Then $S$ has automatically the cancellation property and satisfies $S\cap (-S)\sbs\set0$.
Recall that a map between topological spaces is called \idef{proper} if the preimage of every compact set is compact.
In particular, a map $u:S\to\bR$ (where $S$ has discrete topology) is proper if and only if the preimage of every bounded set in $\bR$ is finite.

\begin{lemma}\label{lm:strict sgr}
Let $S\sbs \Ga$ be a strict semigroup without zero. Then
\begin{enumerate}
\item There exists an additive proper map $u:S\to\bR$
such that $u(S)\sbs\bR_{>0}$.
%$\sets{a\in S}{u(a)\le N}$ is finite $\forall N>0$.

% $\inf_{a\in S}u(a)>0$.
\item For any $a\in S$, the lower set $\loc a=\sets{b\in S}{b\le a}$ is finite.
\item For any $a\in S$, the set 
$\sets{(a_1,\dots,a_k)\in S^k}{\sum a_i=a,\,k\ge1}$ 
is finite.
\end{enumerate}
\end{lemma}
\begin{proof}
Let $\nn\cdot$ be a norm on $\Ga_\bR$.
By Lemma \ref{lm:strict}, there exists a linear map $u:\Ga\to\bR$ 
such that $u(a)\ge \nn a$ for all $a\in S$.
This map satisfies the required properties.
%This implies that $u(a)>0$ for all $a\in S$ (as $0\notin S$).
%For any $a_0\in S$, the set $S_0$ consisting of $a\in S$ such that $\nn a\le u(a)\le u(a_0)$ is finite,
%hence $\inf_{a\in S}u(a)=\inf_{a\in S_0}u(a)>0$.
The other statements of the lemma follow immediately.
\end{proof}

Note that if $u:S\to\bR$ is an additive proper map such that $u(S)\sbs\bR_{>0}$, then the subset
\begin{equation}
I_N=\sets{a\in S}{u(a)\le N},\qquad N\ge1,
\end{equation}
is a finite lower set.
Moreover, every finite lower subset of $S$ is contained in $I_N$ for some $N\ge1$.

\subsection{Families of strict semigroups}
\label{sec:sgr-fam}
Let as before $\Ga$ be a free abelian group of finite rank and let $\nn\cdot$ be a norm on $\Ga_\bR$.
For any strict cone $V\sbs\bC$, $Z\in\Ga_\bC\dual=\Hom(\Ga,\bC)$ and $\eps>0$, we define the semigroup
\begin{equation}\label{S1}
S(V,Z,\eps)=\sgr\sets{a\in \Ga}{Z(a)\in V,\, \n{Z(a)}\ge\eps\nn a}\sbs\Ga
\end{equation}
which is strict by Lemma \ref{lm:C VZe}.
%\begin{gather}
%S^\circ(V,Z,\eps)=\sets{a\in \Ga}{Z(a)\in V,\, \n{Z(a)}>\eps\nn a},\\
%S(V,Z,\eps)=\sgr S^\circ(V,Z,\eps),\qquad
%C(V,Z,\eps)=\cone S^\circ(V,Z,\eps).
%\end{gather}
Note that for a ray $\ell\sbs\bC$, we have
\begin{equation}\label{S ray}
S(\ell,Z,\eps)=\sets{a\in \Ga}{Z(a)\in \ell,\, \n{Z(a)}\ge\eps\nn a}.
\end{equation}

\begin{lemma}
Let $M$ be a topological space, $Z:M\to\Ga_\bC\dual$ be a continuous map and $x\in M$.
Let $\eps>0$ and $V\sbs\bC$ be a strict cone such that
$$Z_x(\Ga)\cap\dd V=\set0,\qquad 
\n{Z_x(a)}\ne\eps\nn a\quad
\forall 0\ne a\in\Ga.$$
Then, for any finite set $D\sbs\Ga$, there exists an open set $x\in U\sbs M$ such that
$$D\cap S(V,Z_x,\eps)
=D\cap S(V,Z_y,\eps)\qquad\forall y\in U.$$
\end{lemma}
\begin{proof}
It is enough to assume that $D=\set\ga$.
Let us define
$$ S^\circ(V,Z,\eps)=\sets{a\in \Ga}{Z(a)\in V,\, \n{Z(a)}\ge\eps\nn a}$$
so that $S(V,Z,\eps)$ is the semigroup generated by $ S^\circ(V,Z,\eps)$.
It follows from our assumptions that for any finite set $D'\sbs\Ga$, there exists an open set $x\in U\sbs M$ such that
$$D'\cap  S^\circ(V,Z_x,\eps)=D'\cap S^\circ(V,Z_y,\eps)\qquad\forall y\in U.$$

If $\ga\in S(V,Z_x,\eps)$, then $\ga=\sum_i a_i$ with $a_i\in  S^\circ(V,Z_x,\eps)$. 
By the previous remark, there exists an open set $x\in U\sbs M$
such that $\set{a_i}_i\sbs  S^\circ(V,Z_y,\eps)$ for all $y\in U$.
This implies that $\ga\in S(V,Z_y,\eps)$ for all $y\in U$,
proving one of the required inclusions.

Let us prove the other inclusion.
There exists an open set $x\in U\sbs M$ such that
$\nn{Z_x-Z_y}\le\toh\eps$ for all $y\in U$.
If $\n {Z_y(a)}\ge\eps\nn a$, then
$\n {Z_x(a)}\ge\toh\eps\nn a$.
Therefore 
$$ S^\circ(V,Z_y,\eps)\sbs S^\circ(V,Z_x,\toh\eps)\sbs S(V,Z_x,\toh\eps).$$
By Lemma \ref{strict sums}, there exists $c>0$ such that $\sum\nn{a_i}\le c\nn{\sum a_i}$ for all $a_i\in S(V,Z_x,\toh\eps)$.
%For a fixed $\ga\in\Ga$, 
Let us consider the finite set $D'=\sets{a\in\Ga}{\nn a\le c\nn\ga}$.
%By the previous remark, 
Then we can shrink $U$ so that
$D'\cap  S^\circ(V,Z_x,\eps)=D'\cap S^\circ(V,Z_y,\eps)$ for all $y\in U$.

If $\ga\in S(V,Z_y,\eps)$ for some $y\in U$, then $\ga=\sum a_i$ with $a_i\in S^\circ(V,Z_y,\eps)\sbs S(V,Z_x,\toh\eps)$, hence $\nn {a_i}\le c\nn{\sum a_i}=c\nn\ga$ and $a_i\in D'$.
This implies that $a_i\in  S^\circ(V,Z_x,\eps)$, hence $\ga\in S(V,Z_x,\eps)$.
\end{proof}

%Let us assume that $\ga\not\in S(V,Z_x,\eps)$.
%By Lemma \ref{}, there exists $\eps_1>0$ (that depends only on $V$ and $\eps$) such that if $a\in S(V,Z_y,\eps)$ for some $y\in M$, then $\n{Z_y(a)}>\eps_1\nn{a}$.
%We can assume that $x\in U\sbs M$ is such that $\nn{Z_x-Z_y}\le\toh\eps_1$ for all $y\in U$.
%If $a\in S(V,Z_y,\eps)$, then $\n{Z_y(a)}>\eps_1\nn{a}$, hence $\n{Z_x(a)}>\toh\eps_1\nn{a}$ and $a\in S(V,Z_x,\toh\eps_1)$.
%If $\ga\in S(V,Z_y,\eps)$, then $\ga=\sum a_i$ with $a_i\in\bar S(V,Z_y,\eps)$.
%There is a constant $c$ (independent of $y$) such that $\sum\nn a_i\le c\nn{\sum a_i}=c\nn\ga$.
%The set $D=\sets{a\in\Ga}{\nn a\le c\nn\ga}$ is finite, hence we can assume that $D\cap \bar S(V,Z_y,\eps)=D\cap\bar S(V,Z_x,\eps)$.
%Similarly, $S(V,Z_y,\eps)\sbs S(V,Z_x,\toh\eps)$.
%such that $Z_x(\ga_i)\in V$ and $\n{Z_x(\ga_i)}>\eps\nn\ga$.
%Then there exists open $x\in U\sbs M$ such that
%\begin{enumerate}
%\item 
%if $\ga\in S(V,Z_x,\eps)$, then $\ga\in S(V,Z_y,\eps)$ for all $y\in U$.
%\item 
%if $\ga\notin S(V,Z_x,\eps)$, then $\ga\notin S(V,Z_y,\eps)$ for all $y\in U$.
%\end{enumerate}
%\end{claim}

\begin{lemma}\label{lower set stability}
Let $M$ be a topological space, $Z:M\to\Ga_\bC\dual$ be a continuous map and $x\in M$.
Let $\eps>0$ and $V\sbs\bC$ be a strict cone such that
$$Z_x(\Ga)\cap\dd V=\set0,\qquad 
\n{Z_x(a)}\ne\eps\nn a\quad
\forall 0\ne a\in\Ga.$$
Then, for any finite lower set $I\sbs S(V,Z_x,\eps)$,
there exists an open set $x\in U\sbs M$ such that
$I\sbs S(V,Z_y,\eps)$ is a lower set for all $y\in U$.
\end{lemma}
\begin{proof}
By the previous result, we can assume that $I\sbs S(V,Z_y,\eps)$ for all $y\in U$.
Using the same argument as before, we can assume that
$$S(V,Z_y,\eps)\sbs S(V,Z_x,\toh\eps)\qquad \forall y\in U$$
and $c>0$ is such that $\sum \nn {a_i}\le c\nn{\sum a_i}$ for all $a_i\in S(V,Z_x,\toh\eps)$.
Let us consider the finite set
$$D=\bigcup_{\ga\in I}\sets{a\in\Ga}{\nn a\le c\nn\ga}.$$
By the previous result, we can assume that
$D\cap S(V,Z_y,\eps)=D\cap S(V,Z_x,\eps)$ for all $y\in U$.

We need to show that for all $y\in U$ the subset $I\sbs S(V,Z_y,\eps)$ is a lower set, meaning that if $\ga=a_1+a_2 \in I$ and $a_i\in S(V,Z_y,\eps)$,
then $a_i\in I$.
We have $a_i\in S(V,Z_y,\eps)\sbs S(V,Z_x,\toh\eps)$, hence $\nn {a_i}\le c\nn\ga$ and 
$a_i\in D\cap S(V,Z_y,\eps)=D\cap S(V,Z_x,\eps)$.
As $I\sbs S(V,Z_x,\eps)$ is a lower set, we conclude that $a_i\in I$.
\end{proof}

\begin{remark}
We don't know if a similar statement is true if one uses semigroups $C(V,Z_x,\eps)\cap\Ga$ instead of $S(V,Z_x,\eps)$.
\end{remark}

\subsection{Cones and quadratic forms}
\label{secCQ}
\begin{lemma}\label{convex lQ}
[\cf~{\cite{kontsevich_stability,bayer_space}}]
Let $Z:E\to F$ be a linear map between finite-dimensional real vector spaces, $Q:E\to\bR$ be a quadratic form, negative semi-definite on $\Ker Z$, and $\ell\sbs F$ be a ray.
Then the cone
\begin{equation}
C(\ell,Z,Q)=\sets{x\in E}{Z(x)\in\ell,\,Q(x)\ge0}
\end{equation}
is convex and $Q(x+y)\ge Q(x)+Q(y)$ for all $x,y\in C(\ell,Z,Q)$.
This cone is strict if $Q$ is negative-definite on $\Ker Z$.
\end{lemma}
\begin{proof}
%We can assume that $F=\bR$ and $\ell=\bR_{>0}$.
It is enough to show that $Q(x+y)\ge Q(x)+Q(y)$
for all $x,y\in C(\ell,Z,Q)$.
There exists $\la>0$ such that $y-\la x\in\Ker Z$.
%If $x$ and $y$ are proportional, then it is clear that $x+y\in C$.
%Otherwise we can assume that $E$ is generated by $x,y$ and $\Ker Z=\bR(y-\la x)$ for some $\la>0$.
Let $q$ be the symmetric bilinear form corresponding to $Q$.
Then 
$2\la q(x,y)=\la^2Q(x)+Q(y)-Q(y-\la x)\ge0$.
Therefore $q(x,y)\ge0$, hence $Q(x+y)\ge Q(x)+Q(y)$ as required.
The last statement can be proved directly or by applying Lemma \ref{strict VQ}.
%To prove the last statement, we note that if $Z(x)\in\ell\cup\set0$ and $Q(x)\ge0$, then either $x\in C(\ell,Z,Q)$ or $x=0$.
%Therefore the cone $C(\ell,Z,Q)\cup\set0$ is closed and it is clear that it is strict.
\end{proof}

\begin{corollary}
Let $Z:E\to F$ be a linear map between finite-dimensional real vector spaces, $Q:E\to\bR$ be a quadratic form, negative semi-definite on $\Ker Z$, and $\ell\sbs F$ be a ray.
Then the cone
\begin{equation}
C^+(\ell,Z,Q)=\sets{x\in E}{Z(x)\in\ell,\,Q(x)>0}
\end{equation}
is convex.
\end{corollary}

\begin{lemma}\label{strict VQ}
%[See {\cite{kontsevich_stability,bayer_space}}]
Let $Z:E\to F$ be a linear map between finite dimensional real vector spaces, $Q:E\to\bR$ be a quadratic form, negative-definite on $\Ker Z$, and $V\sbs F$ be a strict cone.
Then the convex cone
\begin{equation}
C(V,Z,Q)=\cone\sets{x\in E}{Z(x)\in V,\,Q(x)\ge0}
\end{equation}
is strict.
\end{lemma}
\begin{proof}
Let $K=\Ker Z$ and $K^\perp\sbs E$ be its orthogonal complement \wrt $Q$.
Then $E=K\oplus K^\perp$ and $Z:K^\perp\emb F$ is injective.
Let $F$ be equipped with an inner product and the induced norm.
There exists $\eps>0$ such that $\eps Q(y)<\nn{Z(y)}^2$ for all $0\ne y\in K^\perp$. 
We define the norm $\nn\cdot$ on $E$ by 
%(\cf Lemma \ref{norm-Q})
$$\nn{x}^2=\nn{Z(x)}^2-\eps Q(x).$$
It is indeed a norm: 
for any $x\in K$ and $y\in K^\perp$ we have
$$\nn{x+y}^2=\nn{Z(y)}^2-\eps Q(x)-\eps Q(y)\ge -\eps Q(x)\ge0$$
and the equality is only possible for $x=y=0$.
For $x\in E$, we have $Q(x)\ge0$ if and only if $\nn x\le\nn {Z(x)}$.
As $V$ is strict, there exists $u\in F\dual$ such that $V\sbs C_u=\sets{y\in F}{u(y)\ge\nn y}$.
If $Z(x)\in V$ and $Q(x)\ge0$, then
$uZ(x)\ge\nn{Z(x)}\ge\nn x$.
Therefore $C(V,Z,Q)$ is contained in the strict cone
$\sets{x\in E}{uZ(x)\ge\nn x}$.
\end{proof}

\def\ffs{S}
Given a linear map $Z:E\to F$, we will say that a subset $\ffs\sbs E$ satisfies the \idef{support property} (\wrt $Z$) if there exist norms on $E$ and $F$ such that
\begin{equation}\label{supp1}
\ffs\sbs\sets{x\in E}{\nn {Z(x)}\ge\nn x}.
\end{equation}

\begin{remark}
The support property is equivalent to the condition that 
the following function on $\Hom(E,F)$ defines a norm (\cf Lemma \ref{lm:supp norm})
$$\nn{u}_{Z,S}:=\sup_{x\in S\ms\set0}\frac{\nn{u(x)}}{\nn{Z(x)}},\qquad u\in\Hom(E,F).$$

\end{remark}

\begin{lemma}\label{norm-Q}
A subset $\ffs\sbs E$ satisfies the support property if and only if there exists a quadratic form $Q$ on $E$ such that
\begin{enumerate}
\item $Q$ is negative definite on $\Ker Z$.
\item For any $x\in \ffs$, we have $Q(x)\ge0$.
\end{enumerate}
\end{lemma}
\begin{proof}
Let us assume that $\ffs$ satisfies the support property.
We can assume that the norms on $E$ and $F$ are induced by inner products.
Then we define the quadratic form
$$Q(x)=\nn{Z(x)}^2-\nn x^2$$
which satisfies the required properties.

Conversely, let us assume that $Q$ satisfies the required properties.
Let $K=\Ker Z$ and $K^\perp\sbs E$ be its orthogonal complement \wrt $Q$. 
Then $E=K\oplus K^\perp$ and 
$Z:K^\perp\emb F$ is injective.
Let $F$ be equipped with an inner product and the induced norm.
There exists $\eps>0$ such that $\eps Q(y)<\nn{Z(y)}^2$ for all $0\ne y\in K^\perp$. 
We define the norm on $E$ by
(\cf Lemma \ref{strict VQ})
$$\nn{x}^2=\nn{Z(x)}^2-\eps Q(x).$$
For $x\in \ffs$, we have $Q(x)\ge0$, hence 
$\nn {Z(x)}\ge \nn x$.
\end{proof}

\subsection{Families of stable supports}
\label{sec:supp famil}
Let $\Ga$ be a free abelian group of finite rank and $\nn\cdot$ be a norm on $\Ga_\bR$.
We define a \idef{stable support}
to be a triple $(Z,\fr,\fs)$,
where $Z:\Ga\to\bC$ is a linear map
and $\fr\sbs\fs\sbs\Ga\ms\set0$ are subsets
satisfying for some $\eps>0$ (\cf \eqref{supp1})
\begin{equation}
\fs\sbs\sets{\ga\in\Ga}{\n{Z(\ga)}\ge\eps\nn\ga}
\end{equation}
and such that for every ray $\ell\sbs\bC$, we have
\begin{equation}
\fs\cap Z\inv(\ell)\sbs \sgr(\fr\cap Z\inv(\ell)).
%S(\ell,Z,\eps).
\end{equation}

% strict semigroup
%\begin{equation}
%\hat\fs_\ell=\sgr(\fs\cap Z\inv(\ell))\sbs S(\ell,Z,\eps)
%\end{equation}
%and we define 
%\begin{equation}
%\hat\fs=\bigcup_{\ell\sbs\bC}\hat\fs_\ell
%\sbs\sets{\ga\in\Ga}{\n{Z(\ga)}\ge\eps\nn\ga}
%\end{equation}

%and its set of generators
%\begin{equation}
%\fm_\ell=\min\fs_\ell\sbs\fs\cap Z\inv(\ell).
%\end{equation}
%We define $\fm=\bigcup_{\ell}\fm_\ell\sbs\fs$.
%In what follows, we will assume for simplicity that $\fs\cap Z\inv(\ell)$ is already a semigroup, but all arguments work without this assumption.
Intuitively, one can think about $\fr$ as the set of classes of stable objects and about $\fs$
as the set of classes of semistable objects.

Let $M$ be a topological space.
% and $Z:M\to\Ga_\bC\dual$ be a continuous map.
We define a \idef{continuous family of stable supports} on $M$
to be a collection of stable supports $(Z_x,\fr_x,\fs_x)_{x\in M}$ such that $Z:M\to\Ga_\bC\dual$ is continuous and
\begin{enumerate}
\item For any $x\in M$, there exists $\eps>0$ and an open set $x\in U\sbs M$ such that
\begin{equation*}
\fs_y\sbs\sets{\ga\in\Ga}{\n{Z_y(\ga)}\ge\eps\nn\ga}\qquad
\forall y\in U.
\end{equation*}
\item
For any $\ga\in\Ga$, the set $\sets{x\in M}{\ga\in\fr_x}$ is open.
\item
For any $\ga\in\Ga$, the set $\sets{x\in M}{\ga\in\fs_x}$ is closed.
\end{enumerate}

\begin{theorem}\label{th:Qpos1}
Let $M$ be path-connected and $(Z,\fr,\fs)$ be a continuous family of stable supports on ~$M$.
Let $Q:\Ga_\bR\to\bR$ be a quadratic form that is negative semi-definite on $\Ker Z_x$ for all $x\in M$
and such that there exists $x\in M$ satisfying $Q(\ga)\ge0$ for all $\ga\in\fs_x$.
Then we have $Q(\ga)\ge0$ for all $\ga\in\fs_y$ and $y\in M$.
\end{theorem}
\begin{proof}
By Lemma \ref{convex lQ} it is enough to show that $Q(\ga)\ge0$ for all $\ga\in\fr_y$ and $y\in M$.
We can assume that $M=[0,1]$ and that $Q(\ga)\ge0$ for all $\ga\in\fs_1$.
As $M$ is compact, there exists $\eps>0$ such that 
$$\fs_t\sbs \sets{\ga\in\Ga}{\n{Z_t(\ga)}\ge\eps\nn\ga}$$
for all $t\in M$.
Let $0<\theta<\pi/4$.
By dividing $[0,1]$ into a finite union of intervals,
we can assume that 
$\nn{Z_t-Z_1}\le \frac\eps2\sin(\theta)$
for all $t\in[0,1]$.
For any $\ga\in\fs_t$, we have 
$\n{Z_t(\ga)-Z_1(\ga)}\le\frac\eps2\nn\ga\le\toh\n{Z_t(\ga)}$, hence 
$$\n{Z_1(\ga)}\ge\toh \n{Z_t(\ga)}\ge\tfrac\eps2\nn\ga.$$
On the other hand, if $\ga\in\Ga\ms\set0$ satisfies
$\n{Z_1(\ga)}\ge\tfrac\eps2\nn\ga$, then
$$\n{Z_t(\ga)-Z_1(\ga)}\le \tfrac\eps2\sin(\theta)\nn{\ga}\le\sin(\theta)\n{Z_1(\ga)},$$
hence the angle between $Z_t(\ga)$ and $Z_1(\ga)$ is $\le\theta$.
If $\ga=\ga_1+\ga_2$ with $Z_t(\ga_i)$ contained in the same ray and $\n{Z_1(\ga_i)}\ge\tfrac\eps2\nn{\ga_i}$, 
then the angle between $Z_t(\ga_i)$ and $Z_1(\ga_i)$ is $\le\theta$,
hence the angle between $Z_1(\ga_1)$ and $Z_1(\ga_2)$ is $\le 2\theta<\pi/2$.
This implies that
$$\tfrac\eps2\nn{\ga_i}\le\n{Z_1(\ga_i)}<\n{Z_1(\ga)},\qquad i=1,2.$$
For a fixed $\ga\in\fr_0$, we can assume by induction that the result is true for all elements of the finite set 
$$D=\sets{\ga_1 \in\bigcup_t\fs_t}
{\tfrac\eps2\nn{\ga_1}\le\n{Z_1(\ga_1)}<\n{Z_1(\ga)}}.$$
Let us choose maximal $t>0$ such that $\ga\in\fr_{t'}$ for all $t'\in[0,t)$.
Then $\ga\in\fs_{t}$.
If $t=1$, then we are done.
Otherwise, we have $\ga\in\fs_{t}\ms\fr_{t}$, hence we can decompose $\ga=\sum_{i=1}^n\ga_i$, where $\ga_i\in\fr_t\cap Z\inv(\ell)$, $\ell=\bR_{>0}Z_t(\ga)$ and $n\ge2$.
By the above discussion, we have $\ga_i\in D$, hence $Q(\ga_i)\ge0$ by inductive assumption.
By Lemma \ref{convex lQ} we conclude that $Q(\ga)\ge0$.
%By the second axiom (and compactness of $M$), there exists a finite sequence $0=t_0<\dots<t_n=1$ such that
%$a_{t_i,\ga}$ can be expressed
%as a finite Lie expression of $a_{t,\ga_1}$
%(for any $t\in[t_i,t_{i+1}]$)
%with $Z_{t_i}(\ga_1)$ contained in the same ray as $Z_{t_i}(\ga)$.
%By the discussion above we conclude that $\ga_1\in D\cup\set{\ga}$.
%Therefore by the inductive assumption we only need to express $a_{t_{i+1},\ga}$ in terms of $a_{1,\ga'}$.
%Now we repeat the previous step.
\end{proof}

\begin{corollary}
Let $M$ be locally path-connected and $(Z,\fr,\fs)$ be a continuous family of stable supports on ~$M$.
Let $Q:\Ga_\bR\to\bR$ be a quadratic form and $x\in M$ be such that $Q$ is negative-definite on $\Ker Z_x$ 
and $Q(\ga)\ge0$ for all $\ga\in\fs_x$.
Then we have $Q(\ga)\ge0$ for all $\ga\in\fs_y$ and $y$ in some \nbd of $x$.
\end{corollary}
\begin{proof}
There exists an open subset $x\in U\sbs M$ such that $Q_x$ is negative-definite on $\Ker Z_y$ for all $y\in U$.
We can assume that $U$ is path-connected.
Now we apply the previous result.
\end{proof}

%Let $M=[0,1]$, $Z:M\to\Ga_\bC\dual$ be a continuous map and $(\fs_t)_{t\in M}$ be a continuous family of supports.
%Starting with $\ga_0\in\fs_0$, we want to study its decompositions as we increase $t$.

\section{Stability conditions}
\label{sec:stab cond}
Given a triangulated category $\cD$ and a family of subcategories (or objects) $(\cA_i)_{i\in I}$ in $\cD$, we denote by $\ang{\cA_i\col i\in I}$ the minimal full subcategory of $\cD$ that contains all $\cA_i$ and is closed under extensions.
For two subcategories $\cA_1,\cA_2\sbs\cD$, we will write $\Hom(\cA_1,\cA_2)=0$ if $\Hom(X,Y)=0$ for all $X\in\cA_1$, $Y\in\cA_2$.
Given two objects $X,Y\in\cD$ and $n\in\bZ$, we will sometimes denote $\Hom(X,Y[n])$ by $\Hom^n(X,Y)$.
For any morphism $f:X\to Y$ in \cD, we will write $\cone(f)$ for the third object (unique up to an isomorphism) of the corresponding distinguished triangle.
Sometimes we will use the word triangle for a distinguished triangle.

\subsection{t-structures}
%\begin{comm}
%This reflects the fact that a triangulated category has an enhancement over the category of \bZ-graded vector spaces.
%\end{comm}

A \idef{t-structure} on a triangulated category \cD is a pair $(\cD^{\le0},\cD^{\ge0})$ of full subcategories in \cD such that, using 
$\cD^{\le n}=\cD^{\le0}[-n]$ and 
$\cD^{\ge n}=\cD^{\ge0}[-n]$ for $n\in\bZ$, we have
\begin{enumerate}
\item $\Hom(\cD^{\le0},\cD^{\ge1})=0$.
\item $\cD^{\le0}\sbs\cD^{\le1}$ and
$\cD^{\ge0}\sps\cD^{\ge1}$.
\item
For any $X\in\cD$, there exists a distinguished triangle $X'\to X\to X''\to$ with $X'\in\cD^{\le0}$ and $X''\in\cD^{\ge1}$.
\end{enumerate}
We have \cite{BBD}
\begin{gather}
\cD^{\ge1}=(\cD^{\le0})^\perp
:=\sets{Y\in\cD}{\Hom(\cD^{\le0},Y)=0},\\
\cD^{\le0}={}^\perp(\cD^{\ge1})
:=\sets{X\in\cD}{\Hom(X,\cD^{\ge1})=0},
\end{gather}
hence a t-structure is uniquely determined by the subcategory $\cD^{\le0}$.
The category $\cA=\cD^{\le0}\cap\cD^{\ge0}$ is abelian
and is called the \idef{heart} of the t-structure \cite{BBD}.
A t-structure is called \idef{bounded} if every object in $\cD$ is contained in $\cD^{\le n}\cap \cD^{\ge-n}$ for some $n\ge0$.
In this case the t-structure is uniquely determined by its heart, namely,
\begin{equation}
\cD^{\le 0}=\angs{\cA[n]}{n\ge0},\qquad 
\cD^{\ge 0}=\angs{\cA[n]}{n\le0}.
\end{equation}

\subsection{Slicings}
We define a slicing \cP of a triangulated category \cD to be 
%an $\bR$-slicing such that $\cP_{\vi+1}=\cP_\vi[1]$ for all $\vi\in\bR$.
%This means that a slicing \cP is 
a collection $(\cP_\vi)_{\vi\in\bR}$ of full subcategories in \cD such that
\begin{enumerate}
\item $\cP_{\vi+1}=\cP_{\vi}[1]$ for all $\vi\in\bR$.
\item $\Hom(\cP_{\vi},\cP_{\vi'})=0$ for all $\vi>\vi'$.
\item
For any $0\ne E\in\cD$, there exists a sequence of maps
$$0=E_0\to E_1\to\dots\to E_n=E$$
such that $\cone(E_{k-1}\to E_k)\in\cP_{\vi_k}$ for some $\vi_k\in\bR$ satisfying $\vi_1>\dots>\vi_n$.
% such
%that the cone of $f_i$ is contained in $\cP_{\vi_i}$ for all $i$.
\end{enumerate}
We will call the above sequence the Harder-Narasimhan (HN) filtration of $E$ and we will address axiom (3) as the HN property.
The HN filtration of $E$ is uniquely determined up to isomorphism.
We define $\vi^+_\cP(E)=\vi_1$ and $\vi^-_\cP(E)=\vi_n$.
For any interval $I\sbs\bR$, let 
\begin{equation}
\cP_I=\angs{\cP_\vi}{\vi\in I}\sbs\cD.
\end{equation}
Note that traditionally one denotes $\cP_\vi$ by $\cP(\vi)$ and $\cP_I$ by $\cP(I)$.
The categories
\begin{equation}
\cD^{\le0}=\cP_{>0}=\cP_{(0,+\infty)},\qquad \cD^{\ge0}=\cP_{\le1}=\cP_{(-\infty,1]}
\end{equation}
form a bounded t-structure with the heart $\cA=\cP_{(0,1]}$, called the \idef{heart} of $\cP$ \bris.
The categories $\cP_\vi$ are also abelian \bris[Lemma 5.2].
For any $a\in\bR$, we define a new slicing ~$\cP[a]$
\begin{equation}\label{shift slicing}
\cP[a]_\vi=\cP_{\vi+a},\qquad \vi\in\bR.
\end{equation}
Its heart is equal to $\cP_{(a,a+1]}$.
Note that $\cP[n]_\vi=\cP_{\vi+n}=\cP_\vi[n]$ for $n\in\bZ$.

\begin{example}
Given a full subcategory $\cA\sbs\cD$,
let us define $\cP_n=\cA[n-1]$ for $n\in\bZ$ and $\cP_n=0$ for $n\in\bR\ms\bZ$.
Then $\cA$ is the heart of a bounded $t$-structure if and only if $\cP=(\cP_n)_{n\in\bR}$ is a slicing (\cf \bris[Lemma 3.2]).
The heart of $\cP$ is equal to $\cP_1=\cA$.
\end{example}

\subsection{Bridgeland stability conditions}
Let \cD be a triangulated category,
$\Ga$ be a free abelian group of finite rank and $\cl:K(\cD)\to\Ga$ be a linear map.
We define a \idef{pre-stability condition} on the triangulated category \cD (with respect to $\cl$) to be a pair $\si=(Z,\cP)$, where
\begin{enumerate}
\item $Z:\Ga\to\bC$ is a linear map, called a \idef{central charge}.
\item $\cP$ is a slicing such that for any $0\ne E\in\cP_\vi$, we have $Z(E):=Z(\cl E)\in\bR_{>0}e^{\ip\vi}$.
\end{enumerate}

Nonzero objects of $\cP_\vi$ are called \idef{\si-semistable} objects of phase $\vi$.
Simple objects of $\cP_\vi$ are called \idef{\si-stable}
objects of phase $\vi$.
The heart of $\cP$ is also called the heart of the stability condition ~$\si$.
For any $0\ne E\in\cD$, we will denote $\phi_\cP^\pm(E)$ by $\phi^\pm_\si(E)$.
Assuming that $\phi_\si^+(E)-\phi_\si^-(E)<1$, we have $Z(E)\ne0$ and we define the phase $\vi_\si(E)$ to be the unique value such that
\begin{equation}\label{phase}
\phi_\si(E)\in[\phi_\si^-(E),\phi_\si^+(E)],\qquad
Z(E)\in\bR_{>0}e^{\ip\vi_\si(E)}.
\end{equation}
We define the support of $\si$ to be
\begin{equation}
\supp(\si)=\sets{\cl E}{E\text{ is \si-semistable}}\sbs\Ga.
\end{equation}

Let us equip $\Ga_\bR=\Ga\ts_\bZ\bR$ with a norm $\nn\cdot$.
A pre-stability condition $\si=(Z,\cP)$ on $\cD$ is called a \idef{stability condition} if there exists $\eps>0$ such that 
%for any $0\ne E\in\cP_\vi$, 
we have (the support property)
\begin{equation}
\n{Z(E)}\ge \eps\nn{\cl E}
\end{equation}
for all $\si$-semistable objects $E$.
It is enough to verify this property for $E\in\cP_\vi$ with $\vi\in(0,1]$.

%\begin{remark}
%Consider the set of $x\in\Ga_\bR$ such that $Z(x)\in\ell$ and $\n {Z(x)}\ge\eps\nn x$.
%This is a convex cone.
%Indeed, if $x,y$ are two such elements, then
%$\n {Z(x+y)}=\n{Z(x)}+\n{Z(y)}\ge\eps\nn x+\eps\nn y\ge\eps\nn {x+y}$.
%\end{remark}

\begin{lemma}
A pre-stability condition $\si=(Z,\cP)$ satisfies the support property if and only if there exists a quadratic form $Q$ on $\Ga_\bR$ such that
\begin{enumerate}
\item $Q$ is negative definite on $\Ker (Z:\Ga_\bR\to\bC)$.
\item For any $\si$-semistable object $E$ we have $Q(\cl E)\ge0$.
\end{enumerate}
\end{lemma}
\begin{proof}
See Lemma \ref{norm-Q}.
\end{proof}
%\begin{proof}
%Let $S=\sets{\cl E}{0\ne E\in\cP_\vi,\, \vi\in\bR}\sbs\Ga$.
%Let us assume that $\si$ satisfies the support property and that $\nn\cdot$ is induced by a positive definite symmetric form.
%% $(\cdot,\cdot)$.
%There exists $C>0$ such that $\nn{x}\le C\n{Z(x)}$ for all $x\in S$ and we define
%$$Q(x)=C^2\n{Z(x)}^2-\nn x^2.$$
%If $0\ne x\in\Ker Z$, then $Q(x)<0$.
%If $x\in S$, then 
%%$\n{Z(x)}^2\ge\eps^2\nn x^2=\eps^2(x,x)$, hence 
%$Q(x)\ge0$.
%Conversely, let us assume that $Q$ satisfies the above conditions.
%It is enough enough to construct some norm that satisfies the support property.
%Let $K=\Ker Z$ and $K^\perp$ be its orthogonal space \wrt $Q$. Then $\Ga_\bR=K\oplus K^\perp$ and $Z$ induces an injective map $Z:K^\perp\to\bC$.
%Considering the induced norm on ~$K^\perp$, we can find $C>0$ such that $Q(y)<C^2\n{Z(y)}^2$ for all $0\ne y\in K^\perp$. 
%We define the norm $\nn\cdot$ on $\Ga_\bR$ as
%$$\nn{x}^2=C^2\n{Z(x)}^2-Q(x).$$
%If $x\in K$ and $y\in K^\perp$, then 
%$\nn{x+y}^2=C^2\n{Z(y)}^2-Q(x)-Q(y)\ge -Q(x)\ge0$ and the equality is only possible for $x=y=0$.
%For $x\in S$, we have $Q(x)\ge0$, hence $\nn x\le C\n {Z(x)}$.
%\end{proof}

\begin{remark}
Note that if $\rk\Ga=n$, then a quadratic form $Q$ as above has negative index $\ge n-2$.
If $Q$ has signature $(2,n-2)$, then $\dim(\Ker Z)=n-2$ and $Z:\Ga_\bR\to\bC$ is automatically surjective.
\end{remark}

\begin{lemma}\label{lm:supp norm}
A pre-stability condition $\si=(Z,\cP)$ satisfies the support property if and only if the following map
defines a norm on $\Ga_\bC\dual=\Hom(\Ga,\bC)$
\begin{equation}
\nn\cdot_\si:\Ga_\bC\dual\to[0,+\infty],\qquad
u\mto \sup\sets{\frac{\n{u(E)}}{\n{Z(E)}}}
{E\text{ is \si-semistable}}
\end{equation}
\end{lemma}
\begin{proof}
The above map satisfies all properties of a norm, except that it may be not finite.
Let $S=\sets{\cl E}{0\ne E\in\cP_\vi,\,\vi\in\bR}\sbs\Ga$.
If \si satisfies the support property, then there exists $\eps>0$ such that $\n{Z(x)}\ge\eps\nn x$ for all $x\in S$.
A linear map $u:\Ga_\bR\to\bC$ is bounded, meaning that there exists $C>0$ such that $\n{u(x)}\le C\nn x$ for all $x\in\Ga_\bR$.
This implies that $\n {u(x)}\le \frac C\eps\n{Z(x)}$ for all $x\in S$, hence $\nn u_\si\le\frac C\eps$.

Conversely, we can assume that the norm on $\Ga_\bR$ is given by $\nn x=\sum_i\n{u_i(x)}$
for some basis $(u_1,\dots,u_n)$ of $\Ga\dual_\bR$. For any $x\in S$, we have
$$\nn x=\sum_i\n{u_i(x)}\le\sum_i \nn{u_i}_\si\cdot\n{Z(x)}\le 
\max_i\nn{u_i}_\si\cdot \n {Z(x)}.$$
This proves the support property.
\end{proof}

\begin{remark}
In our applications the group $\Ga$ will be the numerical Grothendieck group of $\cD$.
More precisely, let us assume that $\bop_{n\in\bZ}\Hom^n(E,F)$ is finite-dimensional for all $E,F\in\cD$ and let $\hi(E,F)=\sum_{n\in\bZ}(-1)^n\dim\Hom^n(E,F)$ be the \idef{Euler form} of $\cD$ defined on the Grothendieck group $K(\cD)$.
Assuming that $\cD$ has a Serre functor, meaning an automorphism $\cS:\cD\to\cD$
such that there are natural isomorphisms $\Hom(E,F)\iso\Hom(F,\cS E)\dual$ for all $E,\,F\in\cD$,
we conclude that the right and the left kernels of $\hi$ are equal and we define the \idef{numerical Grothendieck group} $\cN(\cD)$ to be the quotient of $K(\cD)$ by this kernel.
We define the linear map $\cl:K(\cD)\to\Ga=\cN(\cD)$ to be the projection.
\end{remark}

\subsection{Stability functions on abelian categories}
\label{sec:stab ab}
%Let $\cl:K(\cD)\to\Ga$ be a linear map as before.
Let $\cA\sbs\cD$ be the heart of a bounded t-structure so that $K(\cA)\iso K(\cD)$.
We define a \idef{stability function} on $\cA$ 
(\wrt $\cl:K(\cA)\to\Ga$) to be a linear map $Z:
\Ga\to\bC$ such that for any $0\ne E\in\cA$, the number $Z(E):=Z(\cl E)$ is contained in the (semi-closed) upper-half plane
\begin{equation}
\bH=\sets{r e^{\ip\vi}}{r>0,\, 0<\vi\le1}\sbs\bC.
\end{equation}
Let us consider the argument map
\begin{equation}
\Arg:\bC^*\to(-\pi,\pi],\qquad re^{\ip\vi}\mto\pi\vi,\qquad
r>0,\,\vi\in(-1,1],
\end{equation}
and define the \idef{phase} of $0\ne E\in\cA$ to be 
\begin{equation}
\vi(E)=\frac1\pi\Arg Z(E)\in(0,1].
\end{equation}
An object $0\ne E\in\cA$ is called $Z$-\idef{semistable} if for any subobject $0\ne E'\sbs E$ we have $\vi(E')\le\vi(E)$.
Equivalently, this means that
\begin{equation}
\Im(Z(E)\cdot \bar Z(E'))\ge0.
\end{equation}
A stability function $Z:\Ga\to\bC$ on \cA is said to satisfy the \idef{HN property}
if, for any $0\ne E\in\cA$, there exists a (unique) filtration, called the \idef{HN filtration} of $E$,
$$0=E_0\sbs E_1\sbs\dots\sbs E_n=E$$
such that $E_k/E_{k-1}$ are $Z$-semistable and have strictly decreasing phases.

It is proved in \bris[\S5] that to give a pre-stability condition $\si=(Z,\cP)$ on the triangulated category \cD is equivalent to giving a pair $(Z,\cA)$, where 
$\cA\sbs\cD$ is the heart of a bounded t-structure
and $Z:\Ga\to\bC$ is a stability function on \cA satisfying the HN property.
The category ~$\cA$ is defined to be the heart $\cP_{(0,1]}$ of the slicing $\cP$.
Conversely, one defines $\cP_\vi$ for $\vi\in(0,1]$ to be the category of all semistable objects of $\cA$ having phase $\vi$ (plus the zero object).
Then one defines $\cP_{\vi+n}=\cP_\vi[n]$ for $n\in\bZ$.
In what follows we will sometimes write $(Z,\cA)$ for the pre-stability condition $\si$.
% and we will call $Z$ a central charge and a stability condition interchangeably. 

\subsection{Families of stability conditions}
Let $\Slice(\cD)$ be the set of all slicings of \cD
and $\Stab_{\cl}(\cD)$ be the set of all stability conditions on \cD (satisfying the support property).
Given two slicings $\cP,\cQ$, we define the distance between them to be
\begin{equation}\label{d1}
d(\cP,\cQ)=\sup_{0\ne E\in\cD}
\set{\n{\vi_\cP^+(E)-\vi_\cQ^+(E)},
\n{\vi_\cP^-(E)-\vi_\cQ^-(E)}}\in[0,+\infty].
\end{equation}
Equivalently \bris[Lemma 6.1],
\begin{equation}\label{d2}
d(\cP,\cQ)=\inf\sets{\eps\ge0}
{\cQ_\vi\sbs\cP_{[\vi-\eps,\vi+\eps]}\ \forall\vi\in\bR}.
\end{equation}
This distance function induces a topology on $\Slice(\cD)$.
We equip $\Stab_{\cl}(\cD)$ with the induced topology using the inclusion
\begin{equation}
\Stab_{\cl}(\cD)\sbs\Hom(\Ga,\bC)\xx\Slice(\cD).
\end{equation}
It is proved in \bris that the forgetful map $\Stab_{\cl}(\cD)\to\Hom(\Ga,\bC)$ is a local homeomorphism
and one can equip $\Stab_{\cl}(\cD)$ with a structure of a complex manifold. 

We define a \idef{continuous family of stability conditions} to be a continuous map $\si:M\to\Stab_{\cl}(\cD)$, where $M$ is a topological space.
Equivalently, a family of stability conditions $(\si_x=(Z_x,\cP_x))_{x\in M}$ is continuous if
\begin{enumerate}
\item The map $Z:M\to\Hom(\Ga,\bC)$, $x\mto Z_x$, is continuous.
\item For any $x\in M$ and $\eps>0$, there exists a \nbd $x\in U\sbs M$ such that $d(\cP_x,\cP_y)<\eps$ for all $y\in U$.
\end{enumerate}

We can similarly define a continuous family of pre-stability conditions.
In what follows, we will denote $\vi^\pm_{\si_x}$ by $\vi^\pm_x$.
The support property is an open and closed condition
by the following result.

\begin{lemma}[{\bris[Lemma 6.2]}]
Let $\si=(Z,\cP)$ and $\si'=(Z',\cP')$ be pre-stability conditions such that for some $0<\eta<\tfrac14$ we have
$$\nn{Z-Z'}_\si<\sin(\pi\eta),\qquad d(\cP,\cP')<\eta.$$
Then there exist $c_1,c_2>0$ such that
$$c_1\nn{u}_\si\le\nn{u}_{\si'}\le c_2\nn{u}_\si
\qquad\forall u\in\Hom(\Ga,\bC).$$
\end{lemma}

%In the next result we will show that the constant $c_2$ can be chosen uniformly. 

\begin{lemma}\label{uniform support}
Let $(\si_x)_{x\in M}$ be a continuous family of pre-stability conditions.
If $x\in M$ is such that $\si_x$ satisfies the support property, then there exists a \nbd $x\in U\sbs M$ and $\eps>0$ such that
\begin{equation}
\n{Z_y(E)}\ge\eps\nn{\cl E}
\end{equation}
for all $y\in U$ and $\si_y$-semistable objects $E$.
\end{lemma}

%\begin{proof}[Alternative proof]
%Let $\si=(Z,\cP)=\si_x$ satisfy the support property.
%It is enough to show that for some \nbd $x\in U\sbs M$ and $C>0$, we have $\nn{u}_{\si_y}\le C\nn u_{\si}$ for all $u\in\Ga_\bC\dual$ and $y\in U$.
%Indeed, we can assume that the norm on $\Ga_\bR$ is given by $\nn \ga=\sum_i\n{u_i(\ga)}$, for a fixed basis $(u_1,\dots,u_n)$ of $\Ga\dual$.
%Then, for any $y\in U$, $\si_y$-semistable object $E$ and $\ga=\cl E$, we have
%$$\nn {\ga}=\sum_i\n{u_i(\ga)}\le\sum_i \nn{u_i}_{\si_y}\cdot\n{Z_y(\ga)}\le 
%C\max_i\nn{u_i}_{\si}\cdot \n {Z_y(\ga)}$$
%which is equivalent to the statement of the lemma.
%\end{proof}

\begin{proof}
Let $\si=(Z,\cP)=\si_x$ satisfy the support property.
It is enough to show that for some \nbd $x\in U\sbs M$ and $C>0$, we have $\nn{u}_{\si_y}\le C\nn u_{\si}$ for all $u\in\Ga_\bC\dual$ and $y\in U$.
Indeed, we can assume that the norm on $\Ga_\bR$ is given by $\nn \ga=\sum_i\n{u_i(\ga)}$, for a fixed basis $(u_1,\dots,u_n)$ of $\Ga\dual$.
Then, for any $y\in U$, $\si_y$-semistable object $E$ and $\ga=\cl E$, we have
$$\nn {\ga}=\sum_i\n{u_i(\ga)}\le\sum_i \nn{u_i}_{\si_y}\cdot\n{Z_y(\ga)}\le 
C\sum_i\nn{u_i}_{\si}\cdot \n {Z_y(\ga)}$$
which is equivalent to the statement of the lemma.

Let $0<\eta<\toh$.
For any $0\ne E\in\cD$ satisfying $\phi_\si^+(E)-\phi_\si^-(E)\le\eta$ and any $u\in\Ga_\bC\dual$, we have 
%\bris[Lemma 6.2]
\begin{equation}\label{bri6.2}
\n{u(\cl E)}\le\frac{\nn u_\si}{\cos(\pi\eta)}\n{Z(E)}.
\end{equation}
Indeed, let $F_1,\dots,F_n$ be \si-semistable HN factors of $E$.
Then (see Example \ref{ex:norm sum})
$$\n{u(\cl E)}\le\sum_i \n{u(\cl F_i)}
\le\nn u_\si\cdot\sum_i\n{Z(F_i)}
\le\nn u_\si\cdot\frac{\n{Z(E)}}{\cos(\pi\eta)}.$$
%where the last inequality follows from the fact that the vectors $Z(F_i)$ are contained in a sector of angle $<\pi\eta$.
There exists a \nbd $x\in U\sbs M$ such that for all $y\in U$ we have
$$\nn{Z_y-Z}_\si<\toh\cos(\pi\eta),\qquad
d(\cP_y,\cP)<\eta/2.$$
For any $0\ne E\in\cP_{y,\vi}$, we have $E\in\cP_{(\vi-\eta/2,\vi+\eta/2)}$, hence
$\phi_\si^+(E)-\phi_\si^-(E)<\eta$ and \eqref{bri6.2}
$$\n{Z_y(E)-Z(E)}
\le\frac{\nn{Z_y-Z}_\si}{\cos(\pi\eta)}\n{Z(E)}
\le\toh\n{Z(E)}.$$
This implies that
$\n{Z_y(E)}\ge\toh\n{Z(E)}$
and for any $u\in\Ga_\bC\dual$ we have \eqref{bri6.2}
$$\n{u(\cl E)}
<\frac{\nn u_\si}{\cos(\pi\eta)}\n{Z(E)}
\le\frac{\nn u_\si}{\cos(\pi\eta)}2\n{Z_y(E)}
.$$
Therefore $\nn u_{\si_y}\le \frac{2}{\cos(\pi\eta)}\nn u_\si$ as required.
%It follows from \bris[Lemma 6.2] that if $\si=(Z,\cP)$ and $\si'=(Z',\cP')$ are pre-stability conditions such that $\nn{Z'-Z}_\si<\sin(\pi\eps)$ and $d(\cP,\cP')<\eps$ for some $\eps\in(0,\tfrac14)$, then there exist $C_1,C_2>0$ such that 
%$$C_1\nn{U}_\si\le \nn U_{\si'}\le C_2\nn{U}_\si$$
%for all $U\in\Ga_\bC\dual$.
%In particular, if $\si$ satisfies the support property, then $\si'$ also satisfies the support property.
%By assumption, for any $\eps\in(0,\tfrac14)$, there exists an open \nbd $x\in U\sbs M$ such that $\nn{Z_y-Z_x}_{\si_x}<\sin(\pi\eps)$ and $d(\cP_x,\cP_y)<\eps$ for all $y\in\cP$.
%This implies that $\si_y$ satisfies the support property for all $y\in U$.
\end{proof}

The following result is well-known.

\begin{lemma}
Let $(\si_x)_{x\in M}$ be a continuous family of stability conditions 
and $0\ne E\in\cD$. Then
\begin{enumerate}
\item The set $\sets{x\in M}{E\text{ is }\si_x\text{-stable}}$ is open in $M$.
\item The set $\sets{x\in M}{E\text{ is }\si_x\text{-semistable}}$ is closed in $M$.
\end{enumerate}
 \end{lemma}

We will prove its variant.

\begin{lemma}\label{open-closed1}
Let $(\si_x)_{x\in M}$ be a continuous family of stability conditions and let
\begin{equation}
\fr_x=\sets{\cl E}{E\text{ is }\si_x\text{-stable}},\qquad
\fs_x=\sets{\cl E}{E\text{ is }\si_x\text{-semistable}}
\end{equation}
for $x\in M$.
Then, for every $\ga\in\Ga$,
\begin{enumerate}
\item The set $\sets{x\in M}{\ga\in\fr_x}$ is open in $M$.
\item The set $\sets{x\in M}{\ga\in\fs_x}$ is closed in $M$.
\end{enumerate}
 \end{lemma}
\begin{proof}
There exists $\eps>0$ and an open set $x\in U\sbs M$
such that 
$$\fs_y\sbs\sets{\ga\in\Ga}{\n{Z_y(\ga)}\ge\eps\nn \ga}$$ for all $y\in U$.
We can assume that $\nn{Z_y-Z_x}\le\tfrac\eps2$ for all $y\in U$.
Then $\n{Z_y(\ga)-Z_x(\ga)}\le\tfrac\eps2\nn\ga\le\toh\n{Z_y(\ga)}$, hence $\n{Z_x(\ga)}\ge\toh\n{Z_y(\ga)}\ge\tfrac\eps2\nn\ga$ for all $\ga\in\fs_y$.

Let $\ga=\cl E$ for some $\si_x$-stable object $E\in\cP_{x,\vi}$.
For any $0<\eta<1/2$, we can assume that $d(\cP_x,\cP_y)<\eta/2$ for all $y\in U$.
Then $\vi_y^\pm(E)\in (\vi-\eta/2,\vi+\eta/2)$.
This implies that for every $\si_y$-HN-factor $F_i$ of $E$,
we have $\vi_x^{\pm}(F_i)\in I=(\vi-\eta,\vi+\eta)$,
hence $F_i\in\cP_{x,I}$.
All nonzero objects in $\cP_{x,I}$ have classes in the semigroup $S(V,Z_x,\eps)$, where $V=\bR_{>0}e^{\ip I}$ is a strict cone.
There are finitely many such classes that are $\le\ga$.
By decreasing $\eta$ (and shrinking $U$) we can assume that all such classes are contained in $S(\ell,Z_x,\eps)$, where $\ell=\bR_{>0}e^{\ip\vi}$.
This implies that $\phi_x^\pm(F_i)=\vi$.
As $E$ is $\si_x$-stable, the above HN-filtration is trivial and $E$ is $\si_y$-semistable. 
Let $E\in\cP_{y,\vi'}$ and let $0\ne F\sbs E$ be its subobject in $\cP_{y,\vi'}$.
Then $\vi_x^\pm(F)\in(\vi'-\eta/2,\vi'+\eta/2)\sbs I$, hence $F\in\cP_{x,I}$.
As before, we obtain $\vi_x^\pm(F)=\vi$.
As $E$ is $\si_x$-stable, we conclude that $F=E$, hence $E$ is $\si_y$-stable.

To show that $\sets{x\in M}{\ga\in\fs_x}$ is closed, we need to prove that if $x\in M$ is such that for every open $x\in U\sbs M$, there exists $y\in U$ with $\ga\in\fs_y$, then $\ga\in\fs_x$.
Let us assume that $\ga=\cl E$ for some $0\ne E\in\cP_{y,\vi'}$ and $y\in U$.
By the above discussion, we can assume that
$\n{Z_x(\ga)}\ge\tfrac\eps2\nn\ga$.
Let $\vi=\frac1\pi\Arg Z_x(\ga)$.
As before, we assume that $0<\eta<1/2$ and $d(\cP_x,\cP_y)<\eta/2$ for all $y\in U$.
Then $\vi_x^\pm(E)\in I'=(\vi'-\eta/2,\vi'+\eta/2)$.
We can assume that $E$ and $\vi'$ are chosen in such way that $\vi\in I'$.
Then $\vi_x^\pm(E)\in I=(\vi-\eta,\vi+\eta)$, hence $E\in\cP_{x,I}$.
All nonzero objects in $\cP_{x,I}$ have classes in the semigroup $S(V,Z_x,\eps)$, where $V=\bR_{>0}e^{\ip I}$ is a strict cone.
There are finitely many such classes that are $\le\ga$.
By decreasing $\eta$ (and shrinking $U$) we can assume that all such classes are contained in $S(\ell,Z_x,\eps)$, where $\ell=\bR_{>0}e^{\ip\vi}$.
This implies that if we can still find $y\in U$ and a  $\si_y$-semistable object $E$ with $\ga=\cl E$, then $E$ is automatically $\si_x$-semistable.
\end{proof}

\begin{theorem}[\Cf\cite{bayer_space,bayer_short}]
\label{th:Qpos2}
Let $M$ be path-connected, 
$(\si_x)_{x\in M}$ be a continuous family of stability conditions on $M$, and $Q:\Ga_\bR\to\bR$ be a quadratic form negative semi-definite on $\Ker Z_x$ for all $x\in M$.
If there exists $x\in M$ such that $Q(\cl E)\ge0$ for all $\si_x$-semistable objects $E$, then the same is true for all points of $M$.
\end{theorem}
\begin{proof}
By Lemma \ref{uniform support} and Lemma \ref{open-closed1},
we can apply Theorem \ref{th:Qpos1}.
\end{proof}

\begin{corollary}
Let $M$ be locally path-connected and
$(\si_x)_{x\in M}$ be a continuous family of pre-stability conditions on $M$.
Let $Q:\Ga_\bR\to\bR$ be a quadratic form and $x\in M$ be such that $Q$ is negative-definite on $\Ker Z_x$ 
and $Q(\cl E)\ge0$ for all $\si_x$-semistable objects $E$.
Then the same is true in some \nbd of $x$.
\end{corollary}

\subsection{Action on stability conditions}
\label{sec:action}
%Let $\Stab(\cD)$ be the set of stability conditions on $\cD$.
Let $\Aut(\cD)$ be the group of automorphisms of $\cD$,
$\GL_2^+(\bR)$ be the group of elements in $\GL_2(\bR)$ having positive determinant and let $\wtl\GL_2^+(\bR)$ be its universal cover.
There is a left action of the group $\Aut(\cD)$ and the right action of the group $\wtl\GL_2^+(\bR)$ on $\Stab_{\cl}(\cD)$ defined as follows \cite[Lemma 8.2]{bridgeland_stability}. 

For any $T\in\Aut(\cD)$ and stability condition $\si=(Z,\cP)$, we define
\begin{equation}
T\si=(Z',\cP'),\qquad Z'=ZT\inv,\qquad \cP'_\vi=T\cP_\vi.
\end{equation}

%Let $S^1=\sets{z\in\bC}{\n z=1}$.
To describe $\wtl\GL_2^+(\bR)$, let us consider the group homomorphism
\oper{Diff}
\oper{Re}
\begin{equation}
\GL_2^+(\bR)\to\Diff(U(1)),\qquad T\mto \bar T,\qquad
\bar T(z):=T(z)/\n{T(z)},
\end{equation}
where $\Diff(U(1))$ is the group of (orientation-preserving) diffeomorphisms of $U(1)\sbs\bC^*$.
Let $\Diff^*(\bR)$ be the group of (orientation-preserving) diffeomorphisms $f:\bR\to\bR$ that commute with the translation $\vi\mto\vi+2$.
Any $f\in \Diff^*(\bR)$ induces
\begin{equation}
\bar f\in\Diff(U(1)),\qquad
e^{\ip\vi}\mto e^{\ip f(\vi)}.
\end{equation}
Then $\wtl\GL_2^+(\bR)$ can be defined using the following diagram with Cartesian squares
\begin{ctikzcd}
\bR\dar["e^{\ip(-)}"']\rar[hook]&
\bC\dar["e^{\ip(-)}"']\rar[hook]&
\wtl\GL^+_2(\bR)\dar\rar&\Diff^*(\bR)\dar\\
U(1)\rar[hook]&\bC^*\rar[hook]&\GL^+_2(\bR)\rar&\Diff(U(1)).
\end{ctikzcd}
Here $\bC^*\emb\GL_2^+(\bR)$ is given by $a+\bi b\mto\smat{a&-b\\b&a}$ and
\begin{gather}
\bC^*\to\Diff(U(1)),\qquad w\mto[z\mto zw/\n w]\\
\bC\to\Diff^*(\bR),\qquad a+bi\mto[\vi\mto \vi+a].
\end{gather}

%One defines the right action of the group $\wtl\GL_2^+(\bR)$ on the set of stability conditions as follows.
For any stability condition $\si=(Z,\cP)$ and an element $g=(T,f)\in\wtl\GL_2^+(\bR)$
(where $T\in\GL_2^+(\bR)$ and $f\in\Diff^*(\bR)$ are mapped to the same element in $\Diff(U(1))$),
we define a new stability condition
\begin{equation}
\si[g]=(Z[g],\cP[g]),\qquad
Z[g]=T\inv Z,\qquad \cP[g]_\vi=\cP_{f(\vi)}.
\end{equation}
In particular, for $a+\bi b\in\bC$, we define (\cf \eqref{shift slicing})
\begin{equation}\label{stab shift}
Z[a+\bi b]=e^{-\ip a+\pi b}Z,\qquad \cP[a+\bi b]_\vi=\cP_{\vi+a}.
\end{equation}

\subsection{Phase behavior in families}
\label{sec:phase behavior}
Let $(\si_x)_{x\in M}$ be a continuous family of stability conditions.
We denote $\vi^\pm_{\si_x}$ by $\vi^\pm_x$.
For any object $0\ne E\in\cD$ the maps
\begin{equation}
M\to\bR,\ x\mto\vi^-_x(E),\qquad
M\to\bR,\ x\mto\vi^+_x(E),
\end{equation}
are continuous by \eqref{d1}.

\begin{theorem}\label{th:phase behave}
Let $(\si_t=(Z_t,\cP_t))_{t\in [0,1]}$ be a
continuous family of stability conditions on ~\cD such that
the map $Z:[0,1]\to\Ga_\bC\dual$ is differentiable
and
\begin{equation}\label{phase cond}
\Im(Z'_t(E)\cdot \bar Z_t(E))\ge0
\end{equation}
for all $t\in[0,1]$ and $\si_t$-stable objects $E\in\cD$.
%where $Z'_{t_0}(E)=\frac d{dt}Z_t(E)|_{t=t_0}$.
Then, for any object $0\ne E\in\cD$,
the functions 
$$t\mto\vi_t^-(E),\qquad t\mto\vi_t^+(E)$$
are weakly-increasing .
\end{theorem}
\begin{proof}
We will prove just that $\vi_t^-(E)$ is weakly-increasing.
It is enough to show that if $E$ is $\si_{t_0}$-stable and $t\in(t_0,t_0+\eps)$ for $0<\eps\ll1$, then $\vi^-_t(E)\ge\vi_{t_0}(E)$.
We can assume that $E$ is $\si_t$-stable for all $t\in[t_0,t_0+\eps)$.
The condition \eqref{phase cond} means that
$$Z'_t(E)\in Z_{t}(E)\cdot \bar\bH,\qquad \bar\bH=\sets{x+\bi y\in\bC}{y\ge0},$$
as $\bar Z_t(E)$ and $Z_t(E)\inv$ are contained in the same ray.
This condition implies that $Z_t(E)\in Z_{t_0}(E)\cdot\bar\bH$ for $t\in(t_0,t_0+\eps)$ and $0<\eps\ll1$.
Therefore $\vi_t(E)\ge\vi_{t_0}(E)$.
%Let $\cA_t$ denote the heart of the slicing $\cP_t$.
%%As $t$ goes from $0$ to $1$, the object $E$ may become non-semistable and we may find $t_0\in[0,1)$ and $0<\eps\ll1$,
%%such that $E$ is $\si_{t_0}$-semistable, but not $\si_t$-semistable for $t\in(t_0,t_0+\eps)$.
%Consider an exact sequence $0\to E_1\to E\to E_2\to0$ (in $\cA_{t_0}$, assuming without loss of generality that $E$ has phase $\oh$)
%with $\vi_t(E_1)>\vi_t(E)>\vi_t(E_2)$ for $t\in(t_0,t_0+\eps)$.
%Then we have $\vi_{t_0}(E_1)=\vi_{t_0}(E)=\vi_{t_0}(E_2)$
%and $E_1,E_2$ are $\si_{t_0}$-semistable.
%We need to show that
%$$\vi_t(E_2)\ge\vi_{t_0}(E)=\vi_{t_0}(E_2).$$
%This means that the vector $Z_t(E_2)$ is obtained from $Z_{t_0}(E_2)$ by an anti-clockwise rotation (and multiplication by a positive scalar).
%Equivalently,
%$$Z'_t(E_2)\in Z_{t_0}(E_2)\cdot \bar\bH,\qquad \bar\bH=\sets{x+iy\in\bC}{y\ge0}.$$
%Now we note that for $z=Z_{t_0}(E_2)$, the vector $z\inv$ is equal to $\bar z$ up to a positive scalar.
%The proof for $\vi_t^+(E)$ is similar.
\end{proof}

\begin{remark}
Given a stability condition $\si=(Z,\cP)$ on a triangulated category \cD,
we define a stability condition $\bar \si=(\bar Z,\bar\cP)$ on the opposite category $\cD^\op$ as follows.
We define $\bar\cP_\vi=\cP_{-\vi}^\op$ and we define $\bar Z(E)$ to be the conjugate of $Z(E)$.
Note that if $0\ne E\in\bar\cP_\vi$, then $Z(E)\in\bR_{>0}e^{-i\pi\vi}$,
hence
$\bar Z(E)\in\bR_{>0}e^{i\pi\vi}$.
Assuming that we have a family of stability conditions $(\si_t)_t$ on $\cD$ as above, we obtain a family $(\bar\si_t)_t$ of stability conditions on $\cD^\op$. 
Let $\bar\vi_t^{\pm}$ be the phase functions corresponding to these stability conditions.
We have $\Im(\bar Z'_t(E)\cdot Z_t(E))\le0$, hence we obtain from the first part of the above result that $\bar\vi_t^-(E)$ is weakly-decreasing, hence $\vi^+_t(E)=-\bar\vi_t^-(E)$ is weakly-increasing.
\end{remark}

\begin{remark}
Let $(\si_t=(Z_t,\cP_t))_{t\in[0,1]}$ be a continuous family of stability conditions such that the map $Z:[0,1]\to\Ga_\bC\dual$ is differentiable.
For any $t\in[0,1]$ we define the quadratic form
\begin{equation}
Q_t:\Ga_\bR\to\bR,\qquad \ga\mto\Im(Z'_t(\ga)\cdot \bar Z_t(\ga)),
\end{equation}
which is zero on $\Ker Z_t$.
The requirement of Theorem \ref{th:phase behave} means that $Q_t(\cl E)\ge0$ for all $\si_t$-stable (or $\si_t$-semistable) objects $E\in\cD$.
\end{remark}

\subsection{Global dimension of a slicing}
\label{sec:global dim}
We define the \idef{global dimension} of a slicing \cP to be
%\cite{ikeda_q,qiu_global}
\begin{equation}
\gldim(\cP)=\sup\sets{\vi'-\vi}{\Hom(\cP_\vi,\cP_{\vi'})\ne0}
\in[0,+\infty].
\end{equation}
Note that if $\vi'-\vi<0$, then $\Hom(\cP_\vi,\cP_{\vi'})=0$ by assumption, hence the global dimension is indeed non-negative.
If $\si=(Z,\cP)$ is a stability condition, then we define its global dimension to be $\gldim(\si)=\gldim(\cP)$.

\begin{example}
Let us assume that $\cD$ has the Serre functor, meaning an automorphism $\cS:\cD\to\cD$ such that there are natural isomorphisms $\Hom(X,Y)\iso\Hom(Y,\cS X)\dual$ for all $X,Y\in\cD$.
Let us assume that there exists $d\in\bR$, such that
$\cS(\cP_\vi)=\cP_{\vi+d}$ for all $\vi\in\bR$.
If $X\in\cP_\vi$, $Y\in\cP_{\vi'}$ and $\vi'>\vi+d$,
then $\Hom(X,Y)\iso \Hom(Y,\cS X)\dual=0$.
Therefore $\gldim(\cP)\le d$.
Taking $Y=\cS X$ with $X\in\cP_\vi$, we obtain that $\gldim(\cP)=d$.
\end{example}

\begin{example}
Let us assume that the Serre functor has the form $\cS=\Phi[d]$ for some $d\in\bZ$ and a functor $\Phi$ that preserves the heart $\cA=\cP_{(0,1]}$.
If $X\in\cP_\vi$, $Y\in\cP_{\vi'}$
for $\vi\in(0,1]$ and $\vi'>\vi+d+1$, then
$\cS X\in\cP_{(d,d+1]}$, hence
$\Hom(X,Y)\iso\Hom(Y,\cS X)=0$.
This implies that $\gldim(\cP)\le d+1$.
On the other hand, let $X\in\cP_\vi$ for $\vi\in(0,1]$ and let $Y\in\cP_{\vi'}$ be the first term of the HN-filtration of $\cS X\in\cP_{(d,d+1]}$.
Then $\Hom(X,Y)\ne0$ and $\vi'>d$, hence $\gldim(\cP)>d-1$.
\end{example}

\section{Stability data}
\label{sec:stab data}

\subsection{Graded Lie algebras}\label{sec:graded}
Let $\Ga$ be a free abelian group of finite rank.
%We equip $\Ga_\bR$ with a norm.
Given a \Ga-graded Lie algebra $\fg=\bop_{\ga\in\Ga}\fg_\ga$,
we define its support to be
\begin{equation}
\supp(\fg)=\sets{\ga\in\Ga}{\fg_\ga\ne0}.
\end{equation}
We will say that $\fg$ has a strict support if $\supp(\fg)$
generates a strict semigroup without zero in $\Ga$.
% generates a blunt strict cone in $\Ga_\bR$.
If $\fg$ has a finite strict support, then \fg is nilpotent and we can define the corresponding nilpotent group $G=\exp(\fg)$ equipped with the bijective exponential map $\exp:\fg\to G$.

If $\fg$ has a strict support, we define its pro-nilpotent completion $\hat\fg$ as follows.
Let $S\sbs\Ga$ be the strict semigroup (without zero) generated by $\supp(\fg)$.
We equip $S$ with a partial order as in \S\ref{sec:str semigr}.
Given a finite lower set $I\sbs S$,
we define the Lie algebra
\begin{equation}\label{g-nilp-quot}
\fg_I=\fg/\fm_I\iso\bop_{\ga\in I}\fg_\ga,\qquad
\fm_I=\bop_{\ga\in S\ms I}\fg_\ga.
\end{equation}
%Note that if $I$ is a cofinite ideal, then $\fg_I$ is nilpotent.
%We can also characterize ideals as follows.
%We define a partial order on ~$S$, where $x\le y$ if $y-x\in S\cup\set0$ (note that $S\cap (-S)=\es$).
%A subset $I\sbs S$ is an ideal if and only if it is an upper subset, meaning that if $x\le y$ and $x\in I$, then $y\in I$.
%Let us show that every $\ga\in S$ is contained in the complement of a cofinite ideal.
%By Lemma \ref{lm:strict},
%there exists $u\in\Ga_\bR\dual$ such that
%$S\sbs \sets{\ga\in\Ga}{u(\ga)\ge\nn\ga}$.
%For any $n\ge0$, the set $I_n=\sets{\ga\in S}{u(\ga)>n}$ is a cofinite ideal,
%and every $\ga\in S$ is contained in the complement of $I_n$ for some $n\ge0$.

If $I\sbs J$ are two lower sets, then there is a canonical epimorphism of Lie algebras $\fg_J\to\fg_I$.
We define the completion 
\begin{equation}\label{g-compl}
\hat\fg=\ilim_{I\sbs S}\fg_I,
%=\ilim_{n}\fg_{I_n}
%\iso\prod_{\ga\in\Ga}\fg_\ga,
\end{equation}
where the limit is taken over all finite lower sets $I\sbs S$.
Note that we have an isomorphism of vector spaces 
$\hat\fg\iso\prod_{\ga\in\Ga}\fg_\ga$.
The Lie algebra $\hat\fg$ is pro-nilpotent as the Lie algebras $\fg_I$, for finite lower sets $I\sbs S$, are nilpotent.
We define the corresponding pro-nilpotent group $\hat G=\exp(\hat\fg)=\ilim_I\exp(\fg_I)$ and the bijective exponential map $\exp:\hat\fg\to\hat G$ (we denote its inverse by ~$\log$).

For an arbitrary $\Ga$-graded Lie algebra $\fg$, we 
define the vector space
%(note that it doesn't have a Lie algebra structure in general)
\begin{equation}
\hat\fg=\prod_{\ga\in\Ga\ms\set0}\fg_\ga.
\end{equation}
Given $a\in\hat\fg$, we define
\begin{equation}
\supp(a)=\sets{\ga\in\Ga}{a_\ga\ne0}.
\end{equation}

\subsection{Examples of graded Lie algebras}
The following two examples of graded Lie algebras usually appear in the study of stability data and wall-crossing structures.

\begin{example}
Let $\fg$ be a semisimple Lie algebra over $\bC$, with a Cartan subalgebra $\fh$, the set of roots $\De\sbs\fh\dual$
and a set of positive roots $\De_+\sbs\De$.
Let $\Ga\sbs\fh\dual$ be the root lattice of \fg, so that $\Ga_\bC\iso\fh\dual$ and $\Hom_\bZ(\Ga,\bC)=\Ga_\bC\dual\iso\fh$.
Using the root space decomposition $\fg=\bop_{\al\in\De\cup\set0}\fg_\al$ with $\fg_0=\fh$,
we can consider $\fg$ as a \Ga-graded Lie algebra.
Let $\Pi\sbs\De_+$ be the set of simple roots and $e_i\in\fg_{\al_i}$, $f_i\in\fg_{-\al_i}$,
for $\al_i\in\Pi$, be the standard generators.
The algebra $\fg$ is equipped with the Chevalley involution
which is a Lie algebra automorphism $\om:\fg\to\fg$ given by
\begin{equation}
\om(e_i)=-f_i,\qquad \om(f_i)=-e_i,\qquad
\om(h)=-h,\qquad h\in\fh.
\end{equation}

One can also consider $\fg=\gl_n(\bC)$ with the subalgebra $\fh\sbs\fg$ consisting of diagonal matrices
and the root lattice $\Ga\sbs\fh\dual$ generated by $\eps_i\in\fh\dual$, $\sum_j a_jE_{jj}\mto a_i$.
As before, we have a root space decomposition
$\fg=\bop_{\al\in\De\cup\set0}\fg_\al$ with $\fg_0=\fh$ and $\De=\sets{\eps_i-\eps_j}{i\ne j}$.
Therefore $\fg$ is \Ga-graded.
\end{example}

\def\yy{s}
\begin{example}
Let $\Ga\iso\bZ^n$ be equipped with a skew-symmetric \bZ-valued bilinear form $\ang{\cdot,\cdot}$
and let $R$ be a commutative ring with an invertible element
$\yy\in R$ (in our later considerations ~$R$ is the ring of motivic classes and $\yy=\bL^\oh$, where $\bL$ is the class of the Lefschetz motive).
We define the \idef{quantum torus} to be a \Ga-graded associative $R$-algebra
\begin{equation}
\bT=\bop_{\ga\in\Ga}R x^\ga,\qquad
x^\al\circ x^\beta=\yy^{\ang{\al,\beta}}x^{\al+\beta}.
\end{equation}
We can also consider it as a \Ga-graded Lie algebra,
where the Lie bracket is given by the commutator.
The algebra $\bT$ is equipped with the involution 
$x^\al\mto x^{-\al}$.

Assuming that $\yy^2-1$ is invertible,
we define new generators 
$\bar x^\al=x^\al/(\yy-\yy\inv)$, so that
\begin{equation}
[\bar x^\al,\bar x^\beta]
=\frac{\yy^{\ang{\al,\beta}}-\yy^{-\ang{\al,\beta}}}
{\yy-\yy\inv}
\bar x^{\al+\beta}.
\end{equation}
Taking the limit $\yy\to-1$, we define a new \Ga-graded Lie algebra
\begin{equation}
\bar \bT=\bop_{\ga\in\Ga}\bQ \bar x^\ga,\qquad
[\bar x^\al,\bar x^\beta]=(-1)^{\ang{\al,\be}}\ang{\al,\beta}\bar x^{\al+\beta},
\end{equation}
called the torus Lie algebra.
It is equipped with the involution $\bar x^\al\mto \bar x^{-\al}$.
\end{example}

\subsection{Factorizations}
Let \fg be a \Ga-graded Lie algebra.

\begin{lemma}
%[\Cf\cite{mou_scattering}]
\label{decomp1}
Let \fg be a graded Lie algebra with a finite strict support having a decomposition $\supp(\fg)=S_1\sqcup S_2$
such that $\fg_i=\bop_{\ga\in S_i}\fg_\ga$ are subalgebras for $i=1,2$.
Then the map
$$\exp(\fg_1)\xx \exp(\fg_2)\to \exp(\fg),\qquad (g_1,g_2)\mto g_1g_2,$$
is a bijection.
\end{lemma}
\begin{proof}
There exists $\ga\in\supp(\fg)$ such that $[\fg_\ga,\fg]=0$.
Let us assume that $\ga\in S_1$
and let 
$G_\ga=\exp(\fg_\ga)$.
We have the diagram
\begin{ctikzcd}
\exp(\fg_1)\xx \exp(\fg_2)\dar\rar&\exp(\fg)\dar\\
\exp(\fg_1/\fg_\la)\xx \exp(\fg_2)\rar&\exp(\fg/\fg_\la)
\end{ctikzcd}
where both vertical arrows are $G_\ga$-torsors, the top arrow is $G_\ga$-equivariant,
and the bottom arrow is a bijection by induction on the size of $\supp(\fg)$.
We conclude that the top arrow is also a bijection.
\end{proof}

\begin{corollary}\label{decomp2}
Let \fg be a graded Lie algebra with a strict support having a decomposition $\supp(\fg)=S_1\sqcup S_2$
such that $\fg_i=\bop_{\ga\in S_i}\fg_\ga$ are subalgebras for $i=1,2$.
Then the map
$$\exp(\hat\fg_1)\xx \exp(\hat\fg_2)\to \exp(\hat\fg),\qquad (g_1,g_2)\mto g_1g_2,$$
is a bijection.
\end{corollary}

Let $\fg$ be a \Ga-graded Lie algebra with a strict support
and $Z:\Ga\to\bC$ be a linear map.
Let $\hat\fg\iso \prod_\ga\fg_\ga$ be the pro-nilpotent completion of $\fg$ and $\hat G$ be the corresponding pro-nilpotent group.
Given a strict cone $V\sbs\bC$,
%Let $Z\in\Ga_\bC\dual=\Hom(\Ga,\bC)$.
%For any ray $\ell=\bR_{>0}e^{i\vi}\sbs\bC$, 
we define the Lie algebra 
\begin{equation}
\fg_{V,Z}=\bop_{Z(\ga)\in V}\fg_\ga\sbs\fg
\end{equation}
and consider its pro-nilpotent completion $\hat\fg_{V,Z}\sbs\hat\fg$
and the corresponding pro-nilpotent group $\hat G_{V,Z}\sbs\hat G$.
In particular, for any ray $\ell\sbs\bC$, we consider the pro-nilpotent Lie algebra $\hat\fg_{\ell,Z}\sbs\hat\fg$
and the corresponding pro-nilpotent group $\hat G_{\ell,Z}\sbs\hat G$.

\begin{lemma}\label{lm:factorization}
%Let $Z\in\Ga_\bC\dual$ be such that $Z(\supp\fg)$ is contained in a blunt strict cone $V\sbs\bC$.
For any blunt strict cone $V\sbs\bC$, the map
$$\prod_{\ell\sbs V}\hat G_{\ell,Z}\to \hat G_{V,Z},\qquad (g_\ell)_\ell\mto\prod^\to_{\ell\sbs V} g_\ell,$$
is a bijection, where the product is taken in the clockwise order over all rays $\ell\sbs V$.
\end{lemma}
\begin{proof}
It is enough to prove the statement for $\fg$ having a finite strict support contained in $Z\inv(V)$.
For any ray $\ell\sbs V$, let $S_\ell=\sets{\ga\in\supp(\fg)}{Z(\ga)\in \ell}$.
Then $\supp(\fg)=\bigsqcup_{\ell}S_\ell$ and $\fg_{\ell,Z}=\bop_{\ga\in S_\ell}\fg_\ga$.
Now the statement follows from Lemma \ref{decomp1}.
\end{proof}

%The above result implies that we have a bijection
%\begin{equation}
%\exp_Z:\hat\fg\xto\sim\prod_{\ell\sbs V}\hat\fg_{Z,\ell}\xto{\sim}
%\prod_{\ell\sbs V}\hat G_{Z,\ell}\xto{\sim}
%\hat G.
%\end{equation}
%which is generally different from $\exp:\hat\fg\to\hat G$.

\subsection{Stability data}\label{sec:BPS}
Let $\fg$ be a \Ga-graded Lie algebra 
and let us equip $\Ga_\bR=\Ga\ts\bR$ with a norm ~$\nn\cdot$.
%$Z\in\Ga_\bC\dual=\Hom(\Ga,\bC)$ be fixed.
We define \idef{stability data} \KS[\S2.1]
%(this slightly differs from \KS[\S2.1], where one considers $Z$ as a part of stability data)
to be a pair $(Z,a)$, where $Z\in\Hom(\Ga,\bC)$
and $a=(a_\ga)_\ga\in\hat\fg=\prod_{\ga\in\Ga\ms\set0}\fg_\ga$
is such that there exists $\eps>0$ satisfying 
(the \idef{support property})
\begin{equation}
\supp(a)\sbs\sets{\ga\in\Ga}{\n{Z(\ga)}\ge\eps\nn\ga}.
\end{equation}
Sometimes $Z\in\Hom(\Ga,\bC)$ will be clear from the context and we will call $a\in\hat\fg$ stability data.
%We define a stability datum 
%(or a BPS structure \cite{bridgeland_geometry}) 
%for a fixed $Z\in\Hom(\Ga,\bC)$
%to be an element $a\in\hat\fg$ as above.

For any blunt strict cone $V\sbs \bC$, 
let us consider the semigroup \eqref{S1}
\begin{equation}
S(V,Z,\eps)=\sgr\sets{\ga\in\Ga}{Z(\ga)\in V,\,
\n{Z(\ga)}\ge \eps\nn \ga}\sbs\Ga
\end{equation}
which is strict by Lemma \ref{lm:C VZe}.
We define the corresponding Lie algebra and its pro-nilpotent completion
\begin{equation}
\fg_{V,Z,\eps}=\bop_{\ga\in S(V,Z,\eps)}\fg_\ga,\qquad
\hat\fg_{V,Z,\eps}\iso\prod_{\ga\in S(V,Z,\eps)}\fg_\ga.
\end{equation}
We consider the corresponding pro-nilpotent group $\hat G_{V,Z,\eps}$
and the bijective exponential map 
$\exp:\hat\fg_{V,Z,\eps}\to \hat G_{V,Z,\eps}$.
By Lemma \ref{lm:factorization}, we also have a bijection
\begin{equation}\label{expZ1}
\exp_Z:\hat\fg_{V,Z,\eps}\to \hat G_{V,Z,\eps},\qquad
a\mto\prod_{\ell\sbs V}^\to \exp(a_\ell),\qquad a_\ell=\sum_{Z(\ga)\in\ell}a_\ga\in\hat\fg_{\ell,Z,\eps}.
\end{equation}
where the product is taken in the clockwise order over all rays in $V$.

If $\eps>\eps'>0$, then $S(V,Z,\eps)\sbs S(V,Z,\eps')$ and $\fg_{V,Z,\eps}\sbs \fg_{V,Z,\eps'}$.
We define the ind-pro-nilpotent Lie algebra
\begin{equation}\label{ind-pro Lie alg}
\hat\fg_{V,Z}=\dlim_{\eps>0}\hat\fg_{V,Z,\eps}
\sbs\prod_{Z(\ga)\in V}\fg_\ga\sbs\hat\fg,
\end{equation}
the ind-pro-nilpotent group $\hat G_{V,Z}=\dlim_{\eps>0}\hat G_{V,Z,\eps}$
and the bijective exponential map
\begin{equation}
\exp:\hat\fg_{V,Z}\to\hat G_{V,Z}.
\end{equation}
Note that $\hat\fg_{V,Z}$ is independent of a choice of a norm on $\Ga_\bR$.
Note also that the vector space 
$\prod_{Z(\ga)\in V}\fg_\ga$ is not a Lie algebra in general.
The maps in \eqref{expZ1} induce a bijection
\begin{equation}
\exp_Z:\hat\fg_{V,Z}\to\hat G_{V,Z}.
\end{equation}

Let $a\in\hat\fg$ be stability data.
For any blunt strict cone $V\sbs \bC$, 
%we consider $a_V=(a_\ga)_{Z(\ga)\in V}\in\hat\fg_{V,Z}$ and 
we define
\begin{equation}\label{A_V}
A_V=\prod_{\ell\sbs V}^\to\exp(a_\ell)\in\hat G_{V,Z},
\qquad a_\ell=\sum_{Z(\ga)\in\ell}a_\ga\in\hat\fg_{\ell,Z}.
\end{equation}

The above stability data $a\in\hat\fg$ is equivalent 
(see \KS)
to the family $(A_V\in\hat G_{V,Z})_V$, where $V$ runs through all blunt strict cones in $\bC$, such that for a decomposition $V=V_1\sqcup V_2$ (in the clockwise order), we have the \idef{factorization property}
in $\hat G_{V,Z}$
\begin{equation}
A_V=A_{V_1}A_{V_2}.
\end{equation}

Note that rays $\ell\sbs\bC$ can be identified with elements of the quotient group $\bC^*/\bR_{>0}\iso U(1)$, hence stability data is equivalent to a collection of elements
\begin{equation}
(a_{\ell}\in\hat\fg_{\ell,Z})_{\ell\in U(1)}.
\end{equation} 
Let $\eta:\fg\to\fg$ be an involutive automorphism of Lie algebras that maps $\fg_\ga$ to $\fg_{-\ga}$.
Then it induces an automorphism $\eta:\hat\fg_{\ell,Z}\to\hat\fg_{-\ell,Z}$.
We say that stability data is $\eta$-symmetric if $\eta(a_\ell)=a_{-\ell}$ for all rays $\ell\sbs\bC$.
Equivalently, this means that $\eta(a_\ga)=a_{-\ga}$ for all $\ga\in\Ga$.

\subsection{Families of stability data}
\label{sec:fam stab data}
For more information on continuous families of stability data see \cite{kontsevich_stability,kontsevich_analyticity}.
Let $M$ be a topological space.
% and $Z:M\to\Ga_\bC\dual$ be a continuous map.
We define a \idef{continuous family of stability data} on $M$ to be a pair $(Z,a)$, where
$Z:M\to\Ga_\bC\dual$ is a continuous map
and $a=(a_x\in\hat\fg)_{x\in M}$ is a collection such that
\begin{enumerate}
\item \label{ax1}
For any $x\in M$, there exists an open \nbd $x\in U\sbs M$ and $\eps>0$ such that
$$\supp(a_y)\sbs\sets{\ga\in\Ga}{\n{Z_y(\ga)}\ge\eps\nn\ga}
\qquad\forall y\in U.$$
\item \label{ax2}
For any strict cone $V\sbs\bC$, we define \eqref{A_V}
\begin{equation*}
A_{x,V}=\prod^\to_{\ell\sbs V}\exp(a_{x,\ell})\sbs \hat G_{V,Z_x},
\qquad a_{x,\ell}
=\sum_{Z_x(\ga)\in\ell}a_{x,\ga}\in\hat\fg_{\ell,Z_x}.
\end{equation*}
We require that if $Z_x(\supp a_x)\cap\dd V=\es$ and $\ga\in\Ga$, then the map
$$M\ni y\mto p_\ga\log(A_{y,V})\in\fg_\ga$$
is constant in a \nbd of $x$,
where $p_\ga:\hat\fg\to\fg_\ga$ is the projection.
\end{enumerate} 

The second axiom is called the \idef{Kontsevich-Soibelman wall-crossing formula} \KS[Def.~3].
Sometimes $Z:M\to\Ga_\bC\dual$ will be clear from the context and we will call $a=(a_x\in\hat\fg)_{x\in M}$ a continuous family of stability data.
Locally, we can perform calculations by applying Lemma~ \ref{lower set stability}.

\begin{remark}
The above definition is slightly different from \KS[Def.~3].
Instead of the first axiom one requires in loc.cit.~
a seemingly stronger condition that if $a_x$ satisfies the support property \wrt a quadratic form $Q$, meaning that $Q$ is negative definite on $\Ker Z_x$ (this is an open condition) and $Q(\ga)\ge0$ for all $\ga\in\supp(a_x)$, then stability data also satisfy the support property \wrt $Q$
in a \nbd of $x$.
We will prove in Theorem \ref{th:Qpos3} that this condition is automatically satisfied if $M$ is locally path connected.
For a similar result for families of stability conditions
see Theorem \ref{th:Qpos2}.
\end{remark}

\begin{remark}
Because of the continuity of $Z$, the first axiom is equivalent to the requirement that for any $x\in M$ there exists an open \nbd $x\in U\sbs M$ and $\eps>0$ such that
\begin{equation}
\supp(a_y)\sbs\sets{\ga\in\Ga}{\n{Z_x(\ga)}\ge\eps\nn\ga}
\qquad\forall y\in U.
\end{equation}
%Under this condition, the element 
%$\log (A_{y,V})\in\hat\fg_{V,Z_y}$, for $y\in U$, is contained in $\hat \fg_{V,Z_x,\eps'}\sbs \hat \fg_{V,Z_x}$ (not true!) and we can perform all calculations in this Lie algebra.
\end{remark}

\begin{lemma}\label{lm:shrinking V}
Let $Z:M\to\Ga_\bC\dual$ be continuous and
$a=(a_x\in\hat\fg)_{x\in M}$
be a collection satisfying the first axiom.
Then the following conditions are equivalent
\begin{enumerate}
\item 
For any $x\in M$, $\ga\in\Ga$ and a strict cone $V\sbs\bC$ such that $Z_x(\supp a_x)\cap\dd V=\es$, the map
$y\mto p_\ga\log(A_{y,V})$
is constant in a \nbd of $x$.
\item 
For any $x\in M$, $\ga\in\Ga$ and a strict cone $V\sbs\bC$ such that $Z_x(\Ga)\cap\dd V=\set0$, the map
$y\mto p_\ga\log(A_{y,V})$
is constant in a \nbd of $x$.

\item 
For any $x\in M$, $\ga\in\Ga$
and a strict cone $V\sbs\bC$ such that $Z_x(\ga)\in V$,
there exists a strict cone $Z_x(\ga)\in V'\sbs V$
such that $Z_x(\Ga)\cap\dd V'=\set0$ and the map
$y\mto p_\ga\log(A_{y,V'})$
is constant in a \nbd of $x$.
\end{enumerate}
\end{lemma}
\begin{proof}
It is clear that (1) implies (2) and (3).
To see that (2) implies (1), let us choose $\eps>0$ such that the first axiom is satisfied for $x\in\ M$ and $\n{Z_x(\ga)}\ne\eps\nn\ga$ for all $0\ne \ga\in\Ga$.
We embed $V\sbs \bC$ into a strict cone $V'\sbs\bC$ such that $Z_x(\Ga)\cap\dd V'=\set0$.
We consider the semigroup $S=S(V',Z_x,\eps)$ and the lower set $I=\sets{\ga'\in S}{\ga'\le\ga}$.
By Lemma \ref{lower set stability}, we can assume that $I\sbs S(V',Z_y,\eps)$ is a lower set for all $y$.
Therefore we can restrict our considerations from $\fg$ to the nilpotent Lie algebra $\fg_I=\bop_{\ga\in S}\fg_\ga/\bop_{\ga\in S\ms I}\fg_\ga\iso\bop_{\ga\in I}\fg_\ga$.
By assumption, the map $y\mto A_{y,V'}$ (in the nilpotent group $\exp(\fg_I)$) is constant in a \nbd of $x$.
This implies that the map $y\mto A_{y,V}$ is also constant in a \nbd of $x$.

To see that (3) implies (2), let us choose $\eps>0$ as before.
We consider the semigroup $S=S(V,Z_x,\eps)$ and the lower set $I=\sets{\ga'\in S}{\ga'\le\ga}$.
By Lemma \ref{lower set stability}, we can assume that $I\sbs S(V,Z_y,\eps)$ is a lower set for all $y$.
Therefore we can restrict our considerations from $\fg$ to the nilpotent Lie algebra $\fg_I$.
By induction, we can assume that $y\mto p_{\ga'}\log(A_{y,V})$ is constant in a \nbd of $x$ for all $\ga'<\ga$.
By assumption, we can decompose $V=V_1\sqcup V_2\sqcup V_3$, where $V_i$ are strict cones ordered clockwise, such that $Z_x(\Ga)\cap\dd V_i=\set0$, $Z_x(\ga)\in V_2$ and the map $y\mto p_\ga\log A_{y,V_2}$ is constant in a \nbd of $x$.
Shrinking the \nbd of $x$ if needed, we conclude that $y\mto A_{y,V_i}$ (in the nilpotent Lie group $\exp(\fg_I)$) is constant for $i=1,2,3$.
Therefore also $y\mto A_{y,V}$ is constant.
\end{proof}

Note that the second axiom allows us to express $a_{x,\ga}$ locally as finite Lie expressions of elements $a_{y,\ga'}$. The following result shows that this can be done globally
(\cf \KS[\S2.3]).

\begin{lemma}
Let $M$ be path-connected and $(Z,a)$ be a continuous family of stability data on $M$.
For any $x,y\in M$ and $\ga\in\Ga\ms\set0$, the element $a_{x,\ga}$ is a finite Lie expression of elements $a_{y,\ga'}$.
\end{lemma}
\begin{proof}
We can assume that $M=[0,1]$.
We will show that for any $t\in [0,1]$ and $\ga\in\Ga\ms\set0$, the element $a_{t,\ga}$ is a finite Lie expression of elements $a_{1,\ga'}$.
As $M$ is compact, there exists $\eps>0$ such that 
$$\supp a_t\sbs \sets{\ga\in\Ga\ms\set0}{\n{Z_t(\ga)}\ge\eps\nn\ga}$$
for all $t\in M$.
Let $0<\theta<\pi/4$.
By dividing $[0,1]$ into a finite union of intervals,
we can assume that 
$\nn{Z_t-Z_1}\le \frac\eps2\sin(\theta)$
for all $t\in[0,1]$.
For any $\ga\in\supp a_t$, we have 
$\n{Z_t(\ga)-Z_1(\ga)}\le\frac\eps2\nn\ga\le\toh\n{Z_t(\ga)}$, hence 
$$\n{Z_1(\ga)}\ge\toh \n{Z_t(\ga)}\ge\tfrac\eps2\nn\ga.$$
On the other hand, if $\ga\in\Ga\ms\set0$ satisfies
$\n{Z_1(\ga)}\ge\tfrac\eps2\nn\ga$, then
$$\n{Z_t(\ga)-Z_1(\ga)}\le \tfrac\eps2\sin(\theta)\nn{\ga}\le\sin(\theta)\n{Z_1(\ga)},$$
hence the angle between $Z_t(\ga)$ and $Z_1(\ga)$ is $\le\theta$.
If $\ga=\ga_1+\ga_2$ with $Z_t(\ga_i)$ contained in the same ray as $Z_t(\ga)$ and $\n{Z_1(\ga_i)}\ge\tfrac\eps2\nn{\ga_i}$, 
then the angle between $Z_t(\ga_i)$ and $Z_1(\ga_i)$ is $\le\theta$,
hence the angle between $Z_1(\ga_1)$ and $Z_1(\ga_2)$ is $\le 2\theta<\pi/2$.
This implies that
$$\tfrac\eps2\nn{\ga_i}\le\n{Z_1(\ga_i)}<\n{Z_1(\ga)},\qquad i=1,2.$$
For a fixed $\ga\in\Ga\ms\set0$, we can assume by induction that the result is true for all elements of the finite set $$D=\sets{\ga'\in\Ga\ms\set0}
{\tfrac\eps2\nn{\ga'}\le\n{Z_1(\ga')}<\n{Z_1(\ga)}}.$$
By the second axiom, for any $t\in[0,1]$, there exists an open set $t\in U\sbs[0,1]$ such that, for any $t'\in U$,
the element $a_{t,\ga}$ can be written as a finite Lie expression of $a_{t',\ga'}$  with $Z_t(\ga')$ contained in the same ray as $Z_t(\ga)$.
Therefore, there exists a finite sequence $0=t_0<\dots<t_n=1$ such that,
for any $t'\in[t_i,t_{i+1}]$,
the element $a_{t_i,\ga}$ can be written
as a finite Lie expression of $a_{t',\ga'}$
with $Z_{t_i}(\ga')$ contained in the same ray as $Z_{t_i}(\ga)$.
By the discussion above we conclude that $\ga'\in D\cup\set{\ga}$.
Therefore by the inductive assumption we only need to express $a_{t_{i+1},\ga}$ in terms of $a_{1,\ga'}$.
Now we repeat the previous step.
\end{proof}

\begin{theorem}\label{th:Qpos3}
Let $M$ be a path-connected topological space and $(Z,a)$ be a continuous family of stability data on~ $M$.
Let $Q:\Ga_\bR\to\bR$ be a quadratic form that is negative semi-definite on $\Ker Z_x$ for all $x\in M$.
If there exists $x\in M$ such that 
$\supp a_x\sbs\set{Q\ge0}$, then the same is true for all points of $M$.
\end{theorem}
\begin{proof}
We will use the same notation and assumptions as in the previous lemma and we assume that $Q(\ga)\ge0$ for all $\ga\in\supp a_1$.
Let $\ga\in\supp a_0$ and let $0\le i\le n$ be the maximal number such that $\ga\in\supp a_{t_i}$.
If $i=n$, then we are done.
Otherwise, we can express $\ga=\sum_{k=1}^m\ga_k$, where $\ga_k\in\supp a_{t_{i+1}}$ and $Z_{t_i}(\ga_k)$ are contained in the same ray as $Z_{t_i}(\ga)$.
We conclude that $\ga_k\in D$, hence by induction $Q(\ga_k)\ge0$.
By Lemma \ref{convex lQ}, we obtain $Q(\ga)\ge0$.

Alternatively, we can apply Theorem \ref{th:Qpos1}.
For any $x\in M$, let $\fs_x=\supp(a_x)$ and let $\fr_x=\bigcup_{\ell\sbs\bC}\fr_{x,\ell}$, where $\fr_{x,\ell}$ is the set of minimal elements in the strict semigroup $\sgr(\fs_x\cap Z\inv(\ell))$.
It is easy to see that $(Z_x,\fr_x,\fs_x)_{x\in M}$ is a continuous family of stable supports
(\cf Lemma \ref{open-closed1}),
%\S\ref{sec:supp famil}, 
hence we can apply Theorem \ref{th:Qpos1}.
\end{proof}

The following results follows from \KS[Theorem 3].

\begin{theorem}
\label{th:unique lift}
Let $M$ be a path-connected topological space, $(Z,a)$ be a continuous family of stability data on~ $M$ and $Q:\Ga_\bR\to\bR$ be a quadratic form that
is negative-definite on $\Ker Z_x$ and satisfies $\supp a_x\sbs\set{Q\ge0}$ for all $x\in M$.
If $x,y\in M$ are such that $Z_x=Z_y$, then $a_x=a_y$.
\end{theorem}
%\begin{proof}
%We can assume that $M=[0,1]$ and $Z_0=Z_1$.
%By \KS[Theorem 3] there exists an open connected set $U\sbs \Hom(\Ga,\bC)$ such that $Z_t\in U$ for all $t\in[0,1]$ and
%a continuous family of stability data $(Z'_u,a'_u)_{u\in U}$ such that $Z':U\to \Hom(\Ga,\bC)$ is the canonical embedding and $a'_{Z_0}=a_0$
%
%Then there exists $\eps>0$ such that $\supp(a_x)\sbs \sets{\ga\in\Ga}{\n{Z_x(\ga)}\ge\eps\nn\ga}$.
%\end{proof}
%\begin{example}
%Let us consider $\Ga=\bZ^2$ and the maximal ideal $\hat\fg\sbs\bQ\pser{x,y}$ considered as an abelian Lie algebra. We consider
%$$Z:[0,1]\to\Hom(\Ga,\bC),\qquad Z_t(m,n)=me^{\ip t}+\bi n$$
%and the constant family of stability data $a_t=x+y\in\hat\fg$.
%\end{example}

This result implies that it is enough to parametrize stability data by subsets of $\Hom(\Ga,\bC)$ as long as stability data are supported on $\set{Q\ge0}$ for a fixed quadratic from $Q$.

\subsection{Wall-crossing structures}
\label{sec:WCS}

In this section we introduce wall-crossing structures following ~\KSW 
and we discuss their relationship to continuous families of stability data.
A simplified definition of a wall-crossing structure, without usage of WCS sheaves, is given in \eqref{WCS cond}.

\subsubsection{Group slicing}
\label{sec:gr slicing}
Let $\fg$ be a \Ga-graded Lie algebra with a finite strict support.
For any $x\in\Ga_\bR\dual$, we define the subalgebras
\begin{equation}\label{p0m-parts}
\fg^{(x)}_{+}=\bop_{x(\ga)>0}\fg_\ga,\qquad
\fg^{(x)}_{0}=\bop_{x(\ga)=0}\fg_\ga,\qquad
\fg^{(x)}_{-}=\bop_{x(\ga)<0}\fg_\ga.
\end{equation}
The Lie algebra $\fg$ is nilpotent, hence we can define the corresponding nilpotent groups
\begin{equation}
G=\exp(\fg),\qquad
G^{(x)}_\star=\exp(\fg^{(x)}_\star),\qquad
\star\in\set{+,0,-}.
\end{equation}
By Lemma \ref{decomp1}, there is a bijection
\begin{equation}
G_+^{(x)}\xx G_0^{(x)}\xx G_-^{(x)}\to G,\qquad (g_+,g_0,g_-)\mto g=g_+g_0g_-.
\end{equation}
Taking the inverse and the projection to the middle term,
we obtain a map (which is not a group homomorphism in general)
\begin{equation}\label{proj pi}
\pi_x:G\to G^{(x)}_0,\qquad g\mto g_0.
\end{equation}
Similarly, let $\fg$ be a \Ga-graded Lie algebra with a strict support,
$\hat\fg$ be its pro-nilpotent completion and $\hat G$ be the corresponding pro-nilpotent group.
For any $x\in\Ga_\bR\dual$, we consider Lie algebras ~\eqref{p0m-parts}, the corresponding pro-nilpotent Lie algebras $\hat \fg^{(x)}_\pm,\,\hat\fg^{(x)}_{0}$ and the corresponding pro-nilpotent groups $\hat G^{(x)}_\pm,\,\hat G^{(x)}_{0}$.
Then we have a bijection
\begin{equation}
\hat G_+^{(x)}\xx \hat G_0^{(x)}\xx \hat G_-^{(x)}\to \hat G,\qquad (g_+,g_0,g_-)\mto g=g_+g_0g_-,
\end{equation}
and we consider the corresponding projection
\begin{equation}\label{proj pi2}
\pi_x:\hat G\to \hat G^{(x)}_0,\qquad g\mto g_0.
\end{equation}
These projections are natural with respect to morphisms of \Ga-graded Lie algebras.
More precisely, for any homomorphism $f:\fg\to\fh$ of $\Ga$-graded Lie algebras with strict support and $\hat G=\exp(\hat\fg)$, $H=\exp(\hat\fh)$,
we have a commutative diagram
\begin{equation}
\begin{tikzcd}
\hat G\rar["\pi_x"]\dar["\exp(f)"']& \hat G^{(x)}_0\dar["\exp(f)"]\\
\hat H\rar["\pi_x"]& \hat H^{(x)}_0
\end{tikzcd}
\end{equation}

\subsubsection{Special sheaf construction}\label{sheaf}
Let $M$ be a topological space and $(\pi_x:S\to S_x)_{x\in M}$ be a collection of maps between sets.
Then we define the sheaf $\cS$ over $M$ to be the
sheafification of the presheaf
\begin{equation}
M\sps U\mto \sets{(\pi_x(s))_{x\in U}}{s\in S}\sbs\prod_{x\in U}S_x.
\end{equation}
In particular, let us assume that the maps $\pi_x:S\to S_x$ are surjective and
$\forall s,s'\in S$ the set 
\begin{equation}
\sets{x\in M}{\pi_x(s)=\pi_x(s')}
\end{equation}
is open in $M$ (\cf \KS[\S2.1.2]).
Then the set of sections $\Ga(U,\cS)$ over an open set $U\sbs M$
consists of $a\in\prod_{x\in U}S_x$
such that, for any $x\in U$,
there exists an open \nbd $x\in U'\sbs U$
and $s\in S$ 
satisfying 
%(we can choose any $s\in S$ such that $\pi_x(s)=a_x$)
\begin{equation}
\pi_y(s)=a_y\qquad \forall y\in U'.
\end{equation}
The stalk $\cS_x$ of the sheaf $\cS$ is equal to $S_x$ for all $x\in M$.

Let us assume additionally that, for all $x\in M$, we have $S_x\sbs S$ such that $S_x\emb S\xto{\pi_x}S_x$
is the identity.
Then $\Ga(U,\cS)$ consists of $a\in\prod_{x\in U}S_x$ (or $a\in S^U=\prod_{x\in U}S$) such that, for any $x\in U$,
there exists an open \nbd $x\in U'\sbs U$
satisfying 
\begin{equation}
\pi_y(a_x)=a_y\qquad \forall y\in U'.
\end{equation}

\subsubsection{WCS sheaf: finite strict support}
Let $\fg$ be a \Ga-graded Lie algebra with a finite strict support.
As in \S\ref{sec:gr slicing}, we have the maps $\pi_x:G\to G_0^{(x)}$ for $x\in\Ga_\bR\dual$.
For any $g,g'\in G$, the set
\begin{equation}
\sets{x\in\Ga_\bR\dual}{\pi_x(g)=\pi_x(g')}
\end{equation}
%=\sets{x\in\Ga_\bR\dual}{a_\ga\ne0\imp x(\ga)\ne0}
%=\Ga_\bR\dual\ms\bigcup_{\ga\col a_\ga\ne b_\ga}\ga^\perp
is open, see \eg \cite{mou_scattering}.
Applying construction from \S\ref{sheaf} to the topological space $\Ga_\bR\dual$ and the collection of surjective maps $(\pi_x)_{x\in\Ga_\bR\dual}$,
we define the sheaf $\WCS_\fg$ of \idef{wall-crossing structures} over ~$\Ga_\bR\dual$ as follows.
For any open set $U\sbs\Ga_\bR\dual$,
the set of sections $\Ga(U,\WCS_\fg)$ consists of elements 
$a\in\prod_{x\in U}\fg_0^{(x)}$
(or $a\in\fg^U$)
such that, for any $x\in U$,
there exists an open \nbd $x\in U'\sbs U$
satisfying 
\begin{equation}
\pi_y(e^{a_x})=e^{a_y}\qquad \forall y\in U'.
\end{equation}

\subsubsection{WCS sheaf: strict support}
Let $\fg$ be a \Ga-graded Lie algebra having a strict support and let $S\sbs\Ga$ be the strict semigroup generated by this support.
As in \eqref{g-compl},
we define the pro-nilpotent Lie algebra
$\hat\fg=\ilim_{I\sbs S}\fg_I\iso\prod_{\ga\in\Ga}\fg_\ga$,
where the limit runs over finite lower sets $I\sbs S$ and $\fg_I$ denotes the corresponding nilpotent Lie algebra \eqref{g-nilp-quot}.
We define the sheaf $\WCS_\fg$ of wall-crossing structures over $\Ga_\bR\dual$ to be
\begin{equation}
\WCS_\fg=\ilim_{I\sbs S} \WCS_{\fg_I},
\end{equation}
where the limit runs over finite lower sets $I\sbs S$.
%where $u\in\Ga_\bR\dual$ is such that $C\sbs C_u$.
For any open set $U\sbs\Ga_\bR\dual$, the set of sections is equal to
\begin{equation}
\Ga(U,\WCS_\fg)=\ilim_{I\sbs S}\Ga(U,\WCS_{\fg_I}).
\end{equation}
Explicitly, this means that $\Ga(U,\WCS_\fg)$ consists of 
$a\in\prod_{x\in U}\hat\fg^{(x)}_0$ 
(or $a\in\hat\fg^U$)
such that
\begin{equation}
(p_Ia_x)_{x\in U}\in\Ga(U,\WCS_{\fg_I}),
\end{equation}
where $p_I:\hat\fg\to\fg_I$ is the projection,
for any finite lower set $I\sbs S$.
%Note that similarly to \eqref{proj pi} we have a projection $\pi_x:\hat G\to\hat G_0^{(x)}$.
This means that $a\in\prod_{x\in U}\hat\fg^{(x)}_0$ is contained in $\Ga(U,\WCS_\fg)$ if and only if 
for any $x\in U$ and $\ga\in S$,
there exists an open set $x\in U'\sbs U$ such that
\begin{equation}
a_{y,\ga}=p_\ga\log(\pi_y(e^{a_x}))\qquad \forall y\in U',
\end{equation}
where $p_\ga:\hat\fg\to\fg_\ga$ is the projection and $\pi_y:\hat G\to\hat G_0^{(y)}$ is the map defined in \S\ref{sec:gr slicing}.

%Note that $\supp a_x\sbs x^\perp$.

%As in \S\ref{sec:graded}, let us choose a norm on $\Ga_\bR$ and an element $u\in\Ga_\bR\dual$ such that $\supp(\fg)\sbs \sets{x\in\Ga_\bR}{u(x)\ge\nn x}$.
%Following \S\ref{sec:graded},
%we define the monoid $S\sbs\Ga$ generated by $\supp(\fg)$.
%For any $n\ge0$, we define the finite lower set $I_n=\sets{\ga\in S}{u(\ga)\le n}$ and
%the corresponding nilpotent Lie algebra
%\begin{equation}
%\fg_n=\fg/\bop_{\ga\in S\ms I_n}\fg_\ga\iso\bop_{\ga\in I_n}\fg_\ga.
%\end{equation}
%Let $\hat\fg=\ilim_n\fg_n$ and $p_n:\hat\fg\to\fg_n$
%be the canonical epimorphism.
%We define the sheaf $\WCS_\fg$ over $\Ga_\bR\dual$ to be
%\begin{equation}
%\WCS_\fg=\ilim_n \WCS_{\fg_{n}},
%\end{equation}
%%where $u\in\Ga_\bR\dual$ is such that $C\sbs C_u$.
%Its set of sections over an open set $U\sbs\Ga_\bR\dual$ is equal to
%\begin{equation}
%\Ga(U,\WCS_\fg)=\ilim_n\Ga(U,\WCS_{\fg_n}).
%\end{equation}
%Explicitly, $\Ga(U,\WCS_\fg)$ consists of 
%$a\in\prod_{x\in U}\hat\fg^{(x)}_0$ 
%(or $a\in\hat\fg^U$)
%such that
%\begin{equation}
%(p_na_x)_{x\in U}\in\Ga(U,\WCS_{\fg_n})
%\qquad\forall n\ge0.
%\end{equation}

%such that, 
%Moreover, for any $k\ge0$, we consider the
%projection  and we require that the map
%$$U'\to \fg_{C,u}^{(k)},\qquad y\mto p_k(a_y),$$
%is a section in $\Ga(U',\WCS_{\fg_{C,u}^{(k)}})$. 
%
%
%for any $x\in U$, there exists an open \nbd $x\in U'\sbs U$ and
%a strict cone $C=C_u$, $u\in\Ga_\bR\dual$,
%such that
%$$\sets{\ga\in\Ga}{a_{y,\ga}\ne0}\sbs C\ms\set0
%\qquad \forall y\in U'$$
%so that we can consider $a_y=\sum_\ga a_{y,\ga}\in\fg_{C}$.
%

\subsubsection{WCS sheaf: arbitrary support}
Let \fg be an arbitrary \Ga-graded Lie algebra.
Let $\cP$ be the set of all strict semigroups $S\sbs\Ga$, equipped with the partial order induced by inclusion.
For every semigroup $S\in\cP$, 
the Lie algebra $\fg_S=\bop_{\ga\in S}\fg_\ga$
has a strict support.
If $S\sbs S'$, then $\fg_S\sbs\fg_{S'}$ and we obtain the induced inclusion of sheaves $\WCS_{\fg_S}\sbs\WCS_{\fg_{S'}}$.
We define the sheaf $\WCS_\fg$ of wall-crossing structures over $\Ga_\bR\dual$ to be the colimit of sheaves
\begin{equation}
\WCS_\fg=\dlim_{S\in\cP}\WCS_{\fg_S}.
\end{equation}
Note that for any open set $U\sbs\Ga_\bR\dual$,
we have $\Ga(U,\WCS_{\fg_S})\sbs \hat\fg_S^U\sbs\hat\fg^U$,
hence $\Ga(U,\WCS_\fg)\sbs\hat\fg^U$.
%\end{remark}
%We define the sheaf $\WCS_\fg$ over $\Ga_\bR\dual$ as follows.
%For any open set $U\sbs\Ga_\bR\dual$,
More explicitly,
the set of sections $\Ga(U,\WCS_\fg)$ consists of elements 
$a\in\prod_{x\in U}\hat\fg^{(x)}_0\sbs\hat\fg^U$ 
such that, for any $x\in U$, there exists an open \nbd $x\in U'\sbs U$ such that the set
$\bigcup_{y\in U'}\supp(a_y)\sbs\Ga$
generates a strict semigroup  $S\sbs\Ga$ and
%so that $\fg_C=\bop_{\ga\in C\cap\Ga}\fg_\ga$ is a graded Lie algebra with strict support,
%$a_y\in\hat\fg_{C}=\prod _{\ga\in C\cap\Ga}\fg_\ga$
%for all $y\in U'$,
%and we require
\begin{equation}
(a_y)_{y\in U'}\in\Ga(U',\WCS_{\fg_{S}}).
\end{equation}

\subsubsection{Wall-crossing structures}
Given a topological space $M$ and a continuous map $\te:M\to\Ga_\bR\dual$, we define the sheaf
\begin{equation}
\WCS_{\fg,\te}=\te^*\WCS_\fg
\end{equation}
of \idef{wall-crossing structures} over $M$.
We define a \idef{wall-crossing structure} over $M$ to be a global section of $\WCS_{\fg,\te}$.
We can actually describe the sheaf $\WCS_{\fg,\te}$ directly in the same way as the sheaf ~$\WCS_{\fg}$:
\begin{enumerate}
\item If $\fg$ has a finite strict support,
then the set of sections $\Ga(U,\WCS_{\fg,\te})$ 
(for $U\sbs M$)
consists of $a\in\prod_{x\in U}\fg_0^{(\te_x)}$ such that, for any $x\in U$, there exists an open set $x\in U'\sbs U$ such that $\pi_{\te_y}(e^{a_x})=e^{a_y}$ for all $y\in U'$,
where $\pi_{\te_y}:G\to G^{(\te_y)}_0$ is the projection.

\item If \fg has a strict support generating the semigroup $S\sbs\Ga$, then $\Ga(U,\WCS_{\fg,\te})$ (for $U\sbs M$)
consists of 
$a\in\prod_{x\in U}\hat\fg^{(\te_x)}_0$ such that
$(p_Ia_x)_{x\in U}\in\Ga(U,\WCS_{\fg_I,\te})$,
where $p_I:\hat\fg\to\fg_I$ is the projection,
for any lower set $I\sbs S$.

\item
If $\fg$ is arbitrary, then $\Ga(U,\WCS_{\fg,\te})$ 
(for $U\sbs M$)
consists of $a\in\prod_{x\in U}\hat\fg^{(\te_x)}_0$
such that, for any $x\in U$, there exists an opens set $x\in U'\sbs U$ such that 
$\bigcup_{y\in U'}\supp(a_y)$ generates a strict semigroup $S\sbs \Ga$ and 
$(a_y)_{y\in U'}\in\Ga(U',\WCS_{\te,\fg_{S}})$.
\end{enumerate}

We conclude that a wall-crossing structure on $M$ (for a fixed continuous map $\te:M\to\Ga_\bR\dual$) is a collection $a\in\prod_{x\in M}\hat\fg_0^{(\te_x)}$ such that for any $x\in M$ and $\ga\in\Ga$, there exists an open set $x\in U\sbs M$ such that $\bigcup_{y\in U}\supp(a_y)$ generates a strict semigroup $S\sbs\Ga$ and
\begin{equation}\label{WCS cond}
a_{y,\ga}=p_\ga(\log\pi_y(e^{a_x}))\qquad \forall y\in U,
\end{equation}
where we consider projections $\pi_y:\hat G_S\to\hat G_{S,0}^{(\te_y)}$ with $\hat G_S=\exp(\hat\fg_S)$, $\fg_S=\bop_{\ga'\in S}\fg_{\ga'}$.

\begin{remark}
It is proved in \KSW[Theorem 2.1.6] that if $\fg$ has a strict support,
then 
\begin{equation}
\Ga(\Ga_\bR\dual,\WCS_\fg)\iso\hat\fg.
\end{equation}
Wall-crossing structures on $\Ga_\bR\dual$ can be identified with scattering diagrams \KSW.
\end{remark}

\subsubsection{WCS induced by families of stability data}
Let $M$ be a topological space and $(Z_x,a_x)_{x\in M}$ be a continuous family of stability data on a \Ga-graded Lie algebra $\fg$.
Then every $Z_x\in\Ga_\bC\dual$ has the form $Z_x=-\te_x+\bi\rho_x$ for some $\te_x,\rho_x\in\Ga_\bR\dual$.
We define the continuous map
\begin{equation}
\te:M\to\Ga_\bR\dual,\qquad \te_x
=-\Re(Z_x)=\Im(\bi\inv\cdot Z_x).
\end{equation}
Note that if $Z_x(\ga)\in\bH$ for some $\ga\in\Ga$, then $\te_x(\ga)=0$ if and only if $Z_x(\ga)\in\ell_0:=\bi\bR_{>0}$.

\begin{lemma}
The collection 
\begin{equation}
b=(b_x)_{x\in M},\qquad
b_x=(a_{x,\ga})_{Z_x(\ga)\in\ell_0}\in\hat\fg_0^{(\te_x)},
\end{equation}
is a wall-crossing structure on $M$.
\end{lemma}
\begin{proof}
For any $x\in M$, there exists an open set $x\in U\sbs M$ and $\eps>0$ such that
$$\supp(a_y)\sbs\sets{\ga\in\Ga}{\n{Z_y(\ga)}\ge\eps\nn\ga}.$$
Let $0<\eta<\pi/4$ and let us assume that $\nn{Z_x-Z_y}\le\frac\eps2\sin(\eta)$ for all $y\in U$.
Then, for any $\ga\in\supp(a_y)$, we have 
$\n{Z_x(\ga)}\ge\frac\eps2\nn\ga$ and
$$\n{Z_x(\ga)-Z_y(\ga)}<\eps\sin(\eta)\nn{\ga}
\le\sin(\eta)\n{Z_y(\ga)}.$$
This implies that the angle between $Z_x(\ga)$ and $Z_y(\ga)$ is $<\eta$.
Let $V\sbs\bC$ be the cone around $\ell_0$ having angle $2\eta<\pi/2$.
If $\ga\in\supp b_y$, then $Z_y(\ga)\in\ell_0$, hence $Z_x(\ga)\in V$.
We conclude that
$$\supp(b_y)\sbs S=S(V,Z_x,\eps/2)\qquad \forall y\in U,$$
hence $\bigcup_{y\in U}\supp(b_y)$ generates a strict semigroup.
Let us consider $0\ne\ga\in\Ga$ with $\n{Z_x(\ga)}\ge\frac\eps2\nn\ga$.
If $Z_x(\ga)\notin\ell_0$, then we can decrease $\eta$ and shrink $U$ so that $\ga\notin S$ and the condition \eqref{WCS cond} is automatically satisfied.
If $Z_x(\ga)\in\ell_0$, then we can decrease $\eta$ and shrink $U$ so that the finite lower set $I=\sets{\ga'\in S}{\ga'\le\ga}$ is contained in $Z_x\inv(\ell_0)$.
Note the condition \eqref{WCS cond} follows from the second axiom of a continuous family of stability data.  
\end{proof}

Note that the above wall-crossing structure does not capture the full information of the family of stability data as it remembers only stability data along a particular ray.
To solve this problem, we either need to assume that rays rotate in the family $M$ or we can introduce such rotation explicitly as follows \KSW.
Let us consider a new continuous family of stability data
over $\hat M=U(1)\xx M$
\begin{equation}
U(1)\xx M\ni (z,x)\mto (\bi z\inv Z_x,a_x).
\end{equation}
Then we define the corresponding continuous map
\begin{equation}
\hat\te:\hat M\to\Ga_\bR\dual,\qquad (z,x)\mto\Im(z\inv Z_x)
\end{equation}
By the previous lemma, we have a wall-crossing structure
\begin{equation}
b=(b_{z,x})_{(z,x)\in\hat M}\qquad
b_{z,x}=(a_{x,\ga})_{Z_x(\ga)\in z\bR_{>0}}\in\hat\fg_0^{(\hat\te_{z,x})}.
\end{equation}
This wall-crossing structure can be used to completely recover the original continuous family of stability data on $M$.

\section{Wall-crossing formulas}
\label{sec:WCF}

\subsection{Grothendieck ring of stacks}
\label{sec:Gr ring of stacks}
For more details see \eg \cite{bridgeland_introduction,joyce_motivic,toen_grothendieck}.
In this section we will consider only algebraic (or Artin) stacks locally of finite type over \bC.
Let $\Sta$ denote the $2$-category of algebraic stacks of finite type over $\bC$ having affine stabilizers.
Given an algebraic stack $S$ with affine stabilizers, we denote by $K(\Sta\qt S)$ the corresponding Grothendieck group with rational coefficients generated by isomorphism classes $[X\to S]$ of objects  in $\Sta\qt S$ subject to usual relations \cite[\S3]{bridgeland_introduction}.
We can similarly consider the category $\Var$ of finite type algebraic varieties over $\bC$ and the Grothendieck group $K(\Var\qt S)$ with rational coefficients.

In particular, for $S=\Spec\bC$, the corresponding Grothendieck groups $K(\Sta)=K(\Sta\qt \bC)$ 
and $K(\Var)=K(\Var\qt\bC)$
have a ring structure with the product defined by $[X]\cdot[Y]=[X\xx Y]$.
We define $\bL=[\bA^1]$ so that $[\GL_n]=\prod_{i=0}^{n-1}(\bL^n-\bL^i)$.
It is proved in \cite{toen_grothendieck} (see also \cite{bridgeland_introduction}) that there is a natural isomorphism
\begin{equation}
K(\Sta)\iso K(\Var)[[\GL_n]\inv\col n\ge1]\\
=K(\Var)[(\bL-1)\inv,\bL\inv,[\bP^n]\inv\col n\ge1].
\end{equation}
For any algebraic stack $S$ with affine stabilizers, the Grothendieck group $K(\Sta\qt S)$ is a module over $K(\Sta)$ with the product defined by
\begin{equation}
[X]\cdot[Y\to S]=[X\xx Y\to Y\to S].
\end{equation}
Let us also define
\begin{equation}
K^\circ(\Sta\qt S)
=\Im\rbr{K(\Var\qt S)[\bL^{-1},[\bP^n]\inv\col n\ge1]
\to K(\Sta\qt S)}.
\end{equation}
which is a module over $K^\circ(\Sta)=K^\circ(\Sta\qt\bC)$.
Note that the above map is not necessarily injective as $\bL-1$ can be a zero divisor in $K(\Var\qt\bC)$, \cf \cite{borisov_class,martin_class}.

Given a morphism of stacks $f:S\to T$, we define
\begin{equation}
f_*:K(\Sta\qt S)\to K(\Sta\qt T),\qquad [X\to S]\mto[X\to S\to T].
\end{equation}
If $f:S\to T$ is of finite type, then it induces
\begin{equation}
f^*:K(\Sta\qt T)\to K(\Sta\qt S),\qquad [Y\to T]\mto[S\xx_TY\to S].
\end{equation}

In what follows we will 
%consider Grothendieck groups with rational coefficients and we 
introduce the square root $\bL^\oh$ of $\bL$
and define
\begin{equation}\label{R St}
\R(\Sta\qt \bC)=K(\Sta\qt \bC)[\bL^\oh],\qquad
\R^\circ(\Sta\qt \bC)=K^\circ(\Sta\qt \bC)[\bL^\oh].
\end{equation}
%which are modules over $\R(\Sta\qt \bC)$ and $\R_\reg(\Sta\qt \bC)$ respectively.
%We similarly define
%\begin{equation}
%\R_\reg(\Sta\qt S)=K(\Var\qt S)[\bL^{\pm\oh},[\bP^n]\inv\col n\ge1].
%\end{equation}
By the results of Deligne \cite{deligne_theorie}, there exists a ring homomorphism
\begin{equation}\label{E-pol}
E:K(\Var\qt \bC)\to \bQ[u,v],\qquad
[X]\mto
\sum_{p,q,n}(-1)^n \dim \rbr{\Gr^p_F\Gr^W_{p+q}H^n_c(X,\bC)} u^pv^q,
\end{equation}
called the \idef{Hodge-Deligne polynomial map} or the \idef{E-polynomial map}.
Note that $H^*_c(\bA^1,\bQ)=\bQ(-1)[-2]$, hence $E(\bL)=uv$.
We can extend $E$ to
\begin{equation}
E:\R(\Sta\qt \bC)\to \bQ(u,\sqrt{uv}),
\end{equation}
where $\bL^\oh\mto-\sqrt{uv}$.
Similarly, we define the \idef{Poincar\'e polynomial map}
\begin{equation}\label{P-pol}
P:K(\Var\qt \bC)\to \bQ[y],\qquad
[X]\mto E(X;y,y),
\end{equation}
which extends to 
$P:\R(\Sta\qt \bC)\to \bQ(y)$,
where $\bL^\oh\mto -y$.

\subsection{Motivic Hall algebra of an exact category}
\label{sec:motivic Hall}
For more details see \eg \cite{bridgeland_introduction,joyce_configurations2,dyckerhoff_highera}.
Let ~$\cE$ be an exact category, meaning an additive category equipped with a class of short exact sequences (to be called admissible exact sequences) satisfying Quillen's axioms \cite{quillen_higher,keller_chain}.
We define admissible monomorphisms and epimorphisms to be respectively monomorphisms and epimorphisms of admissible short exact sequences.
Let us recall the Waldhausen construction of the simplicial groupoid associated with an exact category $\cE$.
Let us define $\cM_n$ to be the groupoid of chains
\begin{equation}\label{ex chain}
0=E_0\emb E_1\emb\dots\emb E_n,
\end{equation}
where $E_{i-1}\emb E_i$ are admissible monomorphisms,
together with choices of cokernels $E_{j/i}$ of $E_i\emb E_j$.
Note that $\cM_0$ is a point, $\cM=\cM_1$ parametrizes objects of $\cE$, and $\cM_2$ parametrizes admissible exact sequences $0\to E_1\to E_2\to E_{2/1}\to0$ in $\cE$.

Groupoids $\cM_n$ form a simplicial groupoid,
where for any weakly increasing map $\al:[m]\to[n]=\set{0,1,\dots,n}$, we define the face map 
\begin{equation}
\al^*:\cM_n\to\cM_m,\qquad
[E_0\emb \dots\emb E_n]\mto
[E_{\al_0}\emb \dots\emb E_{\al_m}]/E_{\al_0}.
\end{equation}
The fundamental property of this simplicial groupoid is the 2-Segal property \cite{dyckerhoff_highera}.
We will identify increasing maps $\al:[m]\to[n]$ with subsets $\set{\al_0,\dots,\al_m}\sbs[n]$ and we will denote $\al^*$ by $p_{\al_0,\dots,\al_m}$.
In particular, we consider
\begin{gather}
q=(p_{01},p_{12}):\cM_2\to \cM\xx\cM,\qquad [E_1\emb E_2]\mto (E_1,E_{2/1}),\\
p=p_{02}:\cM_2\to\cM,\qquad [E_1\emb E_2]\mto E_2.
\end{gather}

%The above simplicial groupoid satisfies the 2-Segal property, meaning that for $i<j$ in $[n]$, $\al=\set{i,i+1,\dots,j}$ and $\be=[n]\ms\al\cup\set{i,j}$ we have a $2$-Cartesian diagram
%\begin{ctikzcd}
%\cM_n\dar["p_\al"']\rar["p_\be"]&\cM_\be\dar["p_{ij}"]\\
%\cM_\al\rar["p_{ij}"]&\cM
%\end{ctikzcd}

%Then $(\cM_n)_{n\ge0}$ form a simplicial groupoid, satisfying the $2$-Segal property.
%This property implies, in particular, that the simplicial groupoid  $(\cM_n)_{n\ge0}$ is completely determined by its $2$-skeleton.

Let us assume that every $\cM_n$ is equipped with a structure of an 
algebraic stack locally of finite type over $\bC$ with affine stabilizers,
so that all face maps are morphisms of algebraic stacks and 
the map $q:\cM_2\to \cM\xx\cM$ is of finite type.
We define the \idef{motivic Hall algebra}
\begin{equation}
H(\cE)=K(\Sta\qt \cM)
\end{equation}
with the product given by
\begin{equation}
K(\Sta\qt \cM)\ts K(\Sta\qt \cM)\to K(\Sta\qt \cM^2)\xto {q^*}K(\Sta\qt \cM_2)\xto{p_*}K(\Sta\qt \cM),
\end{equation}
The associativity is a consequence of the 2-Segal property.

\begin{remark}
It is actually enough to equip groupoids $\cM_n$ with a slightly weaker structure than an algebraic stack structure.
Given a groupoid $\cG$, let us consider the set of pairs $(X,f)$, where $X\in\Sta$ and $f:X(\bC)\to \cG$ is an equivalence of categories.
We will say that two such pairs $(X_1,f_1)$ and $(X_2,f_2)$ are equivalent if there exist morphisms $g_i:Z\to X_i$ in $\Sta$ such that the functors $Z(\bC)\xto {g_i(\bC)}X_i(\bC)\xto{f_i}\cG$ for $i=1,2$ are naturally isomorphic.
Note that morphisms $g_i:Z\to X_i$ are geometric bijections, meaning that $g_i:Z(\bC)\to X_i(\bC)$ are equivalences of categories.
Therefore there exist finite stratifications of $Z$ and $X_i$ such that $g_i$ induce isomorphisms between the strata \cite[Lemma 3.2]{bridgeland_introduction}.
This implies that $g_i$ induce isomorphisms $g_{i*}:K(\Sta\qt Z)\to K(\Sta\qt X_i)$, hence we obtain an isomorphism $K(\Sta\qt X_1)\iso K(\Sta\qt X_2)$.
We will call an equivalence class of above pairs an \idef{atlas} of $\cG$.
In what follows, we will have a linear map $\cl:K(\cE)\to\Ga$ and a decomposition $\cM=\bigsqcup_{\ga\in\Ga}\cM(\ga)$, where $\cM(\ga)$ is the groupoid parameterizing objects $E\in\cE$ with $\cl E=\ga$.
For our considerations it will be enough to equip every groupoid $\cM(\ga)$ (as well as the fibers of $\cM_n\xto{(p_{i-1,i})_i}\cM^n\to\Ga^n$)
with an atlas.
\end{remark}

\subsection{Wall-crossing in the Hall algebra}
\label{sec:WC hall}
In this section we will restrict our considerations to the triangulated category $\cD=D^b(\Coh X)$, where $X$ is a smooth projective variety over $\bC$, although most of  the results can be formulated for other triangulated categories.
By the results of \cite{lieblich_moduli}, the
stack $\bar\cM$ parameterizing objects $E\in \cD$ with $\Hom^{<0}(E,E)=0$ is algebraic and locally of finite type over $\bC$.
More precisely, we consider the $2$-functor
\begin{equation}
\bar\cM:\Sch\qt\bC\to\Grpd
\end{equation}
that sends a scheme $S$ to the groupoid of objects $E\in D^b(X\xx S)$ that are relatively perfect \cite{lieblich_moduli} and satisfy $\Hom^{<0}(E_s,E_s)=0$ for $s\in S$.
Then $\bar \cM$ is an algebraic stack locally of finite type over \bC.

\begin{remark}
The stack $\bar\cM$ has affine stabilizers.
Indeed, the stabilizer of any object $E$ can be identified with the group $\Aut(E)$ which is the group of invertible elements of the finite-dimensional algebra $A=\Hom_\cD(E,E)$.
Let us consider an injective morphism of algebras  $A\emb\End_\bC(A)$ given by the left multiplication. 
Then the group of invertible elements of $A$ can be identified with $A\cap\GL(A)$, which is a closed subgroup of $\GL(A)$.
\end{remark}

We consider the numerical Grothendieck group $\Ga=\cN(\cD)$ and the projection $\cl:K(\cD)\to \Ga$.
%Recall that we have an isomorphism $\Ga_\bQ=\cN(X)_\bQ\xto\ch\ N^*(X)_\bQ$, hence we can identify $\Ga$ with a subgroup of $N^*(X)_\bQ$.
Let $\si=(Z,\cP)$ be a stability condition on $\cD$ (with respect to $\cl$).
For any interval $I\sbs\bR$ of length $\le1$ (we assume that $I$ is non-closed if it has length $1$),
let $\cM_{\si,I}$ be the substack of $\bar\cM$ parameterizing objects in $\cP_I$.
We will assume that 
\begin{assumption}\label{ass1}
For $I$ of length $1$, the substack $\cM_{\si,I}$ is open in ~$\bar\cM$.
\end{assumption}
\begin{assumption}\label{ass2}
For any $\ga\in\Ga$ with $Z(\ga)\ne0$, 
the stack $\cM_{\si}(\ga)$ is of finite type,
where $\cM_{\si}(\ga)$ is the stack of \si-semistable objects in $\bar\cM$ having class $\ga$ and phase $\frac1\pi\Arg Z(\ga)$.
\end{assumption}

%\begin{enumerate}
%\item for $I$ of length $1$, the substack $\cM_{\si,I}$ is open in ~$\bar\cM$.
%\item
%For any $\ga\in\Ga$ with $Z(\ga)\ne0$, 
%the stack $\cM_{\si}(\ga)$ is of finite type,
%where $\cM_{\si}(\ga)$ is the stack of \si-semistable objects in $\bar\cM$ having class $\ga$ and phase $\frac1\pi\Arg Z(\ga)$.
%\end{enumerate}

Note that the first assumption implies that for any interval $I$ of length $<1$, the substack $\cM_{\si,I}$ is open in $\bar\cM$.
In particular, the stack $\cM_{\si,\vi}=\cM_{\si,[\vi,\vi]}$, $\vi\in\bR$, is open in $\bar\cM$.
The stack $\cM_\si(\ga)$ is the substack of $\cM_{\si,\vi}$ for $\vi=\frac1\pi\Arg Z(\ga)$, hence it is also open in $\bar\cM$.
The second assumption implies that, for any interval $I$ of length $<1$, the stack $\cM_{\si,I}(\ga)$ of objects in $\cM_{\si,I}$ having class $\ga$ is of finite type.
The above assumptions are proved for a large class of stability conditions on surfaces in \cite{toda_moduli} and on 3-folds in \cite{piyaratne_moduli}.
% (more precisely, only the case $I=(\vi-1,\vi]$ is considered there).

Let $I\sbs\bR$ be an interval of length $<1$ and
\begin{equation}
V=\bR_{>0}e^{\ip I}=\sets{re^{\ip\vi}}{r>0,\,\vi\in I}\sbs\bC
\end{equation}
be the corresponding strict cone.
The category $\cE=\cP_I$ is quasi-abelian \cite{bridgeland_stability}.
In particular, it is automatically exact, with the class of admissible exact sequences consisting of all short exact sequences.
By our assumptions, its objects are parametrized by the algebraic stack $\cM=\cM_{\si,I}$.
The stacks $\cM_n$ parameterizing chains \eqref{ex chain} in $\cE$ are also algebraic.
For example, the fiber of $q:\cM_2\to\cM\xx\cM$ over $(E_1,E_2)$ can be identified with the stack
\begin{equation}\label{q-fiber}
\Ext^1(E_2,E_1)/\Hom(E_2,E_1)
\end{equation}
and we can stratify $\cM\xx\cM$ so that $q$ is a trivial fibration over the strata (\cf \cite[Prop.~6.2]{bridgeland_introduction}).
This also implies that $q$ is of finite type.

%Let $\cM_\bul$ be the Waldhausen simplicial groupoid associated to $\cE$.
We consider the decomposition $\cM=\bigsqcup_{\ga\in\Ga}\cM(\ga)$,
where $\cM(\ga)=\cM_{\si,I}(\ga)$ is the stack parameterizing objects $E\in\cE$ with $\cl E=\ga$.
By the support property of \si, there exists $\eps>0$ such that $\n{Z(E)}\ge\eps\nn{\cl E}$ for any $\si$-semistable object $E$.
Therefore, for any $0\ne E\in\cE$ we have
\begin{equation}
\cl E\in S:=S(V,Z,\eps),
\end{equation}
where $S(V,Z,\eps)$ is the strict semigroup defined in \eqref{S1}.
Let $S_0=S\cup\set0$.
%Note that $\cM(\ga)$ is empty unless $\ga=0$ or $Z(\ga)\in V$.
Then the Hall algebra $H(\cE)=K(\Sta/\cM)$ is $S_0$-graded
\begin{equation}
H(\cE)=\bop_{\ga\in S_0}H_\ga(\cE),\qquad 
H_\ga(\cE)=K(\Sta\qt \cM(\ga)).
\end{equation}

We define the completion of the Hall algebra to be
\begin{equation}\label{call completion}
\hat H(\cE)=\ilim_{J\sbs S}H(\cE)_J,\qquad
H(\cE)_J=H(\cE)/\bop_{\ga\in S\ms J}H_\ga(\cE)
\iso\bop_{\ga\in J\cup\set0}H_\ga(\cE),
\end{equation}
where the limit is taken over all finite lower sets $J\sbs S$.
Note that $\hat H(\cE)\iso\prod_\ga H_\ga(\cE)$ as a vector space.
Let us define
\begin{equation}
\one_{\si,V}=1+\sum_{Z(\ga)\in V}\one_{\ga}\in\hat H(\cE),\qquad
\one_\ga=[\cM(\ga)\to\cM(\ga)]\in H_\ga(\cE).
\end{equation}
On the other hand, for any $\ga\in\Ga\cap Z\inv(V)$, we can consider $\cM_\si(\ga)$ as an open substack of ~$\cM(\ga)$.
We define
\begin{equation}
\one_{\si,\ell}=1+\sum_{Z(\ga)\in\ell}\one_{\si,\ga}\in\hat H(\cE),\qquad
\one_{\si,\ga}=[\cM_\si(\ga)\to\cM(\ga)]\in H_\ga(\cE),
\end{equation}
for any ray $\ell\sbs V$.
The following result is well-known \cite{joyce_configurations2,kontsevich_stability}.
%, albeit in a slightly different setting 
\begin{theorem}\label{th:wc1}
We have
\begin{equation}
\one_{\si,V}=
\prod_{\ell\sbs V}^\to\one_{\si,\ell}.
\end{equation}
\end{theorem}
\begin{proof}
Given $\ga_1,\ga_2\in S\sbs \Ga$, we have $Z(\ga_i)\in\bR_{>0}e^{\ip\vi_i}$ for some $\vi_i\in I$, $i=1,2$.
We will write $\ga_1\succ\ga_2$ if $\vi_1>\vi_2$.
By the HN-property of $\si$, for any $E\in\cE$, there exists a unique filtration 
$$0=E_0\sbs\dots\sbs E_n=E$$
such that $E_i/E_{i-1}\in\cP_{\vi_i}$ with $\vi_i\in I$ satisfying $\vi_1>\dots>\vi_n$.
We define the \si-HN type of $E$ to be the tuple $(\ga_1,\dots,\ga_n)$, where $\ga_i=\cl(E_i/E_{i-1})\in S$
and $\ga_1\succ\ga_2\succ\dots\succ\ga_n$.
By Lemma \ref{lm:strict sgr}, for any fixed $\ga\in S$, there exist finitely many tuples $(\ga_1,\dots,\ga_n)$ in $S$ such that $\ga=\sum_i\ga_i$.
This implies that $\cM(\ga)$ has a finite stratification with strata $\cM_\si(\ga_1,\dots,\ga_n)$ parameterizing objects in $\cM(\ga)$ having \si-HN type $(\ga_1,\dots,,\ga_n)$.
By the uniqueness of the HN-filtration we have
$$[\cM_\si(\ga_1,\dots,\ga_n)\to\cM]
=[\cM_\si(\ga_1)\to\cM]*\dots*[\cM_\si(\ga_n)\to\cM].$$
This implies that
$$[\cM(\ga)\to\cM]
=\sum_{\ov{\ga_1+\dots+\ga_n=\ga}{\ga_1\succ\dots\succ\ga_n}}
[\cM_\si(\ga_1)\to\cM]*\dots*[\cM_\si(\ga_n)\to\cM]$$
which is a reformulation of the required formula.
\end{proof}

\begin{remark}
Note that it is important in the above theorem that the interval $I$ has length $<1$.
If $I$ has length $1$, for example $I=(0,1]$, then the corresponding semigroup $S(V,Z,\eps)$ is not necessarily strict and the ordered product in the theorem is not necessarily well-defined.
\end{remark}

%By the support property of \si, there exists $\eps>0$ such that $\n{Z(F)}\ge\eps\nn{\cl F}$ for any $\si$-semistable object $F$.
%In particular, 
%$$\n{Z(\ga_i)}\ge\eps\nn{\ga_i},\qquad
%Z(\ga_i)\in\bR_{>0}e^{\ip\vi_i}\sbs V,$$
%Note that $\ga_i\in S$ and $\ga=\sum_i\ga_i=\cl(E)\in S$.

\subsection{Wall-crossing in the quantum torus}
\label{sec:wc qt}
We use the same notation and assumptions as in the previous section as well as
\begin{assumption}\label{ass3}
The stability condition $\si=(Z,\cP)$ has global dimension $\le2$.
\end{assumption}

Let $\hi$ be the Euler form on $\Ga=\cN(\cD)$ and let $\ang{\cdot,\cdot}$ be its anti-symmetrization
\begin{equation}
\ang{\ga_1,\ga_2}=\hi(\ga_1,\ga_2)-\hi(\ga_2,\ga_1).
\end{equation}
We define the algebra over $R=\R(\Sta\qt\bC)$ \eqref{R St}
\begin{equation}\label{T}
\bT=\bop_{\ga\in\Ga}R x^\ga,\qquad
x^{\ga_1}*x^{\ga_2}=\bL^{\oh\ang{\ga_1,\ga_2}}\cdot x^{\ga_1+\ga_2},
\end{equation}
called the \idef{quantum torus}.
For a strict semigroup $S=S(V,Z,\eps)\sbs\Ga$, we consider the subalgebra 
\begin{equation}
\bT_S=\bop_{\ga\in S\cup\set0}Rx^\ga
\end{equation}
and define its completion $\hat\bT_S$ similarly to 
\eqref{call completion}.
We define the \idef{integration map}
\begin{equation}
\cI:H(\cE)\to \bT_S,\qquad [X\to\cM_{\si,I}(\ga)]\mto \bL^{\oh\hi(\ga,\ga)}[X]\cdot x^\ga
\end{equation}
which extends to
$\cI:\hat H(\cE)\to \hat\bT_S$.
Note that
\begin{gather}
\cI(\one_{\si,V})=1+\sum_{Z(\ga)\in V}
\bL^{\oh\hi(\ga,\ga)}[\cM_{\si,I}(\ga)]x^\ga,\\
\cI(\one_{\si,\ell})
=1+\sum_{Z(\ga)\in\ell}
\bL^{\oh\hi(\ga,\ga)}[\cM_\si(\ga)]x^\ga.
\end{gather}

%\begin{remark}
%Using instead the generators $\hat x^\ga=\bL^{\oh\hi(\ga,\ga)}x^\ga$, we have
%$$\hat x^{\ga_1}*\hat x^{\ga_2}=\bL^{-\hi(\ga_2,\ga_1)}\hat x^{\ga_1+\ga_2}$$
%\end{remark}

\begin{lemma}
Let $X_1,X_2\in\Sta$ and let $X_i\to\cM$, $i=1,2$, be morphisms such that for arbitrary $x_i\in X_i$ and the corresponding objects $E_i\in\cE$, we have $\vi^-_\si(E_1)>\vi^+_\si(E_2)$.
Then
\begin{equation}
\cI([X_1\to\cM]*[X_2\to\cM])
=\cI([X_1\to\cM])*\cI([X_2\to\cM]).
\end{equation}
\end{lemma}
\begin{proof}
We can stratify $X_i$ and assume that they are of the from $Y_i/\GL_{n_i}$, where $Y_i$ are algebraic varieties.
Furthermore, we can substitute $X_i\to\cM$ by $Y_i\to\cM$ and divide both sides of the equation by 
$[\GL_{n_1}]\cdot [\GL_{n_2}]$.
Therefore we can assume that $X_i$ are algebraic varieties.
We can also assume that $X_i$ is mapped to $\cM(\ga_i)$ for some $\ga_i\in S$. Let $\ga=\ga_1+\ga_2$.
Expression on the left is given by the motivic class of the Cartesian product
\begin{ctikzcd}
Z\rar\dar&\cM_2\dar\\
X_1\xx X_2\rar&\cM\xx\cM
\end{ctikzcd}
multiplied by $\bL^{\oh\hi(\ga,\ga)}x^\ga$.
%For any $E_1,E_2\in\cA$ with $\vi^-(E_1)>\vi^+(F)$, 
%the fiber ??? (
%By \cite[Prop.~6.2]{bridgeland_introduction} 
The fiber of $\cM_2\to\cM\xx\cM$ over $(E_1,E_2)$ is given by the stack
%$\RHom^{\le1}(F,E)$) is given by 
$$\Hom^1(E_2,E_1)/\Hom(E_2,E_1).$$
For any $k\ge2$, we have $\vi_\si^-(E_1[k])-\vi^+_\si(E_2)>k\ge2$,
hence $\Hom^k(E_2,E_1)=0$ by the assumption that \si has global dimension $\le2$.
We conclude that the motivic class of the fiber is equal to $\bL^{-\hi(\ga_2,\ga_1)}$.
We note that
$$\oh\hi(\ga,\ga)-\hi(\ga_2,\ga_1)
=\oh\hi(\ga_1,\ga_1)
+\oh\hi(\ga_2,\ga_2)+\oh\ang{\ga_1,\ga_2}
$$
which is exactly the power of $\bL$ on the right hand side of the required equation.
\end{proof}

%Let $\cM_\si(\ga)\sbs\cM_\ga\sbs\cM$ be the substack of \si-semistable objects having class $\ga$. We define
%\begin{equation}
%\one_\si(\ga)=[\cM_\si(\ga)\to\cM]
%\end{equation}

As in Theorem \ref{th:wc1}
we will write $\ga_1\succ\ga_2$ if 
$Z(\ga_i)\in\bR_{>0}e^{\ip\vi_i}$ for $\vi_i\in I$ satisfying $\vi_1>\vi_2$.

\begin{corollary}\label{cor:WC1}
Let $\ga_1\succ\dots\succ\ga_n$ be elements in $S(V,Z,\eps)$.
Then
\begin{equation}
\cI(\one_{\si,\ga_1}*\dots*\one_{\si,\ga_n})=
\cI(\one_{\si,\ga_1})*\dots*\cI(\one_{\si,\ga_n}).
\end{equation}
\end{corollary}

\begin{corollary}\label{cor:WC2}
We have
\begin{equation}
\cI(\one_{\si,V})=
\prod_{\ell\sbs V}^\to \cI(\one_{\si,\ell}).
\end{equation}
\end{corollary}
\begin{proof}
By Theorem \ref{th:wc1} we have
$\cI(\one_{\si,V})
=\cI\rbr{\prod_{\ell\sbs V}^\to\one_{\si,\ell}}$
and by the previous result we have
$\cI\rbr{\prod_{\ell\sbs V}^\to\one_{\si,\ell}}
=\prod_{\ell\sbs V}^\to \cI(\one_{\si,\ell})$.
\end{proof}

\subsection{Stability data associated to a stability condition}
\label{sec:sd for sc}
Let us consider the \Ga-graded Lie algebra
\begin{equation}
\fg=\bop_{\ga\in\Ga}Rx^\ga=\bT,\qquad R=\R(\Sta\qt\bC),
\end{equation}
with the Lie bracket given by the commutator of $\bT$ \eqref{T}.
For any strict semigroup $S\sbs\Ga$ without zero,
we consider the Lie algebra $\fg_S=\bop_{\ga\in S}\fg_\ga$ and the corresponding pro-nilpotent Lie algebra ~$\hat\fg_S$ and the pro-nilpotent group $\hat G_S$.
Note that we can identify $\hat\fg_S$
with the ideal $\hat \bT_S^+=\prod_{\ga\in S}\fg_\ga$ of $\hat\bT_S=\prod_{\ga\in S\cup\set0}\fg_\ga$ and identify $\hat G_S$ with the group $1+\hat \bT_S^+\sbs\hat\bT_S$
so that we have a commutative diagram
\begin{equation}
\begin{tikzcd}
\hat\fg_S\dar\rar["\exp"]&\hat G_S\dar\\
\hat\bT_S^+\rar["\exp"]&1+\hat\bT_S^+  
\end{tikzcd}
\end{equation}
Let $\eps>0$ be such that $\n{Z(E)}\ge\eps\nn{\cl E}$ for any \si-semistable object $0\ne E\in\cD$.
For any ray $\ell\sbs\bC$, we consider the strict semigroup $S(\ell,Z,\eps)$ and define
\begin{gather}
A_{\si,\ell}=\cI(\one_{\si,\ell})
=1+\sum_{Z(\ga)\in\ell}
\bL^{\oh\hi(\ga,\ga)}[\cM_\si(\ga)]x^\ga
\in 1+\hat \bT_{S(\ell,Z,\eps)}^+\iso \hat G_{S(\ell,Z,\eps)}\sbs\hat G_{\ell,Z},\label{A-si-l}
\\
a_{\si,\ell}=\sum_{Z(\ga)\in\ell}a_{\si,\ga}=\log(A_{\si,\ell})\in\hat\fg_{S(\ell,Z,\eps)}\sbs\hat\fg_{\ell,Z},
\end{gather}
where $\hat\fg_{\ell,Z}$ is the ind-pro-nilpotent Lie algebra defined in \eqref{ind-pro Lie alg} and 
$\hat G_{\ell,Z}$ is the corresponding ind-pro-nilpotent group.
We define \idef{stability data associated to the stability condition $\si$} to be the family 
\begin{equation}
a_\si=(a_{\si,\ga})_{\ga\in\Ga\ms\set0}\in\hat\fg.
\end{equation}

\begin{remark}
In the above discussion we considered the quantum torus \eqref{T} over the algebra $\R(\Sta\qt\bC)$ and constructed stability data $a_\si$ with values in $\R(\Sta\qt\bC)$.
Alternatively, we can consider the quantum torus over the algebra $\bQ(u,\sqrt{uv})$ and apply the Hodge-Deligne (polynomial) map \eqref{E-pol} to construct stability data $E(a_\si)$ with values in $\bQ(u,\sqrt{uv})$.
Similarly, we can consider the quantum torus over the algebra $\bQ(y)$ and apply the Poincar\'e (polynomial) map \eqref{P-pol} to construct stability data $P(a_\si)$ with values in $\bQ(y)$.
\end{remark}

\subsection{Families of stability data and stability conditions}
\label{sec:fam of SD for SC}
Let $M$ be a topological space and $(\si_x=(Z_x,\cP_x))_{x\in M}$ be a continuous family of stability conditions on $\cD$ satisfying assumptions \ref{ass1}, \ref{ass2}, and \ref{ass3}.
For every $x\in M$, we define stability data
\begin{equation}
a_x=(a_{\si_x,\ga})_{\ga\in\Ga\ms\set0}\in\hat\fg
\end{equation}
as in the previous section.

%In \KS[Theorem 7] one claims without proof that there is a continous map from the space of stability conditions to the space of stability data
\begin{theorem}\label{th:fam sd from sc}
The family of stability data $(Z_x,a_x)_{x\in M}$ is continuous.
\end{theorem}
\begin{proof}
Axiom \eqref{ax1} follows from Lemma \ref{uniform support}.
Let us verify axiom \eqref{ax2}.
Given a strict cone $V=\bR_{>0}e^{\ip I}$ 
(where $I\sbs\bR$ is an interval of length $<1$),
we have by Corollary \ref{cor:WC2}
$$\cI(\one_{\si_y,V})
=\prod^\to_{\ell\sbs V}\cI(\one_{\si_y,\ell})
=\prod^\to_{\ell\sbs V}A_{y,\ell}=A_{y,V}.$$
Therefore we need to show that components of
$\cI(\one_{\si_y,V})$ are constant in a \nbd of a fixed $x\in M$.
The component of $\cI(\one_{\si_y,V})$ at $\ga\in\Ga$ is 
equal to $[\cM_{\si_y,I}(\ga)]$ (up to some factor depending on $\ga$).
We will compare it to $[\cM_{\si_x}(\ga)]$.

%Let $\fs_y=\sets{\cl E}{E\text{ is }\si_y\text{-semistable}}$ for $y\in M$.
There exists $\eps>0$ and an open set $x\sbs U\sbs M$ such that 
$\n{Z_y(E)}\ge\eps\nn{\cl E}$ for all $\si_y$-semistable objects $E$ and $y\in U$.
We can assume that $\nn{Z_x-Z_y}\le\tfrac\eps2$ for all $y\in U$.
This implies that if $\n{Z_y(\ga)}\ge\eps\nn\ga$, 
then $\n{Z_x(\ga)}\ge\frac\eps2\nn\ga$.
Therefore $S(V,Z_y,\eps)\sbs S(V,Z_x,\frac\eps2)$.
%There exists $\eps_1>0$ such that $\n{Z_x(\ga)}\ge\eps_1\nn\ga$ for $\ga\in S(V,Z_x,\eps)$ by Lemma \ref{}.
If $\cM_{\si_y,I}(\ga)$ is non-empty, then 
$\ga\in S(V,Z_y,\eps)\sbs S(V,Z_x,\frac\eps2)$.

Let us fix $\ga\in S(V,Z_x,\frac\eps2)$ and let
$Z_x(\ga)\in\ell_0=\bR_{>0}e^{\ip\vi_0}$ for some $\vi_0\in I$.
By Lemma \ref{lm:shrinking V} it is enough to consider open cones $V$ around $\ell_0$ of arbitrarily small angle.
For any $0<\eta<\oh$, let $I_\eta=(\vi_0-\eta,\vi_0+\eta)$
and $V_\eta=\bR_{>0}e^{\ip I_\eta}$.
We can assume that for any $\ga'\le\ga$ in $S(V_\eta,Z_x,\eps/2)$, we have $Z_x(\ga')\in\ell_0$.
%Let $V=\bR_{>0}e^{\ip I}$ and $V'=\bR_{>0}e^{\ip I'}$,
%where $I=(\vi_0-\eta/2,\vi_0+\eta/2)$,
%$I'=(\vi_0-\eta,\vi_0+\eta)$ and $0<\eta\ll1$.
We can also assume that $d(\cP_x,\cP_y)<\eta/2$ for all $y\in U$.

Let $I=I_{\eta/2}$ and $V=V_{\eta/2}$.
If $E\in \cP_{y,I}$, then $\vi_y^\pm(E)\in I$, hence 
$\vi_x^\pm(E)\in I_{\eta}$.
If $\cl E=\ga$, then by our assumption, for any $\si_x$-HN-factor $F$ of $E$, we have $Z_x(\cl F)\in\ell_0$,
hence $E$ is $\si_x$-semistable and $E\in\cP_{x,\vi_0}$.
Conversely, 
if $E\in\cP_{x,\vi_0}$, then $\vi_y^\pm(E)\in I$,
hence $E\in\cP_{y,I}$.
This proves that $\cM_{\si_y,I}(\ga)=\cM_{\si_x}(\ga)$.
The same can be done for all $\ga'\le\ga$ in $S(V_\eta,Z_x,\eps/2)$.
Taking the logarithm $\log(A_{y,V})$, we obtain that its \ga-component is constant for $y\in U$.
\end{proof}
\section{Stability conditions on surfaces}
\label{sec:stab cond surf}

\subsection{Some conventions}

\subsubsection{Numerical intersection ring}
Let $X$ be a smooth projective (connected) variety over $\bC$ of dimension $d$.
Let $A^*(X)$ be its intersection ring (also called the Chow ring) \ful[\S8.3].
It is a graded commutative ring with the involution
\begin{equation}\label{invol}
A^*(X)\to A^*(X),\qquad a\mto a^*=\sum_k (-1)^ka_k,
\end{equation}
where $a_k\in A^k(X)$ denotes the $k$-th component of $a\in A^*(X)$.
We consider the \idef{degree map}
\begin{equation}\label{int}
\int:A^d(X)=A_0(X)\to\bZ,\qquad
\sum_x n_x[x]\mto\sum_x n_x,
\end{equation}
and extend it to $\int:A^*(X)\to\bZ$.
We define the \idef{intersection form} on $A^*(X)$
\begin{equation}\label{pairing}
(a,b)=-\int a^*b.
\end{equation}
It satisfies
\begin{equation}
(a,b)=(-1)^d(b,a).
\end{equation}
%is symmetric if $X$ has even dimension and is anti-symmetric if $X$ has odd dimension.
Let $A_\num^*(X)\sbs A^*(X)$ be the kernel of this form,
having components 
\begin{equation}
A_\num^k(X)=\sets{\Big.a\in A^k(X)\Big.}{\int ab=0\ \forall b\in A^{d-k}(X)}.
\end{equation}
It is an ideal of $A^*(X)$.
We define the \idef{numerical intersection ring}
\begin{equation}
N^*(X)=A^*(X)/A_\num^*(X).
\end{equation}
The above intersection form induces a non-degenerate form on $N^*(X)$.
Every component $N^k(X)$ is a free abelian group of finite rank, as it can be embedded into the torsion free part of $H^{2k}(X,\bZ)$.
We have $N^0(X)\iso\bZ$ and $N^d(X)\iso\bZ$.
The group $N^1(X)$ is called the (torsion free) \idef{N\'eron-Severi group}.
The elements of $N^1(X)$ will be called divisors and the elements of $\NSR(X)=N^1(X)_\bR$ will be called $\bR$-divisors.
For $a\in N^k(X)$ and $b\in N^{d-k}(X)$, we will interpret $ab\in N^d(X)$ as an integer, using the isomorphism $N^d(X)\iso\bZ$.
We also define $N_k(X)=N^{d-k}(X)$ for $k\in\bZ$.
The involution \eqref{invol} produces adjoint operators, namely
\begin{equation}
(ab,c)=(b,a^*c),\qquad a,b,c\in N^*(X).
\end{equation} 

\begin{remark}
The pairing $\si(a,b)=\int ab$ on $N^*(X)$ makes $N^*(X)$ a (graded) symmetric Frobenius algebra, meaning that $\si$ is symmetric, non-degenerate and $\si(ab,c)=\si(a,bc)$ for all $a,b,c\in N^*(X)$.
However, we will not use this pairing.
\end{remark}

If $X$ is a surface, then we can write the intersection form on $N^*_\bR(X)$ as
\begin{equation}\label{int form surf}
(a,b)=a_1b_1-a_0b_2-a_2b_0,\qquad a,b\in N^*_\bR(X).
\end{equation}
%Note that the intersection form \eqref{int form surf} restricted to $N^0_\bR(X)\oplus N^2_\bR(X)\iso\bZ^2$ has signature $(1,1)$.
This intersection form has signature $(2,\rho)$,
where $\rho=\rk N^1(X)$ is the Picard number of ~$X$,
because of the following result.

\begin{theorem}[Hodge index theorem]
Let $X$ be a smooth projective surface.
Then the intersection form on $\NSR(X)$ has signature $(1,\rho-1)$, where $\rho=\rk N^1(X)$.
In particular, for any ample divisor ~$H$, the intersection form restricted to $H^\perp\sbs\NSR(X)$ is negative-definite.
\end{theorem}

\subsubsection{Numerical Grothendieck group}
Let $K(X)=K_0(\Coh X)$ be the Grothendieck group of coherent sheaves on~ $X$.
It is a commutative ring with an involution, where
\begin{equation}
[E]\cdot [F]=[E\ts F],\qquad
[E]^*=[E\dual],
\end{equation}
for $E,F\in\Coh(X)$ with $E$ locally-free. 
The Chern character map $\ch:K(X)\to A^*(X)_\bQ$ is a ring homomorphism preserving involutions.
It induces an isomorphism \ful[\S18.3]
\begin{equation}
\ch:K(X)_\bQ\to A^*(X)_\bQ.
\end{equation}

Let $\hi$ be the Euler form on $K(X)$, defined by
\begin{equation}
\hi([E],[F])=\hi(E,F):=\sum_k (-1)^k\dim\Ext^k(E,F).
\end{equation}
By Serre duality, we have $\hi(E,F)=(-1)^d\hi(F,E\ts\om_X)$,
where $\om_X$ is the canonical bundle of ~$X$.
This implies that the left and the right kernels of the Euler form coincide.
The kernel $K_\num(X)\sbs K(X)$
is an ideal of the ring $K(X)$ and we define the 
\idef{numerical Grothendieck group} (which is a ring)
\begin{equation}
\cN(X)=K(X)/K_\num(X).
\end{equation}
By the Grothendieck-Riemann-Roch theorem, we have
\begin{equation}
\hi([E],[F])=-(\ch E,\ch F\cdot \Td(X)),
\end{equation}
where the Todd class 
$\Td(X)=\td(T_X)\in A^*(X)_\bQ$ is an invertible element.
%\begin{comm}
%$\Td(X)=(1,-\tfrac12{K_X},\hi(\cO_X))$ for a surface.
%\end{comm}
The Chern character induces the ring homomorphism
\begin{equation}
\ch:\cN(X)\to N^*(X)_\bQ
\end{equation}
which becomes an isomorphism after tensoring with $\bQ$.

%We will say that a linear map $Z:K(X)\to\bC$ is \idef{numerical} if it vanishes on $K(X)^\perp$.
%The group of such linear maps can be identified with 
%$\Hom_\bZ(\cN(X),\bC)\iso\Hom_\bZ(N^*(X),\bC)$.

%Let us consider the decreasing topological filtration $F_\bul$ of $K(X)$, where $F_k$ is the subgroup generated by coherent sheaves with support of dimension $\le k$.
%Then there is a surjective graded group homomorphism
%\begin{equation}
%\vi: A_*X\to\Gr^F_\bul K(X),\qquad [V]\mto[\cO_V]
%\end{equation}

\subsubsection{Ample cone}
Let $X$ be a smooth projective variety.
We define the \idef{ample cone} $\Amp(X)\sbs \NSR(X)$ to be the convex (blunt) cone generated by ample classes in $N^1(X)$.
%Its elements will be called ample divisors.
We define the \idef{cone of curves}
$\NE(X)\sbs N_1(X)_\bR$
\begin{equation}
\NE(X)=\sets{\sum \bR_{\ge0}[C_i]}{C_i\sbs X\text{ is an integral curve}}
\end{equation}
and let $\bNE(X)$ be its closure.
Then
\begin{equation}
\Amp(X)=\sets{D\in\NSR(X)}{DC>0\ \forall C\in\bNE(X)\ms\set0}.
\end{equation}
We define the \idef{nef cone} to be the dual of $\bNE(X)$
\begin{equation}
\Nef(X)=\sets{D\in\NSR(X)}{DC\ge 0\ \forall C\in\bNE(X)}.
\end{equation}
It is equal to the closure of $\Amp(X)$, while $\Amp(X)$ is equal to the interior of $\Nef(X)$ (see \eg \cite{lazarsfeld_positivity}).
%Then [Laz 1.4.29]
By the Nakai–Moishezon criterion,
if $X$ is a surface, then $D\in\NSR(X)$ is ample if and only if $D^2>0$ and $DC>0$ for all curves $C\sbs X$.

%It is not true in general that an $\bR$-divisor is ample if $DC>0$ for all curves $C\sbs X$.
%However, this is true for Fano varieties.

%\begin{theorem}[Mori's cone theorem]
%For any ample \bR-divisor $H$, there exist finitely many curves $C_1,\dots,C_k$ such that
%$$\bNE(X)=\bNE(X)_{K_X+H\ge0}+\sum \bR_{\ge0}[C_i],$$
%where $\bNE(X)_{K_X+H\ge0}
%=\sets{C\in \bNE(X)}{(K_X+H)C\ge0}$.
%\end{theorem}

%The following result is a corollary of the Mori's cone theorem (see \eg \cite{lazarsfeld_positivity})
%
%\begin{lemma}\label{fano-ample}
%Let $X$ be a Fano variety, meaning that $-K_X$ is ample. Then
%\begin{enumerate}
%\item  $\bNE(X)$ is a polyhedral cone, spanned by the classes of finitely many curves in $X$.
%\item If $D$ is an $\bR$-divisor such that $DC>0$ for all curves $C\sbs X$, then $D$ is ample. 
%\end{enumerate}
%\end{lemma}
%
%Let us formulate for completeness the following particular case of the Campana-Peternell theorem (see \eg \cite{lazarsfeld_positivity}).
%
%\begin{theorem}
%Let $X$ be a smooth projective surface and $D$ be an $\bR$-divisor such that $D^2>0$ and $DC>0$ for all curves $C\sbs X$.
%Then $D$ is ample.
%\end{theorem}

\subsection{Construction of stability conditions}
Let $X$ be a smooth projective surface and $\cD=D^b(\Coh X)$ be its bounded derived category.
We consider the numerical Grothendieck group $\Ga=\cN(X)=\cN(\cD)$ and the projection $\cl:K(\cD)\to\Ga$.
For any ample divisor $\om\in\NSR(X)$ and a torsion-free sheaf $E\in\Coh(X)$, we define the \idef{slope}
\begin{equation}
\mu_\om(E)=\frac{c_1(E)\cdot\om}{\rk (E)}.
\end{equation}
We say that a torsion-free sheaf $E$ is \idef{$\mu_\om$-semistable}
if, for any subsheaf $0\ne E'\sbs E$,
we have $\mu_\om(E')\le\mu_\om(E)$.
Then every $E\in\Coh(X)$ has a unique filtration (called the \idef{Harder-Narasimhan filtration})
$$E_0\sbs E_1\sbs\dots\sbs  E_n=E,$$
such that $E_0$ is the torsion part of $E$
and $E_i/E_{i-1}$ are $\mu_\om$-semistable and 
have decreasing slopes.
%$$\mu_\om(E_1/E_0)>\dots>\mu_\om(E_n/E_{n-1}).$$
We define 
\begin{equation}
\mu^+_\om(E)=\begin{cases}
+\infty&E_0\ne 0,\\
\mu_\om(E_1)&E_0=0,
\end{cases}
\qquad
\mu^-_\om(E)=\mu_\om(E_n/E_{n-1}).
\end{equation}

\begin{remark}
The above discussion can be interpreted in the framework of \S\ref{sec:stab ab} by considering the linear map
\begin{equation}
Z:\cN(X)\to\bC,\qquad E\mto -c_1(E)\cdot\om+\bi\rk (E).
\end{equation}
Note, however, that $Z(\cO_x)=0$ for $x\in X$, so that $Z$ is not a stability function on $\Coh(X)$.
\end{remark}

Let us define the \idef{discriminant} $\De$ to be the quadratic form on $\cN(X)$ corresponding to the intersection form, meaning that
\begin{equation}\label{disc1}
\De(E):=(\ch(E),\ch(E))=a_1^2-2ra_2,\qquad
\ch(E)=(r,a_1,a_2).
\end{equation}
Recall that it has signature $(2,\rho)$,
where $\rho=\rk N^1(X)$.

\begin{theorem}[Bogomolov inequality]
\label{bog1}
Let $E\in\Coh(X)$ be a $\mu_\om$-semistable sheaf.
%and let $\ch(E)=(r,c_1,\ch_2)$ be its Chern character.
Then 
$$\De(E)\ge0.$$
%\begin{equation}
%\De(E):=(\ch(E),\ch(E))=c_1^2-2r\ch_2\ge0.
%\end{equation}
%%=2rc_2-(r-1)c_1^2
\end{theorem}

%\begin{remark}
%For any $E\in\Coh(X)$, let $c_1,c_2$ be its Chern classes and let $\ch(E)=(r,c_1,\ch_2)$ be its Chern character,
%so that $\ch_2=\toh c_1^2-c_2$.
%One defines the discriminant
%\begin{equation}
%\De(E)=2rc_2-(r-1)c_1^2=c_1^2-2r\ch_2=(\ch(E),\ch(E)).
%\end{equation}
%This means that $\De$ is just the quadratic form corresponding to the intersection pairing.
%By the Bogomolov inequality, if $E$ is $\mu_\om$ semistable, then $\De(E)\ge0$.
%\end{remark}

For any $s\in\bR$,
let $\Coh_{\om,>s}$ be the subcategory of $\Coh(X)$
consisting of $E\in\Coh(X)$ with $\mu^-_\om(E)>s$
and let $\Coh_{\om,\le s}$ be the subcategory of $\Coh(X)$ consisting of $E\in\Coh(X)$ with $\mu_\om^+(E)\le s$.
Then $(\Coh_{\om,>s},\Coh_{\om,\le s})$ is a torsion pair on $\Coh(X)$ and we define the category
\begin{equation}
\Coh_{\om,\#s}
=\ang{\Coh_{\om,\le s}[1],\Coh_{\om,>s}}\sbs \cD
\end{equation}
which is the heart of a bounded t-structure on $\cD$.
In particular, for any $\be\in\NSR(X)$, we define
\begin{equation}
\cA_{\be,\om}=\Coh_{\om,\#\be\om}
=\ang{\Coh_{\om,\le\be\om}[1],\Coh_{\om,>\be\om}}.
\end{equation}

The following result was formulated in \cite[\S2]{arcara_bridgeland}, where it was proved
that the linear map $Z_{\be,\om}$ maps $\cA_{\be,\om}$ to the upper half-plane.
The support property and the HN property were proved in~ 
\cite{toda_stabilityb} (see also \cite{bayer_space}).
This result is a generalization of 
\cite[Lemma 6.2]{bridgeland_stabilityb}.

\begin{theorem}
Let $\be,\om\in\NSR(X)$ and let $\om$ be ample. 
Then the linear map
\begin{equation}\label{Z1}
Z_{\be,\om}:\cN(X)\to\bC,\qquad
Z_{\be,\om}(E)
=(e^{\be+\bi\om},\ch(E))
=-\int e^{-\be-\bi\om}\ch(E)
%=-\int e^{\be+\bi\om}\ch(E\dual)
\end{equation}
is a stability function on the abelian category $\cA_{\be,\om}$.
\end{theorem}

We denote by $\si_{\be,\om}$ the stability condition corresponding to $(Z_{\be,\om},\cA_{\be,\om})$.
If $\ch(E)=(r,a_1,a_2)$, then
\begin{equation}\label{Z2}
Z_{\be,\om}(E)
=\rbr{\frac r2(\om^2-\be^2)+a_1\be-a_2}
+\bi\rbr{a_1-r\be}\om.
\end{equation}
We can also write this expression in the form
\begin{equation}\label{Z-simp}
Z_{\be,\om}(E)=
\begin{cases}
\frac{1}{2r}\rbr{r^2\om^2+\De(E)-(a_1-r\be)^2}
+\bi\rbr{a_1-r\be}\om&r\ne0,\\
(a_1\be-a_2)+\bi a_1\om&r=0.
\end{cases}
\end{equation}
It is convenient to define
\begin{equation}
\ch^\be(E)=e^{-\be}\ch(E)\in N^*(X).
\end{equation}
If $\ch^\be(E)=(r,b_1,b_2)$, then
\begin{equation}\label{Z3}
Z_{\be,\om}(E)=(\toh r\om^2-b_2)+\bi b_1\om.
\end{equation}

\begin{lemma}
The discriminant $\De:\cN(X)_\bR\to\bR$ is negative-definite on $\Ker Z_{\be,\om}$.
\end{lemma}
\begin{proof}
Let $E$ be such that $Z_{\be,\om}(E)=0$ and let $\ch^\be(E)=(r,b_1,b_2)$.
Then $b_1\om=0$ and $b_2=\toh r\om^2$, hence $b_1^2\le0$.
If $(r,b_1,b_2)\ne0$, then either $b_1^2<0$ or $r\ne0$.
Therefore
$$\De(E)=b_1^2-2rb_2=b_1^2-r^2\om^2<0.$$
%and the equality is possible only if $r=b_2=0$ and $b_1^2=0$.
%As $b_1\om=0$, we conclude that $b_1^2<0$ whenever $b_1\ne0$.
\end{proof}

\subsection{Cone of positive planes}
\label{sec:pos plane}
Let $X$ be a smooth projective surface as before.
%Recall that the intersection for on $\NSR(X)$ has signature $(2,\rho)$.
We will say that a subspace $V\sbs N^*_\bR(X)$ is a positive plane if $\dim V=2$ and the intersection form restricted to $V$ is positive-definite.
We define
\begin{align}
\cP(X)=&\sets{x+\bi y\in N^*_\bC(X)}{\ang{x,y}\text{ is a positive plane}},\\
\bar\cP(X)=&\sets{e^{\be+\bi\om}\in N^*_\bC(X)}{\be,\om\in\NSR(X),\,\om^2>0},\\
\bar\cP^+(X)=&\sets{e^{\be+\bi\om}\in N^*_\bC(X)}{\be\in\NSR(X),\,\om\in\Amp(X)}.
\end{align}

\begin{lemma}[\Cf \bri]
\label{lm:pos planes}
We have $\bar\cP^+(X)\sbs\bar\cP(X)\sbs\cP(X)$ and
\begin{equation}\label{barP}
\bar\cP(X)=\sets{u\in N^*_\bC(X)}
{u_0=1,\, (u,u)=0,\,(u,\bar u)>0}.
\end{equation}
The action of $\GL_2^+(\bR)$ on $\cP(X)$ is free and $\bar\cP(X)\to\cP(X)/\GL_2^+(\bR)$ is a bijection.
\end{lemma}
\begin{proof}
For $u\in N^*_\bC(X)$ to be of the form $u=e^{\be+\bi\om}$,
%$D\in N^1_\bC(X)$,
we require that $u_0=1$ and $u_2=\oh u_1^2$.
If $u_0=1$, then $(u,u)=u_1^2-2u_2$.
Therefore we require $u_0=1$ and $(u,u)=0$.
If $u=e^{\be+\bi\om}$, then
$$(u,u)=-\int u^*u=0,\qquad
(u,\bar u)=-\int e^{-2\bi\om}=2\om^2.$$
This implies that $\bar\cP(X)$ is given by \eqref{barP}.
If $u=x+\bi y=e^{\be+\bi\om}$ and $\om^2>0$,
then 
\begin{equation}\label{xy}
(x,x)=(y,y)=\om^2>0,\qquad (x,y)=0,
\end{equation}
hence the vectors $x,y$ span a positive plane and 
%$$x=(1,\be,\toh(\be^2-\om^2)),\qquad y=(0,\om,\be\om)$$
%are linearly independent.
%Therefore 
$\bar\cP(X)\sbs\cP(X)$.

If $u=x+\bi y\in \bar\cP(X)$ and $g=\smat{a&b\\c&d}\in\GL_2^+(\bR)$ are such that $gu\in\bar\cP(X)$, then $g\smat{1\\0}=\smat{1\\0}$, hence $a=1$, $c=0$.
Therefore $gu=(x+by)+idy$.
We conclude from \eqref{xy} that $1+b^2=d^2$ and $bd=0$, hence $b=0$ and $d=\pm1$. 
As $g\in \GL_2^+(\bR)$, we conclude that $g$ is the identity.
This proves that the map $\bar\cP(X)\to\cP(X)/\GL_2^+(\bR)$ is injective.

If $u=x+\bi y\in\cP(X)$, then by applying an element of $\GL_2^+(\bR)$ we can assume that 
%$(x,x)=(y,y)>0$ and 
$(x,y)=0$.
% is such that $x,y$ span a positive plane, then we can find an orthogonal basis (with the change of basis matrix from $\GL_2^+(\bR)$ by changing the order of vectors if needed).
If $x_0=y_0=0$, then $(x,x)=x_1^2>0$, $(y,y)=y_1^2>0$ and $(x,y)=x_1y_1=0$,
which would imply that $N^1_\bR(X)$ contains a positive plane.
This is impossible, hence $\smat{x_0\\y_0}\ne0$ and there exists a matrix in $\bR_{>0}\cdot \SO(2)$ that sends this vector to $\smat{1\\0}$.
Therefore we can assume that $(x,y)=0$, $x_0=1$, $y_0=0$.
By applying a positive scalar to $y$, we can make $(x,x)=(y,y)$.
Then $x+\bi y\in\bar\cP(X)$.
\end{proof}

Let us define
\begin{equation}
\cP^+(X)=\GL_2^+(\bR)\cdot\bar\cP^+(X)\sbs\cP(X)\sbs N^*_\bC(X).
\end{equation}
Note that we have the quotient maps $\cP(X)\to\bar\cP(X)$ and $\cP^+(X)\to\bar\cP^+(X)$ and the subset $\bar\cP^+(X)\sbs\cP^+(X)$ is open, hence the subset $\cP^+(X)\sbs\cP(X)$ is open.

%Let \cA be an abelian category and $\cT\sbs\cA$ be a subcategory such that $${}^\perp\cT=\sets{X\in\cA}{\Hom(X,\cT)=0}$$
%contains just the zero object.
%Let $\cD=D^b(\cA)$ be equipped with the standard t-structure.
%If $X\in\cT^{\le n}$, then
%
%\subsection{Geometric stabilities}

\subsection{Geometric stability conditions}
\label{sec:GSC}
In this section we will give a characterization of stability conditions that are equal to $\si_{\be,\om}$ up to the action of $\wtl\GL_2^+(\bR)$.
Let $X$ be a smooth projective surface, $\cD=D^b(\Coh X)$.
We consider $\Ga=\cN(X)=\cN(\cD)$ and the projection $\cl:K(\cD)\to\Ga$.
We will say that a stability condition $\si=(Z,\cP)$ on $\cD$ is \idef{geometric} if the skyscraper sheaves $\cO_x$ are stable and have the same phase for all $x\in X$.
By rotating the stability condition~ \eqref{stab shift}, we can assume that all $\cO_x$ have phase $1$.
We will say that a stability condition $\si=(Z,\cP)$ on $\cD$ is \idef{special} if the discriminant $\De$ is negative-definite on $\Ker Z$.

\begin{lemma}
Let $(\si_{x})_{x\in M}$ be a continuous family of stability conditions on $\cD$.
Then the subsets 
$$\sets{x\in M}{\si_x\text{ is geometric}},\qquad
\sets{x\in M}{\si_x\text{ is special}}$$
are open in $M$.
\end{lemma}
\begin{proof}
The statement about special stability conditions is clear.
The proof of the statement about geometric stability repeats the arguments of Lemma \ref{open-closed1}. 
\end{proof}

\begin{lemma}[See {\bri[Lemma 10.1]}]
\label{lm:tilt prop}
Let $\si=(Z,\cP)$ be a stability condition on $X$ such that all sheaves $\cO_x$ are stable of phase $1$, and let $\cA=\cP_{(0,1]}$ be its heart.
Then
\begin{enumerate}
\item If $E\in\cA$, then $H^i(E)=0$ for $i\ne0,-1$ and $H^{-1}(E)$ is torsion free.
\item If $E\in\cP_1$ is stable, then either $E=\cO_x$ for some $x\in X$ or $E[-1]$ is a locally-free sheaf.
\item If $E\in\Coh(X)$, then $E\in\cP_{(-1,1]}$ and if $E$ is a torsion sheaf, then $E\in\cA$.
\item The pair of categories
$$\cT=\Coh(X)\cap\cA,\qquad \cF=\Coh(X)\cap(\cA[-1])$$
is a torsion pair in $\Coh(X)$ and $\cA=\ang{\cF[1],\cT}$.
\end{enumerate}
\end{lemma}
%\begin{proof}
%We will write $\Ext^i(E,F)=\Hom(E,F[i])$ for $E,F\in D^b(\Coh X)$.
%Let $E\in\cP_\vi$ be stable with $0<\vi<1$.
%Let us show that $\Hom^i(E,\cO_x)=0$ for $x\in X$ and $i\ne0,-1$.
%For any $i<0$, the phase of $\cO_x[i]$ is $\le 0$, hence $\Hom^i(E,\cO_x)=0$.
%For $i\ge2$, we have by Serre duality 
%$\Hom^{i}(E,\cO_x)=\Hom^{2-i}(\cO_x,E)\dual=0$ as the phase of $E[2-i]$ is $<1$.
%By \cite[Prop.~5.4]{bridgeland_fouriera}, the complex $E$ is isomorphic to a complex $[E^{-1}\to E^0]$ with locally-free $E^i$.
%
%If $E\in\cP_1$ is stable and $E$ is not of the form $\cO_x$, then $\Hom(E,\cO_x)=\Hom(\cO_x,E)=0$ for all $x\in X$. Then the same argument as above shows that $E[-1]$ is locally-free.
%We conclude that if $E\in\cA$ is stable, then (1) is satisfied. Then by applying extensions we conclude that (1) is satisfied for all $E\in\cA$.
%
%\dots
%\end{proof}

%The following result can be compared to .

\begin{theorem}[\Cf{\bri[Prop.~10.3]}]
\label{th:unique shift}
%Let $X$ be a del Pezzo surface, meaning that $-K_X$ is ample.
Let $\si=(Z,\cP)$ be a special geometric stability condition.
% such that the discriminant $\De$ is negative-definite on $\Ker Z$.
Then there exist unique $g\in\wtl\GL_2^+(\bR)$, $\be\in\NSR(X)$ and $\om\in\Amp(X)$ such that $\si[g]=\si_{\be,\om}$.
\end{theorem}
\begin{proof}
There exists a unique $u=x+\bi y\in N^*_\bC(X)$ such that
$Z(E)=(u,\ch E)$ for all $E\in\cD$.
The condition that $Z$ is special means that $\ang{x,y}$ is a positive plane \S\ref{sec:pos plane}.
By Lemma \ref{lm:pos planes} there exist unique $T\in \GL_2^+(\bR)$ and $\be,\om\in\NSR(X)$ with $\om^2>0$ such that 
$T\inv u=e^{\be+\bi\om}$.
We have $T\inv Z(\cO_p)=(e^{\be+\bi\om},\ch(\cO_p))=-1$ for all $p\in X$.
The kernel of $\wtl \GL_2^+(\bR)\to \GL_2^+(\bR)$ (see \S\ref{sec:action}) acts on stability conditions by even shifts, hence there exists a unique $g\in \wtl \GL_2^+(\bR)$ that is mapped to $T\in \GL_2^+(\bR)$ and such that $\cO_p$ are $\si[g]$-stable objects of phase $1$.

Therefore we can assume that $u=x+\bi y=e^{\be+\bi\om}$ and $\cO_p$ are \si-stable objects of phase ~$1$.
Note that $y=(0,\om,\be\om)$.
For any $E$ with $\ch(E)=(r,a_1,a_2)$, we have
$$\Im Z(E)=a_1\om-r\be\om.$$
%Let us show that $\om$ is ample. 
If $C\sbs X$ is a curve, then $\cO_C$ is torsion, hence $\cO_C\in\cA=\cP_{(0,1]}$ by Lemma \ref{lm:tilt prop}(3).
If $Z(\cO_C)$ is real, then $\cO_C\in\cP_1$ which is impossible by Lemma \ref{lm:tilt prop}(2). 
This implies that $\Im Z(\cO_C)=C\cdot\om>0$ for any curve $C\sbs X$.
As $\om^2>0$, we conclude that $\om$ is ample.
%The Nakai–Moishezon criterion 

Now we only need to show that $\cA=\cA_{\be,\om}$.
The proof of this statement goes through the lines of \bri[Prop.~10.3 (Step 2)].
\end{proof}

Let 
\begin{equation}\label{sg stab}
\Stab^+_{\cl}(\cD)\sbs\Stab_{\cl}(\cD)
\end{equation}
be the open subset of special geometric stability conditions.
Let us identify $N^*_\bC(X)$ with $\Hom(\Ga,\bC)$ (where $\Ga=\cN(X)$) by sending $u\in N^*_\bC(X)$ to
\begin{equation}
Z_u:\Ga\to\bC,\qquad Z_u(E)=(u,\ch E).
\end{equation}
Then the open subsets $\cP^+(X)\sbs\cP(X)\sbs N^*_\bC(X)$ can be interpreted as open subsets of $\Hom(\Ga,\bC)$.
By the previous theorem,
the local homeomorphism $\Stab_{\cl}(\cD)\to\Hom(\Ga,\bC)$, $(Z,\cP)\mto Z$, restricts to a local homeomorphism
\begin{equation}
\Stab_{\cl}^+(\cD)\to\cP^+(X)
\end{equation}
and we have bijections
\begin{equation}
\Stab_{\cl}^+(\cD)\qt\wtl\GL_2^+(\bR)\iso
\cP^+(X)\qt\GL_2^+(\bR)\iso\bar\cP^+(X).
\end{equation}

%\oper{NE}
%\begin{remark}
%By Mori's cone theorem, if $-K_X$ is ample,
%then there are countably many curves $C_i$ on $X$ satisfying $0<-K_X\cdot C_i\le \dim X+1$ such that $\ub \NE(X)=\sum \bR_{\ge0}C_i$.
%By Kleiman's criterion an element $\om\in\NSR(X)$ is ample if and only if $\om\cdot C>0$ for all $0\ne C\in\ub\NE(X)$.
%Therefore, if $-K_X$ is ample and $\om C>0$ for all curves $C$, then $\om$ is ample.
%\end{remark}

\subsection{Gieseker and twisted stability}
%Given an ample line bundle $L=\cO(D)$, we define the Hilbert polynomial of $E\in\Coh(X)$ to be 
%$$P(E,n)=\hi(E\ts L^{\ts n})=\int e^{nD}\ch(E)\Td(X)$$
Given an ample divisor $\om\in\NSR(X)$, we define the 
\idef{Hilbert polynomial} of $E\in\Coh(X)$ to be 
\begin{equation}
P_\om(E,n)=\int e^{n\om}\ch(E)\Td(X).
\end{equation}
If $E$ is torsion-free, then we define its \idef{reduced Hilbert polynomial} to be
\begin{equation}
p_\om(E,n)=\frac{P_\om(E,n)}{\rk E}.
\end{equation}
We say that a torsion-free sheaf $E$ is 
\idef{$\om$-Gieseker semistable} if for any $0\ne E'\sbs E$ we have
\begin{equation}
p_\om(E',n)\le p_\om(E,n),\qquad n\gg0.
\end{equation}
Let us give a more explicit characterization 
of Gieseker stability.
We have
\begin{equation}
\Td(X)=(1,-\toh K_X,\hi),\qquad \hi=\hi(\cO_X).
\end{equation}
If $\ch(E)=(r,a_1,a_2)$, then
\begin{equation}
P_\om(E,n)
%=\int a(1,n\om,\toh n^2\om^2)\Td(X)
%=\int a(1,n\om-\toh K_X,\hi+\toh n^2\om^2-\toh nK_X\om)
=\frac r2 n^2\om^2+n(a_1-r\tfrac {K_X}2)\om
+(a_2-a_1\tfrac{K_X}2+r\hi).
\end{equation}
Therefore we can substitute the reduced Hilbert polynomial by
\begin{equation}
\bar p_\om(E,n)
=\frac{na_1\om +(a_2-a_1\tfrac{K_X}2)}{r},\qquad
\ch(E)=(r,a_1,a_2),
\end{equation}
and use it instead of $p_\om(E,n)$ to test Gieseker stability.
More generally, given another divisor $\be\in\NSR(X)$, we define
for any torsion-free sheaf $E$ with $\ch(E)=(r,a_1,a_2)$
\begin{equation}
\nu_{\be}(E)=\frac{a_2-a_1\be}{r},\qquad
p_{\be,\om}(E,n)=n\mu_\om(E)+\nu_{\be}(E)
=\frac{na_1\om +(a_2-a_1\be)}{r}.
\end{equation}
We say that a torsion-free sheaf $E$ is 
$(\be,\om)$-\idef{twisted semistable} (\cf \bri) if, 
for any subsheaf $0\ne E'\sbs E$, we have
\begin{equation}
p_{\be,\om}(E',n)\le p_{\be,\om}(E,n),\qquad n\gg0.
\end{equation}
Equivalently, this means that either $\mu_\om(E')<\mu_\om(E)$ or 
$\mu_\om(E')=\mu_\om(E)$
and
$\nu_{\be}(E')\le\nu_{\be}(E)$.
In particular, $E$ is $\mu_\om$-semistable.

\subsection{Large volume limit}
\label{sec:large volume}
Let us consider $\be,\om\in\NSR(X)$ with ample $\om$ and let us define
\begin{equation}
\si_t=(Z_t,\cP_t)=\si_{\be,t\om},\qquad t>0.
\end{equation}
Note that the heart of $\si_t$ is equal to $\cA_{\be,\om}$ for all $t>0$.
We will say that an object $0\ne E\in\cD$ is $\si_\infty$-semistable (or semistable in the \idef{large volume limit}) if $E$ is $\si_t$-semistable for $t\gg0$.

\begin{lemma}
[See {\bri[Prop.~14.2]}]
\label{limit1}
Let $0\ne E\in\cD$ with $\ch(E)=(r,a_1,a_2)$ be such that
\begin{equation}
r>0,\qquad \Im Z_{\be,\om}(E)=(a_1-r\be)\om>0.
\end{equation}
Then $E$ is $\si_\infty$-semistable if and only if $E$ is a shift of a $(\be,\om)$-twisted semistable sheaf.
\end{lemma}

\begin{lemma}[\Cf{\cite[Lemma 2.7]{bayer_space}}]
\label{limit2}
Let $E\in\cA_{\be,\om}$ be $\si_\infty$-semistable.
Then
\begin{enumerate}
\item If $\rk(E)<0$, then $H\inv(E)$ is a $\mu_\om$-semistable sheaf and $H^0(E)$ has dimension ~$0$.
\item If $\rk(E)\ge0$, then $E=H^0(E)$ is either a $\mu_\om$-semistable or a torsion sheaf.
\end{enumerate}

\end{lemma}
\begin{proof}
Let $\phi_t$ be the phase function on $\cA=\cA_{\be,\om}$ corresponding to $\si_t$.
For any object $F\in\cD$ with $\ch(F)=(r,a_1,a_2)$, 
we obtain from \eqref{Z-simp}
\begin{equation}\label{inf phase}
\vi_\infty(F):=\lim_{t\to\infty}\frac1\pi\Arg Z_t(F)
=\begin{cases}
0&r>0\\
\toh&r=0,\,a_1\om>0\\
1&r=a_1=0,\,a_2>0
\end{cases}
\end{equation}
Note that if $0\ne F\in\cA$, then $\vi_\infty(F)=\lim_{t\to\infty}\vi_t(F)$.
Let 
$$\cT=\Coh_{\om,>\be\om},\qquad \cF=\Coh_{\om,\le \be\om}$$
so that $\cA=\ang{\cF[1],\cT}$.
Then there is an exact sequence in $\cA$
$$0\to E_1[1]\to E\to E_0\to0$$
where $E_1\in\cF$ and $E_0\in\cT$.

Let us assume that $E_1\ne0$.
Then $E_1\in\cF=\Coh_{\om,\le \be\om}$ is torsion-free,
hence $\vi_\infty(E_1)=0$ by \eqref{inf phase}.
This implies that $\vi_{\infty}(E_1[1])=1$.
If $E_0$ has dimension $2$ or $1$, then 
$\vi_{\infty}(E_0)=0$ or $\toh$ by \eqref{inf phase}.
But this implies that $\vi_t(E_1[1])>\vi_t(E_0)$ for $t\gg0$, which is a contradiction.
We conclude that $E_0$ has dimension $0$,
hence $Z_t(E)=-Z_t(E_1)-\ch_2(E_0)$.
Let $\ch(E_1)=(r,a_1,a_2)$.
If $(a_1-r\be)\om=0$, then $\mu_\om(E_1)=\be\om$, hence $E_1\in\Coh_{w,\le\be\om}$ is automatically $\mu_\om$-semistable.
Let us assume that $(a_1-r\be)\om<0$.
If $E_1$ is not $\mu_\om$-semistable, then there exists $0\ne F\sbs E_1$ with $\mu_\om(F)>\mu_\om(E)$ and $F, E_1/F\in\cF$.
This implies that we have an exact sequence in \cA
$$0\to F[1]\to E_1[1]\to (E_1/F)[1]\to0$$
Let $\ch(F)=(s,b_1,b_2)$.
%Therefore we can assume that $(a_1-r\be)\om<0$.
% and $\vi_t(E_1[1])<1$.
If $(b_1-s\be)\om=0$, then $\vi_t(F[1])=1>\vi_t(E_1[1])$ for any $t$, which is a contradiction.
Therefore $(b_1-s\be)\om<0$.
We obtain from \eqref{Z2} and $\si_\infty$-semistability of $E$ that
$$\frac{-r}{(a_1-r\be)\om}\ge \frac{-s}{(b_1-s\be)\om}$$
which implies that $\mu_\om(F)\le\mu_\om(E_1)$.

Let us assume that $E_1=0$, hence $E=E_0\in\cT=\Coh_{w,>\be\om}$.
Let $\ch(E)=(r,a_1,a_2)$.
If $r>0$, then 
$(a_1-r\be)\om>0$ and the statement follows from Lemma \ref{limit1}.
If $r=0$, then $E$ is a torsion sheaf.
\end{proof}

\subsection{Bogomolov-type inequalities}
\label{sec:bog-type}
In this section we will prove a new Bogomolov-type inequality for $\si_{\be,\om}$-semistable objects (Theorem \ref{th:BTI}).
This inequality will be important in our analysis of the behavior of mini/max phases of objects under stability deformations.
Let us first describe the known results.

Let $\be,\om\in\NSR(X)$ with ample $\om$.
We define
\begin{equation}
\bar\De_{\be,\om}(E)=(b_1\om)^2-2b_0b_2\om^2,\qquad b=\ch^\be(E)=e^{-\be}\ch(E).
\end{equation}

Recall that for every $D\in\bNE(X)\ms\set0$, we have $D\om>0$.
Let us introduce a norm on $N_1(X)_\bR\sps \bNE(X)$ and define
\begin{equation}
C=\sup\sets{0,\,\frac{-D^2}{(D\om)^2}}{D\in\bNE(X),\, \nn D=1}
\end{equation}
Then $C\ge0$ and for every effective divisor $D$ we have
\begin{equation}
C(D\om)^2+D^2\ge0.
\end{equation}
We define
\begin{equation}
\bar\De^C_{\be,\om}(E)=\De(E)+C(b_1\om)^2,\qquad
b=\ch^\be(E).
\end{equation}

The following result is proved in 
\cite{toda_stabilityb,bayer_bridgeland,bayer_space}.

\begin{theorem}
Let $E\in D^b(\Coh X)$ be a $\si_{\be,\om}$-semistable object. Then
\begin{equation}
\bar \De_{\be,\om}(E)\ge0,\qquad \bar \De_{\be,\om}^C(E)\ge0.
\end{equation}
\end{theorem}

Before we will formulate a new Bogomolov-type inequality, let us derive a consequence of the classical Bogomolov inequality (Theorem \ref{bog1}).

\begin{lemma}\label{1.4}
Let $\om,\be,H$ be divisors such that $\om$ is ample and $H$ is nef.
For any $\mu_\om$-semistable sheaf $E$, we have
$$\De_{\be,\om}^H(E):=(b_1H)(b_1\om)-b_0b_2(H\om)\ge0,
\qquad b=\ch^\be(E).$$
\end{lemma}
\begin{proof}
We can assume that $\om^2=1$.
Let $t=H\om\ge 0$ and $s\ge0$ be such that $s^2=H^2$.
Then $H=t\om+D_1$ for some $D_1\in\om^\perp$,
hence $s^2=H^2=t^2+D_1^2\le t^2$ and $s\le t$.
By the Bogomolov inequality, we have
$$\De(E)=(\ch(E),\ch(E))
=(\ch^\be(E),\ch^\be(E))=b_1^2-2b_0b_2\ge0.$$
Therefore $2b_0b_2(H\om)\le b_1^2(H\om)$ and we need to prove that
$$b_1^2(H\om)\le 2(b_1H)(b_1\om).$$

We can represent $b_1=xH+y\om+D$, where $x,y\in\bR$ and $D\in \ang{H,\om}^\perp$, hence $D^2\le0$.
It is enough to show that
$$t(x^2s^2+y^2+2xyt)\le 2(xs^2+yt)(xt+y).$$
This simplifies to $ts^2x^2+ty^2+2s^2xy\ge0$.
We have $ts^2x^2+ty^2+2s^2xy\ge s^3x^2+sy^2+2s^2xy
=s(sx+y)^2\ge0$.
\end{proof}

%\begin{lemma}\label{1.4}
%Let $H,\om$ be ample divisors and $\be$ be an arbitrary divisor.
%For any $\mu_\om$-semistable sheaf $E$, we have
%$$\De_{\be,\om}^H(E):=(Hb_1)(\om b_1)-b_0b_2(H\om)\ge0,
%\qquad b=\ch^\be(E).$$
%\end{lemma}
%\begin{proof}
%By the Bogomolov inequality, we have
%$$\De(E)=(\ch(E),\ch(E))
%=(\ch^\be(E),\ch^\be(E))=b_1^2-2b_0b_2\ge0.$$
%Therefore $2b_0b_2(H\om)\le b_1^2(H\om)$ and we need to prove that
%$$b_1^2(H\om)\le 2(Hb_1)(\om b_1).$$
%We can assume that $H^2=\om^2=1$.
%If $H$ and $\om$ are proportional, then $H=\om$.
%We can represent $b_1=xH+D$, where $x\in\bR$ and $D\in H^\perp$, hence $D^2\le0$ by the Hodge index theorem.
%Then
%$$b_1^2(H\om)=x^2+D^2\le x^2\le 2x^2=2(b_1H)(b_1\om).$$
%Let us assume that $H$ and $\om$ are not proportional and let $d=H\om>0$.
%The intersection form restricted to $\ang{H,\om}$ has signature $(1,1)$ and is represented by the matrix $\smat{1&d\\d&1}$, hence $d>1$ by Sylvester’s criterion.
%We can represent $b_1=xH+y\om+D$, where $x,y\in\bR$ and $D\in \ang{H,\om}^\perp$, hence $D^2\le0$.
%It is enough to show that
%$$(x^2+y^2+2dxy)d\le 2(x+dy)(dx+y).$$
%This simplifies to $d(x^2+y^2)+2xy\ge0$
%and we use the fact that $d>1$.
%\end{proof}

\begin{lemma}\label{1.5}
The quadratic form $\De^H_{\be,\om}$ is negative semi-definite on $\Ker (Z_{\be,\om}:\cN(X)_\bR\to\bC)$.
Therefore, for any ray $\ell\sbs\bC$, the cone
$$Z_{\be,\om}\inv(\ell)\cap\set{\De^H_{\be,\om}\ge0}$$
is convex.
\end{lemma}
\begin{proof}
Let $b=\ch^\be(E)$ and $Z_{\be,\om}(E)=\oh b_0\om^2-b_2+ib_1\om=0$.
Then $b_1\om=0$ and $b_2=\oh b_0\om^2$.
Therefore $\De_{\be,\om}^H(E)=-\oh b_0^2 \om^2(H\om)\le0$.
The last assertion follows from 
Lemma \ref{convex lQ}. 
\end{proof}

\begin{theorem}\label{th:BTI}
Let $\om,\be,H\in\NSR(X)$ be divisors such that $\om$ is ample and $H$ is nef.
For any $\si_{\be,\om}$-semistable object $E\in D^b(\Coh X)$,
we have
$$\De_{\be,\om}^H(E)=(b_1H)(b_1\om)-b_0b_2(H\om)\ge0,
\qquad b=\ch^\be(E).$$
\end{theorem}
\begin{proof}
Our proof is similar to the proof of 
\cite[Theorem 3.5]{bayer_space}.
Let us assume first that $\be,\om\in N^1_\bQ(X)$.
Recall that if $a=\ch(E)$, then
\begin{equation*}
Z_{\be,\om}(E)
=\rbr{\frac {a_0}2(\om^2-\be^2)+a_1\be-a_2}
+\bi\rbr{a_1-a_0\be}\om.
\end{equation*}
Let $\rho(E)=\Im Z_{\be,\om}(E)$.
Then the image of the map $\rho:\cN(X)\to\bR$ is discrete.

Let us consider stability conditions $\si_t=\si_{\be,t\om}$ for $t\ge1$.
If $E\in\cA_{\be,\om}$ is $\si_t$-semistable for $t\gg1$, then we can apply Lemma \ref{limit2}.
If $E_1=H^{-1}(E)$ is $\mu_\om$-semistable and $E_0=H^0(E)$ has dimension zero,
let $\ch^\be(E_1)=b$ and $\ch^\be(E_0)=(0,0,n)$.
Then $b_0>0$, $n\ge0$ and
$$\De_{\be,\om}^H(E)
=\De_{\be,\om}^H(E_1)+b_0n(H\om)\ge0$$
by Lemma \ref{1.4}.
If $E=H^0(E)$ is $\mu_\om$-semistable, then 
$\De_{\be,\om}^H(E)\ge0$ by Lemma \ref{1.4}.
If $E=H^0(E)$ is torsion, let $\ch^\be(E)=(0,b_1,b_2)$.
Then $b_1=c_1(E)$ and
$$\De_{\be,\om}^H(E)=(b_1H)(b_1\om)\ge0.$$

By Lemma \ref{1.5}, it is enough to prove the statement of the theorem for any stable object $E\in\cA=\cA_{\be,\om}$.
If $\rho(E)=0$, then $E$ is automatically $\si_t$-semistable of phase $1$ for all $t$, hence $\De_{\be,\om}^H(E)\ge0$ by the previous discussion.
Let $E\in\cA$ be a $\si_{t_0}$-stable object for some $t_0\ge1$ and let us assume by induction on $\rho(E)\ge 0$ (recall that $\rho$ has discrete values) that if $0\ne E'\in\cA$ is $\si_t$-stable for some $t\ge1$ and $\rho(E')<\rho(E)$,
then $\De^H_{\be,\om}(E')\ge0$.
If $E$ is $\si_t$-stable for all $t\ge t_0$, then $\De_{\be,\om}^H(E)\ge0$ by the previous discussion.
Otherwise, there exists $t_1>t_0$ such that $E$ is $\si_t$-stable for all $t\in[t_0,t_1)$ and is strictly $\si_{t_1}$-semistable (\ie semistable, but not stable). 
Then all the $\si_{t_1}$-stable factors $F_i$ of $E$ will have the same $\si_{t_1}$-phase as $E$, hence $\rho(F_i)>0$.
Therefore $\rho(F_i)<\rho(E)$ and we conclude by induction that $\De_{\be,\om}^H(F_i)\ge0$.
As the $\si_{t_1}$-phases of $E$ and $F_i$ are the same, we conclude from Lemma \ref{1.5} that $\De_{\be,\om}^H(E)\ge0$.

Let us assume now that $(\be,\om)\in\NSR(X)\xx\Amp(X)$ is arbitrary and let $E\in D^b(\Coh X)$ be $\si_{\be,\om}$-stable.
Then there exists an open \nbd $U\ni (\be,\om)$ such that $E$ is stable in this \nbd.
We can find a sequence of rational points $(\be_n,\om_n)_n$ in $U$ that converges to $(\be,\om)$.
Then $\De^H_{\be_n,\om_n}(E)\ge0$ by the previous discussion.
Therefore $\De^H_{\be,\om}(E)\ge0$.
\end{proof}

We will also need later a minor modification of the above result.

\begin{theorem}\label{th:BT2}
Let $\om,\be,H\in\NSR(X)$ be divisors such that $\om$ and $H$ are ample.
For any $\si_{\be,\om}$-semistable object $E\in D^b(\Coh X)$ having non-integer phase,
we have
$$\wtl\De_{\be,\om}^H(E)
=(b_1H)(b_1\om)-b_0b_2(H\om)+\toh b_0^2(H\om)\om^2>0,
%=(b_1H)(b_1\om)-b_0b_2(H\om)\ge0,
\qquad b=\ch^\be(E).$$
\end{theorem}
\begin{proof}
Let us use the same notation as in the previous theorem.
We can assume that $E\in\cA_{\be,\om}$ is $\si_{t_0}$-semistable of phase $\ne1$.
If $E$ is $\si_t$-semistable for $t\gg0$ and $b=\ch^\be(E)$, then either $b_0\ne0$ or $b_0=0$
and $E=H^0(E)$ has a $1$-dimensional support. 
In both cases we obtain $\wtl\De_{\be,\om}^H(E)>0$.
Now we repeat the argument of the previous theorem.
More precisely, if we meet an object of non-zero rank in our inductive process, then $\wtl\De^H_{\be,\om}(E)\ge\toh(H\om)\om^2$.
And if $E$ is $\si_t$-stable for $t\in[t_0,t_1)$ and is strictly $\si_{t_1}$-semistable with all $\si_{t_1}$-stable factors having zero rank, then all these $\si_{t_1}$-stable factors have the same $\si_{t_0}$-phase as $E$, which is a contradiction. 
We conclude that $E$ (having zero rank) is $\si_t$-stable for all $t\gg0$. 
Therefore $E=H^0(E)$ has a $1$-dimensional support and 
$\wtl\De_{\be,\om}^H(E)=(b_1H)(b_1\om)>0$.
\end{proof}

\subsection{Alternative parametrization of stability conditions}
\label{altern param}
The following parametrization of stability conditions 
$\si_{\be,\om}$ also appears in the literature.
For any pair of divisors $(\be,\om)\in \NSR(X)\xx\Amp(X)$,
there exists a unique tuple
\begin{equation}
(D,H,s,t)\in \NSR(X)\xx\Amp(X)\xx\bR\xx\bR_{>0}
\end{equation} 
such that
\begin{gather}
H^2=1,\qquad HD=0,
\\
\be+\bi\om
=(sH+D)+\bi tH
=(s+\bi t)H+D.
\end{gather}
Indeed, we require $\om=tH$, hence $t=\sqrt{\om^2}$ and $H=\frac1t\om$.
On the other hand, equation $\be=sH+D$ determines $s$ and $D$ uniquely as we have a decomposition $\NSR(X)=\bR H\oplus H^\perp$.

Note that $\be\om=st$ and we have $\mu_\om(E)\le \be\om$ if and only if $\mu_H(E)\le s$.
Therefore
\begin{equation}
\cA_{\be,\om}=\Coh_{\om,\#\be\om}=\Coh_{H,\#s}.
\end{equation}
From now on, we fix divisors $D$ and $H$ as above and we consider $s$ and $t$ as parameters.
Let us define linear maps
\begin{gather}
v_{H}:N^*_\bR(X)\to\bR^3,\qquad a\mto (a_0,a_1H,a_2),\\
%v_{D,H}:N^*_\bR(X)\to\bR^3,\qquad a\mto v_{0,H}(e^{-D}a),\\
v_{D,H}:K(X)\to\bR^3,\qquad E\mto v_{H}(\ch^D(E)).
\end{gather}
%and similarly $v=v_{D,H}:\cN_\bR(X)\to\bR^3$, using the isomorphism $\ch:\cN_\bR(X)\to N_\bR^*(X)$.
We define a linear map $Z'_{s,t}:\bR^3\to\bC$ such that
$$Z_{\be,\om}(E)=Z'_{s,t}(v_{D,H} (E))\qquad 
\forall  E\in K(X).$$
Explicitly, if $v=v_{D,H}(E)=(a_0,a_1H,a_2)$ with $a=\ch^D(E)$, then
%\begin{equation}
%Z'_{s,t}(v)=Z_{\be,\om}(E)
%=-\int e^{-\be-\bi\om}\ch(E)
%=\int\rbr{-1,(s+it)H,\toh(t^2-s^2)-ist}\cdot(a_0,a_1,a_2)\\
%=\rbr{\frac {a_0}2(t^2-s^2)+sHa_1-a_2}+i\rbr{tHa_1-sta_0}\\
%=\rbr{\frac {v_0}2(t^2-s^2)+sv_1-v_2}+i\rbr{v_1-sv_0}t
%=:\te(v)+it\rho(v).
%\end{equation}
\begin{multline}\label{Zst}
Z'_{s,t}(v)=-\int e^{-(s+\bi t)H}\cdot a
=\rbr{\frac {a_0}2(t^2-s^2)+sHa_1-a_2}+\bi\rbr{tHa_1-sta_0}\\
=\rbr{\frac {v_0}2(t^2-s^2)+sv_1-v_2}+\bi t\rbr{v_1-sv_0}
=:-\te(v)+\bi t\rho(v).
\end{multline}
We can substitute $Z'_{s,t}$ by the equivalent stability function (semistability is preserved)
\begin{gather}
Z''_{s,q}(v)=-\te(v)-s\rho(v)+\bi\rho(v)
=\rbr{q{v_0}-v_2}+\bi\rbr{v_1-sv_0},
\\
q=\oh(s^2+t^2)>\oh s^2.
\end{gather}

\section{Stability data associated to geometric stability conditions}
\label{sec:SD for GSC}

\subsection{Vanishing results}
\label{vanishing}

Let $X$ be a smooth projective surface and $\be,\om\in\NSR(X)$ with ample~$\om$.

\begin{lemma}\label{lm:tensoring stab}
For any line bundle $L=\cO_X(D)$, the automorphism
$$D^b(\Coh X)\to D^b(\Coh X),\qquad E\mto E(D)=E\ts L,$$
maps $(Z_{\be,\om},\cA_{\be,\om})$ to
$(Z_{\be+D,\om},\cA_{\be+D,\om})$.
\end{lemma}
\begin{proof}
We have $\mu_\om(E)\le\be\om$ if and only if
$\mu_\om(E(D))=\mu_\om(E)+D\om\le(\be+D)\om$.
This implies that $\cA_{\be,\om}$ is mapped to $\cA_{\be+D,\om}$.
If $a=\ch(E)$, then
$$Z_{\be+D,\om}(E(D))=(e^{\be+D+i\om},e^D\ch(E))
=(e^{\be+i\om},\ch(E))=Z_{\be,\om}(E).$$
Therefore semistability in $\cA_{\be,\om}$ \wrt $Z_{\be,\om}$ corresponds to 
semistability in $\cA_{\be+D,\om}$ \wrt $Z_{\be+D,\om}$.
\end{proof}

\begin{theorem}\label{th:vanish1}
Let $E,F$ be $\si_{\be,\om}$-semistable objects with $\vi_{\be,\om}(E)<\vi_{\be,\om}(F)$.
Then, for any nef line bundle ~$L$, we have 
$$\Hom(F\ts L,E)=0.$$
\end{theorem}
\begin{proof}
%We need to show that $\Hom(F,E\ts L\inv)=0$.
Let $L=\cO_X(H)$.
We will prove that $\vi_{\be-H,\om}^-(F)\ge\vi_{\be,\om}(F)$.
Then
$$\vi_{\be-H,\om}^-(F)\ge \vi_{\be,\om}(F)
>\vi_{\be,\om}(E)=\vi_{\be-H,\om}(E(-H)).$$
By Lemma \ref{lm:tensoring stab}, the object $E(-H)$ is semistable \wrt $\si_{\be-H,\om}$.
Therefore 
$$\Hom(F\ts L,E)\iso\Hom(F,E(-H))=0.$$

To prove the inequality $\vi_{\be-H,\om}^-(F)\ge\vi_{\be,\om}(F)$,
let us consider the family of stability conditions $\si_t=(Z_t,\cP_t):=\si_{\be -tH,\om}$ for $t\in[0,1]$ and let $\vi_t$ be the corresponding phase function.
Then the required inequality $\vi^-_1(F)\ge\vi_0(F)$ will follow from Theorem \ref{th:phase behave}, after we prove that
$$\Im(Z'_t(E)\cdot\bar Z_t(E))\ge0$$
for any $\si_t$-semistable object $E$.
It is enough to consider just $t=0$.
If $b=\ch^\be(E)$, then \eqref{Z3}
$$Z_0(E)=Z_{\be,\om}(E)
=\toh b_0\om^2-b_2+\bi b_1\om.$$
$$Z'_0(E)=\frac d{dt}Z_{\be-tH,\om}(E)|_{t=0}
=\frac d{dt}(e^{\bi\om-tH},b)|_{t=0}
=(e^{\bi\om},Hb)=-b_1H+\bi b_0(H\om).$$
Therefore
\begin{equation}\label{Z'Z1}
\Im(Z'_0(E)\cdot \bar Z_0(E))
=(b_1H)(b_1\om)-b_0b_2(H\om)+\toh b_0^2(H\om)\om^2.
\end{equation}
By Theorem \ref{th:BTI},
for any $\si_{\be,\om}$-semistable object $E$ with $\ch^\be(E)=b$, we have
$$(b_1H)(b_1\om)-b_0b_2(H\om)\ge0.$$
The summand $\toh b_0^2(H\om)\om^2$ is non-negative.
Therefore $\Im(Z'_0(E)\cdot\bar Z_0(E))\ge0$ as required.
\end{proof}

\begin{corollary}\label{cor:vanish geom}
Let $\si$ be a special geometric stability condition
and $E,F$ be $\si$-semistable objects with $\vi_\si(E)<\vi_\si(F)$.
Then, for any nef line bundle ~$L$, we have 
$$\Hom(F\ts L,E)=0.$$
\end{corollary}
\begin{proof}
By Theorem \ref{th:unique shift} there exists $g=(T,f)\in\wtl\GL_2^+(\bR)$ such that $\si[g]=\si'=\si_{\be,\om}$ for some $\be\in\NSR(X)$ and $\om\in\Amp(X)$.
Then $E,F$ are $\si'$-semistable and $\vi_{\si'}(E)<\vi_{\si'}(F)$.
We conclude from Theorem \ref{th:vanish1} that 
$\Hom(F\ts L,E)=0$ as required.
\end{proof}

\begin{theorem}\label{th:vanish2}
Let $E,F$ be $\si_{\be,\om}$-semistable objects with $\vi_{\be,\om}(E)=\vi_{\be,\om}(F)\notin\bZ$.
Then, for any ample line bundle ~$L$, we have 
$$\Hom(F\ts L,E)=0.$$
\end{theorem}
\begin{proof}
Let us use the same notation as in the previous theorem.
We can assume that $E,F\in\cA_{\be,\om}$ have phase $\vi\in(0,1)$.
In our analysis of the behavior of phases under deformations from the previous theorem, we note that
by Theorem \ref{th:BT2} (where $L=\cO_X(H)$ and $b=\ch^\be(E)$)
\begin{equation}
\Im(Z'_0(E)\cdot \bar Z_0(E))
=(b_1H)(b_1\om)-b_0b_2(H\om)+\toh b_0^2(H\om)\om^2
=\wtl\De_{\be,\om}^H(E)>0
\end{equation}
whenever $E$ is a $\si_{\be,\om}$-semistable object having phase in $(0,1)$.
By the same argument as in Theorem \ref{th:phase behave}, we conclude that $\phi_1^-(F)>\vi_0(F)$.
Therefore
$$\vi_{\be-H,\om}^-(F)> \vi_{\be,\om}(F)
=\vi_{\be,\om}(E)=\vi_{\be-H,\om}(E(-H)).$$
This implies that $\Hom(F\ts L,H)\iso\Hom(F,E(-H))=0$.
\end{proof}

\begin{remark}
Note that the sheaves $\cO_x$ are $\si_{\be,\om}$-semistable objects of phase $1$ while $\Hom(\cO_x\ts L,\cO_x)\ne0$ for any line bundle $L$.
On the other hand, the above theorem can be extended to semistable objects of phase $1$ having non-zero rank (see Lemma \ref{lm:tilt prop}).
\end{remark}

\begin{corollary}
Let $\si$ be a special geometric stability condition
and $E,F$ be $\si$-semistable objects with $\vi_\si(E)=\vi_\si(F)\notin\vi_\si(\cO_x)+\bZ$ for $x\in X$.
Then, for any ample line bundle ~$L$, we have 
$$\Hom(F\ts L,E)=0.$$
\end{corollary}

\begin{theorem}\label{th:vanish3}
Let $X$ be a surface with a nef anticanonical bundle.
Then any special geometric stability condition $\si$ on $D^b(\Coh X)$ satisfies
$$\gldim(\si)=2.$$
\end{theorem}
\begin{proof}
To prove that $\gldim(\si)\le2$, we need to show that given two $\si$-semistable objects $E,F$ with $\vi_\si(F)>\vi_\si(E)+2$, we have $\Hom(E,F)=0$.
By Serre duality,
this means that $\Hom(F\ts\om_X\inv,E[2])=0$, where $\om_X\inv$ is nef and $\vi_\si(E[2])<\vi_\si(F)$.
The last equation follows from Corollary \ref{cor:vanish geom}.
To see that $\gldim(\si)=2$, we note that for any $x\in X$ the object $\cO_x$ is \si-stable and $\Ext^2(\cO_x,\cO_x)\iso\Hom(\cO_x,\cO_x)\dual\ne0$.
\end{proof}

\begin{theorem}\label{th:del Pezzo Ext2}
Let $X$ be a del Pezzo surface (meaning that the anticanonical bundle is ample).
Let $E,F$ be $\si_{\be,\om}$-semistable objects with $\vi_{\be,\om}(E)=\vi_{\be,\om}(F)\notin\bZ$.
Then 
$$\Hom^k(E,F)=0\qquad \forall k\ge2.$$
\end{theorem}
\begin{proof}
We apply Serre duality and Theorem \ref{th:vanish2}.
\end{proof}

\begin{remark}
%Note that $\Ext^2(\cO_x,\cO_x)\iso\Hom(\cO_x,\cO_x)\dual\ne0$.
It was proved in \cite[Lemma 4]{li_smoothness} that $\Ext^2(E,E)=0$ for a $\si_{\be,\om}$-semistable object ~$E$ of nonzero rank under an implicit assumption that $-K_X$ is proportional to $\om$.
A similar method was used in \cite{fan_contractibility} to show that the statement of Theorem \ref{th:vanish3} is true for $X=\bP^2$.
This method doesn't seem to generalize to other surfaces with nef or ample anticanonical bundle.
\end{remark}

\subsection{Geometric stability data}
\label{sec:GSD}
In this section we will combine all previous results and obtain a continuous family of stability data associated to a surface.

Let $X$ be a smooth projective surface, $\cD=D^b(\Coh X)$
and $\Ga=\cN(X)$.
Let $\be\in\NSR(X)$, $\om\in\Amp(X)$ and $\si_{\be,\om}$ be the corresponding stability condition on $\cD$.
This stability condition satisfies assumptions \eqref{ass1} and \eqref{ass2} by the results of \cite{toda_moduli}
(see also \cite{piyaratne_moduli}).
We conclude from Theorem \ref{th:unique shift} that the same is true for any special geometric stability condition $\si\in\Stab^+_{\cl}(\cD)$ \eqref{sg stab}.

Let us assume now that $X$ has a nef anticanonical bundle.
Then we conclude from Theorem~\ref{th:vanish3} that 
assumption \eqref{ass3} is satisfied for any
special geometric stability condition
$\si\in\Stab^+_{\cl}(\cD)$.
Following \S\ref{sec:wc qt} and \S\ref{sec:sd for sc}, we consider the quantum torus $\fg=\bop_{\ga\in\Ga}Rx^\ga$, $R=\R(\Sta/\bC)$,
and define,
for any special geometric stability condition $\si=(Z_\si,\cP_\si)\in\Stab^+_{\cl}(X)$,
the corresponding stability data
$a_\si=(a_{\si,\ga})_{\ga\in\Ga\ms\set0}\in\hat\fg$ by
\begin{gather}
a_{\si,\ell}=\sum_{Z_\si(\ga)\in\ell}a_{\si,\ga}=\log(A_{\si,\ell})\in\hat\fg_{\ell,Z_\si},\\
A_{\si,\ell}
%=\cI(\one_{\si,\ell})
=1+\sum_{Z_\si(\ga)\in\ell}
\bL^{\oh\hi(\ga,\ga)}[\cM_\si(\ga)]x^\ga
\in 
%1+\hat \bT_{S_\ell}^+\iso 
\hat G_{\ell,Z_\si}, \label{Asigma}
\end{gather}
for any ray $\ell\sbs\bC$,
where $\cM_\si(\ga)$ is the moduli stack of \si-semistable objects having class \ga,
$\hat\fg_{\ell,Z_\si}$ is the ind-pro-nilpotent Lie algebra defined in \eqref{ind-pro Lie alg} and 
$\hat G_{\ell,Z_\si}$ is the corresponding ind-pro-nilpotent group.
We obtain from Theorem \ref{th:fam sd from sc}

\begin{theorem}
For any smooth projective surface $X$ with the nef anticanonical bundle, the family of stability data $(Z_\si,a_\si)_{\si\in\Stab_{\cl}^+(\cD)}$ is continuous.
\end{theorem}

\subsection{Relation to intersection cohomology}
\label{sec:inters}
We briefly explain the relationship of the above invariants to the intersection cohomology of the moduli spaces of semistable objects.
Let $X$ be a del Pezzo surface, meaning that the anticanonical bundle is ample.
By the results of \cite{toda_moduli,piyaratne_moduli},
for any $\si\in\Stab_{\cl}^+(\cD)$ and $\ga\in\Ga$,
we can construct the moduli space $M_\si(\ga)$ of \si-semistable objects having class \ga (and phase $\frac1\pi\Arg Z_\si(\ga)$)
and the moduli space $M_\si^\st(\ga)$ of \si-stable objects having class \ga.
The moduli space $M_\si^\st(\ga)$ is smooth if $\Hom^2(E,E)=0$ for all $E\in M_\si^\st(\ga)$ (or more generally, if $\Hom^2(E,E)$ has constant dimension).

Given an algebraic variety $Y$ of dimension $d_Y$, let $\lIC_Y\bQ^H\in\MHM(Y)$ be the corresponding intersection complex (see \eg \cite{saito_introduction}).
It is a simple pure object of weight $d_Y$.
If $Y$ is smooth, then $\lIC_Y\bQ^H=\bQ_Y^H[d_Y]$.
We define (\cf \cite{mozgovoy_intersection})
\begin{equation}
\IH^*(Y,\bQ)=H^*_c(Y,\lIC_Y\bQ^H)[-d_Y]\in D^b(\MHS_\bQ).
\end{equation}
Let us consider the Hodge-Deligne polynomial (\cf \eqref{E-pol})
\begin{equation}
E:K(D^b(\MHS_\bQ))\iso K(\MHS_\bQ)\to\bZ[u^{\pm1},v^{\pm1}],\qquad 
V\mto
\sum_{p,q}\dim \rbr{\Gr^p_F\Gr^W_{p+q}V_\bC}u^pv^q.
\end{equation}
and define
\begin{equation}
\IE(Y;u,v)=E(\IH^*(Y,\bQ))\in \bZ[u^{\pm1},v^{\pm1}].
\end{equation}
Let us consider the closure $\bar M_\si^\st(\ga)$
of $M_\si^\st(\ga)\sbs M_\si(\ga)$ and define
(recall that we defined earlier $E(\bL^\oh)=-\sqrt{uv}$)
\begin{equation}
\Om_\si(\ga)=E(\bL^\oh)^{\hi(\ga,\ga)-1}\cdot
\IE(\bar M_\si^\st(\ga);u,v).
\end{equation}
On the other hand, given a ray $\ell\sbs\bC$, we consider (\cf \eqref{Asigma})
\begin{equation}
E(A_{\si,\ell})=1+\sum_{Z_\si(\ga)\in\ell}
E(\bL^\oh)^{\hi(\ga,\ga)}\cdot E(\cM_\si(\ga))x^\ga.
\end{equation}

Let us assume that a ray $\ell\sbs\bC$ is such that $Z_\si(\cO_x)\notin\pm\ell$ for $x\in X$ and
\begin{equation}
\hi(\ga,\ga')=\hi(\ga',\ga)\qquad
\forall \ga,\ga'\in\Ga\cap Z_\si\inv(\ell).
\end{equation}
By Theorem \ref{th:del Pezzo Ext2}, the first assumption implies that for any \si-semistable objects $E,F$ having the same phase and classes in $Z_\si\inv(\ell)$, we have $\Hom^2(E,F)=0$.
The second assumption implies that the quantum torus restricted to degrees in $\Ga\cap Z_\si\inv(\ell)$ is commutative.
By the results of \cite{meinhardt_donaldson,meinhardt_donaldsona,mozgovoy_intersectiona,mozgovoy_intersection},
we have
\begin{equation}
E(A_{\si,\ell})=
\Exp\rbr{\frac{\sum_{Z_\si(\ga)\in\ell}\Om_\si(\ga)x^\ga}{E(\bL^\oh)-E(\bL^{-\oh})}},
\end{equation}
where $\Exp$ is the plethystic exponential (see \eg \cite{mozgovoy_intersection}).

\begin{remark}
In \cite{bousseau_scattering} one used the last formula as the definition of the invariants $E(A_{\si,\ell})$ in the case of $X=\bP^2$.
Note that this formula can not be used if $Z_\si(\cO_x)\in\ell$.
\end{remark}
\section{Relation to quivers with potentials}
\label{sec:rel QP}

\subsection{Exceptional collections and tilting objects}
\label{sec:except}
For more information on tilting theory see \eg \cite{angelerihugel_handbook}.
Let $X$ be a smooth projective variety of dimension ~$d$ and let
\begin{equation}\label{exc col}
E_1,\,\dots,\,E_n
\end{equation}
be an \idef{exceptional collection} in $D^b(\Coh X)$, meaning that
\begin{enumerate}
\item $\Hom^k(E_i,E_i)=0$ for $k\ne0$ and $\Hom(E_i,E_i)=\bC$.
\item $\Hom^k(E_i,E_j)=0$ for $i>j$ and $k\in\bZ$.
\end{enumerate}
We will say that an exceptional collection is \idef{full} if the objects $E_i$ generate $D^b(\Coh X)$.
We will say that an exceptional collection is \idef{strong}
if $\Hom^k(E_i,E_j)=0$ for all $i,j$ and $k\ne0$.
Given a full and strong exceptional collection,
let us define the algebra
\begin{equation}\label{alg A1}
A=\End(T')^\op,\qquad T'=\bop_i E_i,
\end{equation}
and let $\mmod A$ denote the category of finite-dimensional left $A$-modules.
Then there is an equivalence of categories
\cite{bondal_representations,rickard_morita,angelerihugel_handbook}
\begin{equation}
\Phi:D^b(\Coh X)\to D^b(\mmod A),\qquad E\mto\RHom(T',E).
\end{equation}
The objects $E_i$ are mapped to the projective modules $P_i=Ae_i$, where $e_i\in A$ is the idempotent corresponding to $\Id\in\End(E_i)$.
Let us introduce a $\bZ$-grading $A=\bop_{k\ge0} A_k$, $A_k=\bop_{j-i=k}\Hom(E_i,E_j)$, and let 
$\bar A=\Ker(A\to A_0=\bop_{i=1}^n\bC)$ be the augmentation ideal.
By \cite[Lemma A.1]{bridgeland_helices}, one can represent $A$ as a quotient $\bC Q'/I$ for a unique quiver $Q'$ with the set of vertices $Q'_0=\set{1,\dots,n}$ and the number of arrows from $i$ to $j$ equal to $\dim e_j(\bar A/\bar A^2)e_i$.

On the other hand, let us consider the canonical bundle
$Y=\om_X$ as an algebraic variety and let $\pi:Y\to X$ be the projection.
Note that $Y$ is a non-compact Calabi-Yau variety.
We have $\pi_*\cO_Y=\bop_{i\ge0}\om_X^{-i}$ and $Y=\Spec(\pi_*\cO_Y)$.
Let $\Coh_c Y\sbs \Coh Y$ be the subcategory of coherent sheaves with compact support and let $\Coh_0 Y\sbs \Coh_c Y$ be the subcategory of coherent sheaves supported on $X\sbs Y$ embedded as the zero section.
A coherent sheaf $\bar E\in\Coh_c Y$ can be identified with a pair $(E,\vi)$, where $E=\pi_*\bar E\in\Coh X$ and $\vi:E\to E\ts\om_X$ is a morphism.
We have $\bar E\in\Coh_0 Y$ if and only if $\vi:E\to E\ts\om_X$ is nilpotent, meaning that 
$E\to E\ts\om_X\to E\ts\om_X^{2}\to\dots$ eventually vanishes.

The category $D^b(\Coh_c X)$ can be identified with the subcategory $D^b_c(\Coh Y)\sbs D^b(\Coh Y)$ consisting of complexes with cohomology having compact support.
Similarly, the category $D^b(\Coh_0 X)$ can be identified with the subcategory $D^b_0(\Coh Y)\sbs D^b(\Coh Y)$ consisting of complexes with cohomology supported on $X$.
Note that the full embedding $\Coh X\emb\Coh_0Y$, $E\mto (E,0)$, induces a (non-full) functor $D^b(\Coh X)\to D^b_0(\Coh Y)$.
The category $D^b(\Coh X)$ generates $D^b_0(\Coh Y)$ by extensions.

\begin{lemma}[See \eg \cite{mozgovoy_quiver}]
Let $\bar E=(E,\vi)$, $\bar F=(F,\psi)$ be objects in $\Coh_cY$.
Then there is an exact sequence
\begin{multline*}
0\to\Hom_Y(\bar E,\bar F)\to\Hom_X(E,F)\to\Hom_X(E,F\ts\om_X)\\
\to\Ext^1_Y(\bar E,\bar F)\to\Ext^1_X(E,F)\to
\Ext^1_X(E,F\ts\om_X)\to\dots
\end{multline*}
\end{lemma}

\begin{corollary}\label{cor:loc surf Euler}
Let $\bar E=(E,\vi)$, $\bar F=(F,\psi)$ be objects in $\Coh_cY$.
Then
$$\hi_Y(\bar E,\bar F)=\hi_X(E,F)-(-1)^d\hi_X(F,E)=:\ang{E,F}.$$
\end{corollary}

Let us consider the object $T=\pi^*T'=\bop_i\pi^*E_i$ in $D^b(\Coh Y)$.
We have
\begin{equation}
\Hom^k(T,T)
%=\Hom^k(T,\pi_*\pi^*T)
=\Hom^k(T',T'\ts\pi_*\cO_Y)
%=\bop_{m\ge0}\Hom^k(T,T\ts\om_X^{-m}). 
=\bop_{\ov{1\le i, j\le n}{m\ge0}}\Hom^k(E_i,E_j\ts\om_X^{-m}). 
\end{equation}
Let us assume that
\begin{equation}
\Hom^k(T,T)=0\qquad\forall k\ne0.
\end{equation}
The object $T$ generates $D^b(\Coh Y)$ in the sense that 
$\Hom_Y^\bul(T,F)=0$ for $F\in D^b(\Coh Y)$ implies $F=0$.
Indeed, $\Hom_Y^\bul(T,F)=\Hom_X^\bul(T',\pi_*F)=0$ implies $\pi_*F$, hence $F=0$.
This means that $T$ is a \idef{classical tilting object},
hence there is an equivalence of categories \cite[\S7]{hille_fourier} (see also \cite[Theorem 3.6]{bridgeland_helices})
\begin{equation}
\Phi:D^b_c(\Coh Y)\to D^b(\mmod B),\qquad E\mto\RHom(T,E),\qquad B=\End(T)^\op.
\end{equation}
Let us consider the ideal $\cI=\bop_{m\ge1}\Hom(T',T'\ts\om_X^{-m})$ of $B$ and let $\mmod_0B\sbs\mmod B$ consist of modules on which $\cI$ acts nilpotently.
Then the above functor induces an equivalence
\begin{equation}
\Phi:D^b_0(\Coh Y)\to D^b_0(\mmod B)=D^b(\mmod_0 B),\qquad E\mto\RHom(T,E).
\end{equation}

Let us extend the above exceptional collection to a sequence $(E_i)_{i\in \bZ}$ such that 
\begin{equation}
E_{i+n}=E_i\ts\om_X\inv,\qquad i\in\bZ.
\end{equation}
This is an example of a helix \cite{bridgeland_helices}.
By our assumptions $\Hom^k(E_i,E_j)=0$ for $j>i-n$
and $k\ne0$.
On the other hand, $\Hom(E_i,E_j)\iso\Hom^d(E_{j+n},E_i)\dual=0$ for $i>j$.
Therefore
\begin{equation}
B^\op\iso \bop_{1\le i\le n,j\ge i}\Hom(E_i,E_j).
\end{equation}

%One proves in \cite{bridgeland_helices} that in the cases of most interest one has $\Hom(E_i,E_j)=0$
%(wrong, $a_{ij}=0$) for $j>i+n$.
%The algebra $B$ is not!!! finite-dimensional.

%Assume that we have a geometric helix as above and let $E=E_0\oplus\dots\oplus E_{n-1}$.
%Then $\pi_*\pi^*E=E\ts\pi_*\cO_Y=\bop_{i\ge0}E_i$.
%Therefore
%$$\Hom^k(\pi^*E,\pi^*E)=\Hom^k(E,\pi_*\pi^*E)
%=\bop_{0\le i<n,j\ge0}\Hom^k(E_i,E_j)=0
%$$
%for $k\ne0$.
%
%We will say that this sequence is a \idef{helix} if $(E_{i+1},\dots,E_{i+n})$ is a full exceptional collection for all $i\in\bZ$.
%We will say that a helix is \idef{geometric} if
%\begin{equation}
%\Hom^k(E_i,E_j)=0\qquad\forall k\ne0,\,i<j.
%\end{equation}
%In this case the exceptional collection
%$(E_{i+1},\dots,E_{i+n})$ is strong for all $i\in\bZ$.
%Note that if $0<i-j<n$, then $\Hom(E_i,E_j)=0$ by the requirement that $(E_j,\dots,E_{j+n-1})$ is an exceptional collection.
%On the other hand, if $i\ge j+n$, then $\Hom(E_i,E_j)\iso \Hom^{d}(E_{j+n},E_{i})\dual=0$.
%We conclude that $\Hom(E_i,E_j)=0$ for $i>j$.

\subsection{Local surfaces}
Let $X$ be a smooth projective surface, $Y=\om_X$ be its canonical bundle and $\pi:Y\to X$ be the projection as before.
Then $Y$ is a non-compact 3-Calabi-Yau variety which is sometimes called a \idef{local surface}.
%Note that $Y=\Spec(\bop_{n\ge0}\om_X^{-n})$,
%hence we can identify an object in $\Coh_0Y$ with a pair $(E,\vi)$, where $E\in\Coh X$ and $\vi:E\to E\ts\om_X$ is nilpotent, meaning that 
%$E\to E\ts\om_X\to E\ts\om_X^{2}\to\dots$ eventually vanishes.
Recall that
\begin{gather}
\hi_X(E,F)=\int a^*b\cdot\Td(X),\qquad a=\ch(E),\,b=\ch(F),
\label{hi EF}\\
\Td(X)=(1,-\toh K_X,\hi),\qquad \hi=\hi(\cO_X).
\end{gather}
We have $a^*a=(a_0^2,0,2a_0a_2-a_1^2)$, hence
\begin{equation}\label{hi EE}
\hi_X(E,E)=\hi\cdot a_0^2-\De(E).
\end{equation}
We have $a^*b-b^*a=2(0,a_0b_1-a_1b_0,0)$, hence
\begin{equation}\label{ang EF}
\ang{E,F}
%=\hi(E,F)-\hi(F,E)
=(a_1b_0-a_0b_1)K_X
%$$=-((0,a_0K_X,a_1K_X),b)=-(aK_X,b)
=(a,bK_X).
\end{equation}

For many surfaces $X$ (see \eg \cite{beaujard_vafa})
one can find an exceptional collection \eqref{exc col}
such that $T'=\bop_iE_i$ and $T=\pi^*T'$ satisfy the assumptions of \S\ref{sec:except},
the algebra $B=\End(T)^\op$ can be identified with the Jacobian algebra $J_{W}$ of some quiver with potential $(Q,W)$ (see~ \S\ref{sec:Jacobian})
and the algebra $A=\End(T')^\op$ can be identified with the partial Jacobian algebra $J_{W,I}$ corresponding to a cut $I\sbs Q_1$ (see~ \S\ref{sec:Jacobian}).
This implies that we can construct three, a priori unrelated, families of stability data:
\begin{enumerate}
\item For the category $D^b(\Coh X)\iso D^b(\mmod J_{W,I})$.
\item For the category $D^b_0(\Coh Y)\iso D^b_0(\mmod J_{W})$.
\item For the category $D^b_c(\Coh Y)\iso D^b(\mmod J_{W})$.
\end{enumerate}
Because of Corollary \ref{cor:loc surf Euler}, all these families are defined in the same quantum torus.
We will show later that for $X=\bP^2$ stability data on $D^b(\Coh X)$ and on $D^b(\mmod J_W)$ can be glued together.
This implies that knowing stability data on $D^b(\mmod J_W)$, we can reconstruct stability data on $D^b(\Coh X)$.
A complete (but still conjectural except for low dimension vectors) description of stability data on $D^b(\mmod J_W)$
was given in \cite{beaujard_vafa,mozgovoy_attractor}.
A partial description of stability data on
$D^b(\Coh X)$ was given in \cite{bousseau_scattering} (the scattering diagram used there captures a significant, but incomplete information about the stability data).
It should be not difficult to generalize our results to other del Pezzo surfaces.
A complete (conjectural) description of stability data for $D^b(\mmod J_W)$ is also available in this case
\cite{beaujard_vafa,mozgovoy_attractor}.

Recall that for every $\be\in\NSR(X)$ and $\om\in\Amp(X)$ we have a stability condition $\si_{\be,\om}$ on $D^b(\Coh X)$.
This stability condition can be extended to $D^b_0(\Coh Y)$ (see \eg ~\cite{bayer_spacea}).
Alternatively, one can use the approach of \cite{ikeda_q} to extend a stability condition from
$D^b(\Coh X)$ to $D^b_0(\Coh Y)$.
Motivic DT invariants can be defined for the 3CY categories $D^b_0(\Coh Y)$ and $D^b_c(\Coh Y)$ using the framework of \cite{kontsevich_stability}, as long as one proves that various necessary assumptions are satisfied (and this is not a simple task).
On the other hand, as was observed in \cite{bousseau_scattering} (see also \cite[Prop.~3.1]{cao_genus}), in many situations (assuming that ~$X$ is a del Pezzo surface) semistable objects in $D^b_c(\Coh Y)$ are automatically semistable objects in $D^b(\Coh X)$.
Intuitively, the reason is that for any semistable object $(E,\vi)$, there is a morphism $\vi:(E,\vi)\to(E\ts\om_X,\vi\ts\id_{\om_X})$ which has to vanish if the object on the right is semistable of phase less than the phase of $(E,\vi)$.
This means that for del Pezzo surfaces one can perform computations in $D^b(\Coh X)$ instead of the 3CY categories $D^b_0(\Coh Y)$ and $D^b_c(\Coh Y)$ (at least for some stability conditions and some Chern characters).
%Note that the quantum tori in both cases are the same because of Corollary \ref{cor:loc surf Euler}.

In the next sections we will consider the case of $X=\bP^2$ and we will see how computation of stability data for $D^b(\Coh X)$ reduces to computations of stability data for a certain quiver with potential.
%The derived category of the corresponding (non-complete) Ginzburg DG algebra is a 3CY category equivalent to $D^b_c(\Coh Y)$, hence we obtain an alternative relationship between DT invariants of $D^b(\Coh X)$ and $D^b_c(\Coh Y)$.
%Our approach can be generalized to other del Pezzo surfaces (\cf \cite{beaujard_vafa,mozgovoy_attractor}).

\subsection{Relation to quiver representations}
\label{sec:rel QP plane}
In this section we will identify the derived category $\cD=D^b(\Coh X)$ for $X=\bP^2$ with the derived category of some explicit finite-dimensional algebra $A$ of homological dimension two.
We will see later that this algebra arises from a quiver with a potential and a cut.

\subsubsection{Derived equivalence}
Let us consider the Beilinson exceptional collection on $X=\bP^2$
\begin{equation}\label{exc seq}
\cO,\,\cO(1),\,\cO(2).
\end{equation}
As in \S\ref{sec:except}, we define the algebra
\begin{equation}
A=\End(T')^\op,\qquad T'=\cO\oplus\cO(1)\oplus\cO(2),
\end{equation}
%and let $\mmod A$ denote the category of finite-dimensional left $A$-modules.
%Then there is 
and consider an equivalence of categories
\begin{equation}\label{Phi2}
\Phi:D^b(\Coh X)\to D^b(\mmod A),\qquad E\mto\RHom(T',E).
\end{equation}
The algebra $A$ can be described as the quotient of the path algebra $\bC Q'$ for the quiver $Q'$
\begin{equation}\label{Q1}
\begin{tikzcd}
0&1\lar[->>>,"a_i"']&2\lar[->>>,"b_i"']
\end{tikzcd}
\end{equation}
by the relations
\begin{equation}\label{A-rel}
a_ib_j=a_jb_i,\qquad i\ne j.
\end{equation}
Let $e_i\in A$ be the idempotent corresponding to the trivial path at the vertex $i\in Q'_0$.
Let $S_i\in\mmod A$ denote the corresponding simple module, $P_i=Ae_i\in\mmod A$ denote the projective indecomposable module, and $I_i=D(e_iA)\in\mmod A$ denote the injective indecomposable module, where $DV=\Hom_\bC(V,\bC)$ for a vector space $V$.
Note that $S_0=P_0$ and $S_2=I_2$.

The objects of the exceptional sequence \eqref{exc seq} are mapped to the projective modules $P_0$, $P_1$, $P_2$ respectively.
Note that
\begin{gather*}
\Hom(\cO,\cO(1))\iso\bC^3\iso\Hom(P_0,P_1)=e_0Ae_1,\\
\Hom(\cO(1),\cO(2))\iso\bC^3\iso\Hom(P_1,P_2)=e_1Ae_2,\\
\Hom(\cO,\cO(2))\iso\bC^6\iso\Hom(P_0,P_2)=e_0Ae_2.
\end{gather*}

\subsubsection{Simple objects}
\label{sec:simple}
From now on we will identify $\cD=D^b(\Coh X)$ with $D^b(\mmod A)$ using the equivalence $\Phi$ \eqref{Phi2}.
We claim that the simple objects $S_0$, $S_1$, $S_2$ correspond to 
%(\cf \cite[\S3.1]{beaujard_vafa})
\begin{equation}\label{simple}
\cO,\, \Om_X(1)[1],\,\cO(-1)[2],
\end{equation}
where $\Om_X$ is the cotangent bundle.
It is clear that $S_0=P_0=\cO$.
Note that 
\begin{equation}
\dim\Hom^k(P_i,S_j)=\de_{ij}\de_{k0}.
\end{equation}
It is proved in \cite[Lemma 2.5]{bridgeland_helices} that for a full exceptional collection $(E_1,\dots,E_n)$ there exists a unique collection $(F_n,\dots,F_1)$ 
such that $\dim\Hom^k(E_i,F_j)=\de_{ij}\de_{k0}$
(the new collection is also a full exceptional collection).
Therefore it is enough to show that \eqref{simple} has the required properties.
Below we will give a direct proof instead. 

%For $i\in Q_0$, let $P_i=Ae_i$ be the corresponding indecomposable projective module and $S_i$ be the corresponding simple module.

The category $D^b(\Coh X)$ has the Serre functor (note that $\om_X=\cO(-3)$)
\begin{equation}
\cS:D^b(\Coh X)\to D^b(\Coh X),\qquad E\mto E\ts\om_X[2].
\end{equation}
On the other hand, the category $D^b(\mmod A)$ also has the Serre functor (called the Nakayama functor in this context; the same construction works for any finite-dimensional algebra $A$ of finite homological dimension)
\begin{equation}
\nu:D^b(\mmod A)\to D^b(\mmod A),\qquad M\mto DA\ts^L_A M,
\end{equation}
where we consider $DA=\Hom_\bC(A,\bC)$ as an $A$-bimodule.
Note that $\nu P_i=I_i$.
We can identify the Serre functors $\cS$ and $\nu$, hence
\begin{equation}
I_0=\cO(-3)[2],\qquad I_1=\cO(-2)[2],\qquad S_2=I_2=\cO(-1)[2].
\end{equation}
Let us consider the universal exact sequence on $X=\bP^2$
\cite[B.5.7]{fulton_intersection}
\begin{equation}\label{Qseq}
0\to \cO(-1)\to \cO^{\oplus 3}\to \cQ\to0.
\end{equation}
Then the tangent bundle is given by $T_X=\cQ(1)$.
Dualizing the above exact sequence, we obtain a distinguished triangle
\begin{equation}
\cO^{\oplus 3}\to\cO(1)\to \cQ\dual[1]\to
\end{equation}
which corresponds to the canonical short exact sequence in $\mmod A$
\begin{equation}
0\to S_0^{\oplus 3}\to P_1\to S_1\to 0.
\end{equation}
We conclude that $S_1=\cQ\dual[1]=\Om_X(1)[1]$.

\subsubsection{Chern classes and dimension vectors}
The above equivalence 
$$\Phi:\cD=D^b(\Coh X)\to D^b(\mmod A)$$
induces an isomorphism between the (numerical) Grothendieck groups.
We have linear maps $\ch:\cN(\Coh X)\to N^*_\bQ(X)$ and
\begin{equation}
\udim:\cN(\mmod A)\to\bZ Q'_0,\qquad M\mto(\dim e_iM)_{i\in Q'_0}.
\end{equation} 
Given $E\in\cD$, let us write
$\ch(E)\in N^*_\bQ(X)$ and $\udim E:=\udim \Phi(E)\in\bZ Q'_0$
as column vectors and let $M$ be the transition matrix such that $\ch E=M\cdot\udim E$.
Let us also define matrices 
\begin{equation}
A=(\hi(P_i,P_j))_{ij},\qquad
B=(\hi(S_i,S_j))_{ij},\qquad C=(\ch P_i)_i.
\end{equation}
%We have A*M_{S\to P}=1, M_{P\to S}^t*B=1, hence A*B^t=1
% M=M_{dim\to ch}=M_{P\to ch}*M_{S\to P}
Then
\begin{equation}
A\cdot B^t=I,\qquad M=C\cdot A\inv.
\end{equation}
We can calculate $A$ and $C$ using \eqref{exc seq}
and we obtain
\begin{equation}
A=\smat{
1 & 3 & 6 \\
0 & 1 & 3 \\
0 & 0 & 1}
\qquad
B=\smat{
1 & 0 & 0 \\
-3 & 1 & 0 \\
3 & -3 & 1}
\qquad
C=\smat{1 & 1 & 1 \\
0 & 1 & 2 \\
0 & \frac12 & 2}
\qquad
M=\smat{
1 & -2 & 1 \\
0 & 1 & -1 \\
0 & \frac{1}{2} & \frac{1}{2}}
\end{equation}

In particular, if $a=(a_0,a_1,a_2)=\ch(E)$ and 
$d=(d_0,d_1,d_2)=\udim E$, then
\begin{equation}\label{a-d}
a_0=d_0-2d_1+d_2,\qquad a_1=d_1-d_2,\qquad 2a_2=d_1+d_2.
\end{equation}

\begin{remark}
The above formula also follows from the fact that the Chern classes of simple objects ~\eqref{simple} are
\begin{equation}\label{simple chern}
\ch\cO=(1,0,0),\qquad \ch\Om_X(1)[1]=(-2,1,\toh),\qquad \ch\cO(-1)[2]=(1,-1,\toh).
\end{equation}
%Then \eqref{simple chern} implies
\end{remark}

Let us consider the quiver $Q$ (\cf \eqref{Q1})
\begin{equation}
\begin{tikzcd}\label{hatQ}
&1\ar[dl,->>>,"a_i"']\\
0\ar[rr,->>>,"c_i"]&&2\ar[ul,->>>,"b_i"']
\end{tikzcd}
\end{equation}
and the corresponding Euler form
\begin{equation}
\hi_Q(d,d')=\sum_i d_i(d'_i-3d'_{i-1}).
\end{equation}
%\begin{equation}
%\hat\hi(d,d')=\sum_i d_id'_i-3(d_1d'_0+d_2d'_1+d_0d'_2).
%\end{equation}
Then the above explicit formula for the matrix $B$ implies

\begin{lemma}\label{lm:forms1}
Let $E,F\in\cD$ and $d=\udim E$, $d'=\udim F$. Then
\begin{equation}
\hi_X(E,F)
%=d_0e_0 - 3d_1e_0 + 3d_2e_0 + d_1e_1 - 3d_2e_1 + d_2e_2
=\hi_Q(d,d')+3d_2d'_0+3d_0d'_2.
\end{equation}
\end{lemma}
%\begin{proof}
%Let $a=\ch(E)$ and $b=\ch(F)$.
%We have $\hi(\cO_X,\cO_X)=1$, hence \eqref{hi EF}
%%$\Td(X)=(1,3/2,1)$ and %\eqref{hi EE}
%$$\hi_X(E,F)=\int a^*b\cdot(1,\tfrac32,1)
%=a_0b_0+\tfrac32(a_0b_1-a_1b_0)+(a_0b_2+a_2b_0-a_1b_1).
%$$
%Now we apply \eqref{a-d}.
%%$$\hi(E,E)=a_0^2-\De(E)=a_0(a_0+2a_2)-a_1^2.$$
%%Now we just apply \eqref{a-d}.
%%Similarly, we have by \eqref{ang EF}
%%\begin{multline*}
%%\ang{E,F}=(a_1b_0-a_0b_1)K_X=3(a_0b_1-b_0a_1)\\
%%=3(d_0-2d_1+d_2)(d'_1-d'_2)-3(d'_0-2d'_1+d'_2)(d_1-d_2)\\
%%%=3(d_0-d_1)(d'_1-d'_2)-3(d'_0-d'_1)(d_1-d_2)\\
%%=3(d_0d_1'+d_1d_2'+d_2d_0')-3(d_1d_0'+d_2d_1+d_0d_2')
%%=\hi_Q(d,d')-\hi_Q(d',d).
%%\end{multline*}
%\end{proof}

In particular, we have
\begin{gather}
\hi_X(E,E)=\hi_Q(d,d)+6d_0d_2,\\
\ang{E,F}=\hi_Q(d,d')-\hi_Q(d',d)=:\ang{d,d'}.\label{24}
\end{gather}

\subsection{Quivers with potential and their motivic invariants}
\label{sec:Jacobian}
The quiver $Q$ defined in \eqref{hatQ} is the McKay quiver of the representation of $\bZ_3$ on $\bC^3$ given by $1\mto\diag(\ksi,\ksi,\ksi)$, $\ksi=e^{2\ip/3}$.
This quiver is automatically equipped with the potential \cite[\S4.4]{ginzburg_calabi}
\begin{equation}\label{W1}
W=\sum_{\si\in\Si_3}\sgn(\si)\cdot a_{\si1}b_{\si2}c_{\si3},
\end{equation}
where $\Si_3$ is the permutation group.

Generally, let us assume that we have a quiver $Q$ equipped with a \idef{potential} $W\in\bC Q/[\bC Q,\bC Q]$, that is, a finite linear combination $W=\sum_u c_u u$ of cyclic paths considered up to cyclic shift.
We define the \idef{Jacobian algebra} of $(Q,W)$ to be
\begin{equation}
J_W=\bC Q/(\dd W)=\bC Q/(\dd W/\dd a\col a\in Q_1),
\end{equation}
where $\dd u/\dd a=\sum_{i:a_i=a}a_{i-1}\dots a_1a_n\dots a_{i+1}$ for any cycle $u=a_n\dots a_1$.
Let $I\sbs Q_1$ be a cut, meaning a subset such that every non-zero term of $W$ contains exactly one arrow from $I$.
We consider the quiver $Q'=(Q_0,Q_1\ms I)$ and define
the \idef{partial Jacobian algebra}
\begin{equation}\label{JWI}
J_{W,I}=\bC Q'/(\dd W/\dd a\col a\in I).
\end{equation}
For any dimension vector $d\in\bN Q_0$, we define 
\begin{equation}
R(Q,d)=\bop_{a:i\to j}\Hom(\bC^{d_i},\bC^{d_j})
\end{equation}
and we define $R(J_W,d)\sbs R(Q,d)$ and $R(J_{W,I},d)\sbs R(Q',d)$ to be the closed subsets corresponding to the relations of the algebras $J_W$ and $J_{W,I}$ respectively.
We define 
%(\cf \cite[\S5.3]{mozgovoy_translation})
\begin{equation}\label{27}
\tr W:R(Q,d)\to\bC,\qquad M\mto \tr W|M=\sum_u c_u\cdot \tr u|M,
\end{equation}
where $\tr u|M$ for a cycle $u=a_n\dots a_1$ is defined to be $\tr(M_{a_n}\dots M_{a_1})$.
Then $R(J_W,d)$ can be identified with the critical locus of $\tr W$ (see \eg \cite{segal_a}).
Let us consider the skew-symmetric form on $\Ga=\bZ Q_0$
\begin{equation}
\ang{d,d'}=\hi_Q(d,d')-\hi_Q(d',d)
\end{equation}
and the corresponding quantum torus $\bT$ \eqref{T} (with coefficients in the Grothendieck ring of varieties with exponentials,
see \eg \cite[\S5]{mozgovoy_translation}).
We define the series
\begin{equation}\label{A0QW}
A_0^{Q,W}=\sum_{d\in\bN Q_0}\bL^{\oh\hi_Q(d,d)}\frac{[R(Q,d),\tr W]}{[\GL_d]}x^d\in\hat \bT,
\end{equation}
where $[R(Q,d),\tr W]$ is the exponential motivic class (see \eg \cite[\S5]{mozgovoy_translation}) and $\GL_d=\prod_i\GL_{d_i}$.
This expression can be simplified in the presence of a cut
$I\sbs Q_1$,
%Note that
%\begin{equation}
%\tr W:R(Q,d)=R(Q',d)\xx R(I,d)\to\bC
%\end{equation}
%is linear on $R(I,d)$, hence induces a map
namely \cite[\S5]{mozgovoy_translation}
\begin{gather}
[R(Q,d),\tr W]
=\bL^{\ga_I(d,d)}\cdot [R(J_{W,I},d)],\\
\ga_I(d,d')=\sum_{(a:i\to j)\in I}d_id'_j.
\label{ga-I}
\end{gather}

Let us assume now that we have a stability function $Z:\Ga=\bZ Q_0\to\bC$, meaning that $Z(e_i)\in\bH$ for $i\in Q_0$.
It defines the notion of stability on the abelian category $\Rep Q$ \S\ref{sec:stab ab}.
Let $R_Z(Q,d)\sbs R(Q,d)$ denote the open subset of $Z$-semistable representations.
For any ray $\ell\sbs\bH$, we define
\begin{equation}\label{A-Zl-QW}
A_{Z,\ell}^{Q,W}
=1+\sum_{Z(d)\in\ell}\bL^{\oh\hi_Q(d,d)}\frac{[R_Z(Q,d),\tr W]}{[\GL_d]}x^d\in\hat \bT
\end{equation}
and $A_{Z,-\ell}^{Q,W}=A_{Z,\ell}^{Q,W}$.
We define \idef{stability data of the quiver with potential} $(Q,W)$ to be the collection $A_{Z}^{Q,W}=(A_{Z,\ell}^{Q,W})_\ell$.
It is parametrized by $Z\in\bH^{Q_0}\sbs\Hom(\bZ Q_0,\bC)$, although one can extend the set of parameters to 
$\GL_2^+(\bR)\cdot \bH^{Q_0}\sbs\Hom(\bZ Q_0,\bC)$.

The following result is well-known.
It is important to stress that while we compute invariants $[R(Q,d),\tr W]$ which morally correspond to counting representations of the Jacobian algebra (and these representations form the heart of a bounded t-structure on the 3CY derived category of DG modules over the Ginzburg DG algebra of $(Q,W)$ ~\cite{keller_deriveda}), in reality all computations are performed on the level of $Q$-representations which form an abelian category of homological dimension one.
Therefore the proof of the below wall-crossing formula is the same as in \cite{reineke_harder-narasimhan,joyce_configurations2} and we don't need to invoke the framework of \cite{kontsevich_stability}.

\begin{theorem}\label{WC-potential}
For any stability function $Z:\Ga\to\bC$, we have
\begin{equation}
A_0^{Q,W}=\prod_{\ell\sbs\bH}^\to A_{Z,\ell}^{Q,W}.
\end{equation}
\end{theorem}

\subsection{Relating stability data}
\label{relating SD}
Let us consider again the quiver $Q$ \eqref{hatQ} and its potential $W$ \eqref{W1}
\begin{equation}
\begin{tikzcd}
&1\ar[dl,->>>,"a_i"']\\
0\ar[rr,->>>,"c_i"]&&2\ar[ul,->>>,"b_i"']
\end{tikzcd}\qquad\qquad
W=\sum_{\si\in\Si_3}\sgn(\si)\cdot a_{\si1}b_{\si2}c_{\si3}.
\end{equation}
Let us consider the cut $I=\set{c_1,c_2,c_3}\sbs Q_1$.
Then the quiver $Q'=(Q_0,Q_1\ms I)$ has the form
\begin{equation}
\begin{tikzcd}
0&1\lar[->>>,"a_i"']&2\lar[->>>,"b_i"']
\end{tikzcd}
\end{equation}
which coincides with the quiver \eqref{Q1}
and the relations of the algebra $J_{W,I}$ coincide with \eqref{A-rel}.
We conclude that 
\begin{equation}
A=\End(T')^\op=J_{W,I},\qquad T'=\cO\oplus\cO(1)\oplus\cO(2),
\end{equation}
and $\cD=D^b(\Coh X)\iso D^b(\mmod J_{W,I})$, where $X=\bP^2$.
Similarly, let $\pi:Y=\om_X\to X$ be the projection and $T=\pi^*T'$.
Then
\begin{equation}
B=\End(T)^\op=J_W
\end{equation}
and $D^b_c(\Coh Y)\iso D^b(\mmod J_W)$.

\begin{remark}
The canonical bundle $Y=\om_{\bP^2}$ can be interpreted as a crepant resolution of $\bC^3/\bZ_3$.
This implies that \cite{bridgeland_mukai}
\begin{equation}
D^b_c(\Coh Y)\iso D^b_{\bZ_3}(\Coh_c\bC^3)\iso D^b(\mmod J_W).
\end{equation}
%where $J_{W}$ is the Jacobian algebra of $(Q,W)$, see below.
Note that the left hand side of \eqref{24} is the Euler form of $D^b_c(\Coh Y)$ and the right hand side of ~\eqref{24} is the Euler form of $D^b(\mmod J_W)$.
\end{remark}

We conclude from Lemma \ref{lm:forms1} that the Euler form $\hi_X$ of the derived category 
$D^b(\Coh X)\iso D^b(\mmod J_{W,I})$
is given by
\begin{gather}
\hi_X(E,F)=\bar \hi_X(d,d'),\qquad 
d=\udim E,\, d'=\udim F,\\
\bar\hi_X(d,d'):=\hi_Q(d,d')+\ga_I(d,d')+\ga_I(d',d),
\end{gather}
where $\ga_I(d,d')$ was defined in \eqref{ga-I}.
%This implies that

\begin{remark}
\label{phenom}
This seems to be a rather general phenomenon.
In \cite{beaujard_vafa} one considered a large family of examples of surfaces $X$ with a nef anticanonical bundle (called weak Fano surfaces there) and quivers with potential $(Q,W)$ such that $D^b_c(\Coh\om_X)\iso D^b(\mmod J_W)$.
We verified in all of these examples that there exists a special cut $I\sbs Q_1$ such that $D^b(\Coh X)\iso D^b(\mmod J_{W,I})$ and such that 
\begin{equation}
\hi_X(E,F)
=\hi_Q(d,d')+\ga_I(d,d')+\ga_I(d',d),\qquad 
d=\udim E,\, d'=\udim F.
\end{equation}
It would be interesting to give a conceptual explanation of this phenomenon.
\end{remark}

Let $Z:\Ga=\bZ Q_0\to\bC$ be a stability function on the abelian category $\mmod J_{W,I}$ and let $\si=(Z,\cP)$ be the corresponding stability condition on $\cD=D^b(\Coh X)\iso D^b(J_{W,I})$.
For any ray $\ell\sbs\bH$, we define stability data ~\eqref{A-si-l}
\begin{equation}\label{AZl-quiver}
A_{\si,\ell}
=1+\sum_{Z(d)\in\ell}
\bL^{\oh\bar\hi_X(d,d)}\frac{[R_Z(J_{W,I},d)]}{[\GL_d]}x^d
\end{equation}
and $A_{\si,-\ell}=A_{\si,\ell}$.
Note that the algebra $J_{W,I}$ has homological dimension $2$, hence we don't have a wall-crossing formula for the above stability data.
Nevertheless, we can formulate such a result if a stability condition on $D^b(\mmod J_{W,I})$ has global dimension $\le2$, similarly to Corollary \ref{cor:WC2}.
% under the assumption \ref{ass3}.
We will see later that there exist geometric stability conditions $\si$ on $D^b(\Coh X)$ with the heart $\cA_\si=\mmod J_{W,I}$.
These stability conditions have global dimension $2$ 
%The corresponding central charge $Z_\si:\Ga\to\bC$ satisfies the required condition 
by Theorem~ \ref{th:vanish3}.

%In the next section we will see that such stability conditions actually exist.

\begin{theorem}
\label{th:compat}
Let $Z:\Ga=\bZ Q_0\to\bC$ be a stability function such that the corresponding stability condition $\si$ on $D^b(\mmod J_{W,I})$ has global dimension $\le 2$.
Then, for any ray $\ell\sbs\bH$, we have
$$A_{\si,\ell}=A_{Z,\ell}^{Q,W}.$$
%Then the element $A_{Z,\ell}^{Q,W}$
%\eqref{A-Zl-QW} is equal to
%\begin{equation}
%A_{Z,\ell}=1+\sum_{Z(\ga)\in\ell}\bL^{\oh\hi_X(d,d)}
%\frac{[R_Z(J_{W,I},d)]}{[\GL_d]}x^d
%\end{equation}
\end{theorem}
\begin{proof}
Let $\ell_0\sbs V=\cone\set{Z(e_i)\col i\in Q_0}\sbs \bH$
be a ray and let $Z_0:\Ga\to\bC$ be such that $Z_0(e_i)\in\ell_0$ for all $i\in Q_0$.
Let $\si_0$ be the corresponding stability condition on 
$D^b(\mmod J_{W,I})$
(we will call it trivial as all objects in $\mmod J_{W,I}$ are automatically semistable).
Then we have \eqref{AZl-quiver}
$$A_{\si_0,\ell_0}
=1+\sum_{Z_0(d)\in\ell_0}\bL^{\oh\bar\hi_X(d,d)}
\frac{[R(J_{W,I},d)]}{[\GL_d]}x^d.$$
This series coincides with $A_0^{Q,W}$ \eqref{A0QW} as
$$
\bL^{\oh\bar\hi_X(d,d)}[R(J_{W,I},d)]
=\bL^{\oh\hi_Q(d,d)+\ga_I(d,d))}[R(J_{W,I},d)]
=\bL^{\oh\hi_Q(d,d)}\cdot [R(Q,d),\tr W].
$$
By the assumption that $\si$ has global dimension $\le2$, we obtain (see Corollary \ref{cor:WC2})
$$A_{\si_0,\ell_0}
=A_{\si,V}=\prod^\to_{\ell\sbs\bH}A_{\si,\ell}.$$
By Theorem \ref{WC-potential}, we have $A_{0}^{Q,W}=\prod^\to_{\ell\sbs\bH}A_{Z,\ell}^{Q,W}$.
We conclude that $A_{\si,\ell}=A_{Z,\ell}^{Q,W}$.
%Let us define
%$$A'_{Z,\ell}=1+\sum_{Z(\ga)\in\ell}\bL^{\oh\hi(d,d)}
%\frac{[R_Z(A,d)]}{[\GL_d]}x^d.$$
%Because of our assumption, we obtain (\cf \ref{})
%$$A'_{Z_0,\ell_0}=\prod_{\ell\in\bH}A'_{Z,\ell}.$$
%On the other hand, by Theorem \ref{}, we have
%$$A_{Z_0,\ell_0}=\prod_{\ell\in\bH}A_{Z,\ell}.$$
%By \eqref{}, we have
%$$\bL^{\oh\hi_Q(d,d)}\cdot [R(Q,d),\tr W]
%=\bL^{\oh(\hi_Q(d,d)+2\ga_I(d))}[R(J_{W,I},d)]
%=\bL^{\oh\hi(d,d)}[R(A,d)].$$
%This implies that $A_{Z_0,\ell_0}=A'_{Z_0,\ell_0}$.
%Because of the above wall-crossing formulas, we conclude that $A_{Z,\ell}=A'_{Z,\ell}$.
\end{proof}

\subsection{Gluing families of stability data}
\label{sec:gluing}
Let $X=\bP^2$ and $H$ be the ample divisor such that $H^2=1$.
We will identify $N^*_\bR(X)$ with $\bR^3$ by sending $a\mto(a_0,a_1H,a_2)$.
Let (\cf \S\ref{altern param})
$$\be+\bi\om=(s+\bi t)H,\qquad
s\in\bR,\,t\in\bR_{>0}.$$
We consider the abelian category $\cA_{\be,\om}=\Coh_{\om,\#\be\om}=\Coh_{H,\#s}$
and the stability function 
(see~ \eqref{Zst})
\begin{equation}
Z_{\be,\om}(E)=Z'_{s,t}(a)=\rbr{ya_0+sa_1-a_2}+\bi t\rbr{a_1-sa_0},
\qquad a=\ch(E),
\end{equation}
where $y=\oh(t^2-s^2)>-\oh s^2$.
Let $\si_{\be,\om}$ be the corresponding
stability condition on $D^b(\Coh X)\iso D^b(\mmod J_{W,I})$.
The following result is similar to \cite[\S4]{bousseau_scattering}, although one works with a different exceptional collection there.

\begin{lemma}
\label{lm:special param}
Assume that
\begin{equation}\label{50}
s>-\toh,\qquad t>0,\qquad s+t<0.
\end{equation}
Then the heart of $\si_{\be,\om}[\toh]$ is equal to $\mmod J_{W,I}$.
\end{lemma}
\begin{proof}
%Our goal is to find parameters $(s,t)$ such that the heart 
Let $\si=(Z,\cP):=\si_{\be,\om}$ and let $\cA_\si[\oh]$ denote the heart of $\si[\oh]$. 
%coincides with $\mmod J_{W,I}$ (\cf \cite{bousseau_scattering}).
Recall from \S\ref{sec:simple} that the category 
$\cA_0=\mmod J_{W,I}$ has simple objects 
\begin{equation*}
\cO,\, \Om_X(1)[1],\,\cO(-1)[2].
\end{equation*}
%Assume that we found parameters $(s,t)$ such that 
We will show that the objects
\begin{equation}\label{3obj}
\cO,\, \Om_X(1)[1],\,\cO(-1)[1]
\end{equation}
are contained in $\cA_{\be,\om}$ and
\begin{equation}\label{real parts}
\Re Z(\cO)<0,\qquad
\Re Z(\Om_X(1)[1])<0,\qquad
\Re Z(\cO(-1)[1])>0.
\end{equation}
Then $\Re Z$ is negative on all simple objects of $\cA_0$, 
hence all these simple objects are contained in $\cA_\si[\oh]$.
This implies that $\cA_0\sbs\cA_\si[\oh]$.
As both categories are hearts of bounded t-structures, we conclude that $\cA_0=\cA_\si[\oh]$ as required.

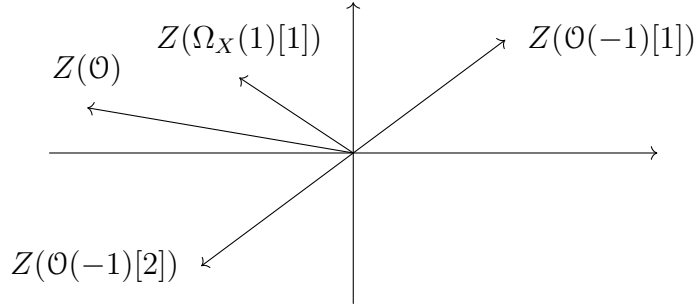
\begin{ctikz}
\draw[->] (-4,0)--(4,0);
\draw[->] (0,-2)--(0,2);
\draw[->](0,0)--(2,1.5)
	node[label={right:$Z(\cO(-1)[1])$}]{};
\draw[->](0,0)--(-2,-1.5)
	node[label={left:$Z(\cO(-1)[2])$}]{};
\draw[->](0,0)--(-1.5,1)
	node[label={$Z(\Om_X(1)[1])$}]{};
\draw[->](0,0)--(-3.5,.6)
	node[label={$Z(\cO)$}]{};
%\draw(1,1)node[label={[label distance=-1ex]right:$S_0$}]	{\bul};
%\draw(-1,0)node[label={[label distance=-1ex]below:$S_1$}]	{\bul};
%\draw[-,dashed](-2,2)--(2,-2);
%\draw[pattern=north east lines](3,0)--(0,0)--(2,-2);
%\draw[pattern=north east lines](-2,2)--(0,0)--(2,2);
\end{ctikz}

Let us show first that $\Om_X(1)$ is $\mu_H$-semistable.
As $\Om_X(1)=\cQ\dual$, where $\cQ$ was defined in \eqref{Qseq},
it is enough to show that $\cQ$ is $\mu_H$-semistable.
We have $\ch \cQ=(2,1,-1/2)$, hence $\mu_H(\cQ)=1/2$.
There are no non-trivial morphisms $\cO(k)\to \cQ$ for $k\ge1$ because of the exact sequence \eqref{Qseq}.
Therefore $\cQ$ is indeed $\mu_H$-semistable.

We have $$\mu_H(\cO)=0,\qquad \mu_H(\Om_X(1))=-1/2,\qquad 
\mu_H(\cO(-1))=-1.$$
As $-\oh<s<0$, we conclude that $\cO\in\Coh_{H,>s}$ and $\Om_X(1),\cO(-1)\in\Coh_{H,\le s}$.
Therefore all objects in \eqref{3obj} are contained in $\Coh_{H,\#s}=\cA_{\be,\om}$.
Let us check that \eqref{real parts} is satisfied.
Recall that ~\eqref{simple chern}
\begin{equation*}
\ch\cO=(1,0,0),\qquad \ch\Om_X(1)[1]=(-2,1,\toh),\qquad \ch\cO(-1)=(1,-1,\toh).
\end{equation*}
Therefore we require (with $y=\oh(t^2-s^2)$)
\begin{equation*}
y<0,\qquad -2y+s-\toh<0,\qquad y-s-\toh<0.
\end{equation*}
Under the assumptions $-\oh<s<0$ and $-\oh s^2<y<0$, the other conditions follow.
As $y=\oh(t^2-s^2)$, these assumptions are equivalent to
\eqref{50}.
%\begin{equation}
%s>-\toh,\qquad t>0,\qquad s+t<0.
%\end{equation}
\end{proof}

Choosing parameters $(s,t)$ as in the previous lemma,
we obtain a special geometric stability condition $\si=\si_{\be,\om}[\oh]$ and central charge $Z=Z_{\be,\om}[\oh]=\bi\inv Z_{\be,\om}$
such that the heart of $\si$ is equal to $\cA_\si=\mmod J_{W,I}$.
The corresponding linear map $Z:\cN(X)\iso\bZ Q_0\to\bZ$ 
satisfies the assumptions of Theorem \ref{th:compat}, hence we obtain $A_{\si,\ell}=A_{Z,\ell}^{Q,W}$.
On the other hand, the series $A_{\si,\ell}$ defined in \eqref{AZl-quiver} (for the category $D^b(\mmod J_{W,I})$) and
$A_{\si,\ell}$ defined in \eqref{A-si-l} (for the category $D^b(\Coh X)$) are the same.

Note that stability data $A_\si=(A_{\si,\ell})_\ell$ is defined for every special geometric stability condition $\si=(Z,\cP)\in\Stab^+_{\cl}(\cD)$
\eqref{sg stab}.
The discriminant $\De$ is negative-definite on $\Ker Z$, hence by Theorem \ref{th:unique lift} stability data $A_\si$ depends just on $Z\in\cP^+(X)\sbs\Hom(\Ga,\bC)$ and we will denote it by $A_Z$.
On the other hand, stability data $A_{Z}^{Q,W}$ is defined for every $Z\in\cP(Q):=\GL_2^+(\bR)\cdot\bH^{Q_0}\sbs\Hom(\bZ Q_0,\bC)\iso\Hom(\Ga,\bC)$.
We conclude from Theorem ~\ref{th:compat} and Lemma \ref{lm:special param}
that we can glue together stability data
for the category $D^b(\Coh X)$ and stability data for the quiver with potential $(Q,W)$ and define a family of stability data parametrized by $\cP^+(X)\cup\cP(Q)$.

\bibliography{../../../tex/papers}

\providecommand{\bysame}{\leavevmode\hbox to3em{\hrulefill}\thinspace}
\providecommand{\href}[2]{#2}
\begin{thebibliography}{10}

\bibitem{angelerihugel_handbook}
Lidia Angeleri~H{\"u}gel, Dieter Happel, and Henning Krause (eds.),
  \emph{{H}andbook of tilting theory}, London Mathematical Society Lecture Note
  Series, vol. 332, Cambridge University Press, Cambridge, 2007.

\bibitem{arcara_bridgeland}
Daniele Arcara and Aaron Bertram, \emph{Bridgeland-stable moduli spaces for
  {$K$}-trivial surfaces}, J. Eur. Math. Soc. (JEMS) \textbf{15} (2013), 1--38,
  \href{http://arxiv.org/abs/0708.2247}{{\ttfamily arXiv:0708.2247}}, With an
  appendix by Max Lieblich.

\bibitem{aspinwall_da}
Paul~S. Aspinwall and Michael~R. Douglas, \emph{D-brane stability and
  monodromy}, J. High Energy Phys. (2002), no.~5, no. 31, 35.

\bibitem{bayer_short}
Arend Bayer, \emph{A short proof of the deformation property of {B}ridgeland
  stability conditions}, Math. Ann. \textbf{375} (2019), no.~3-4, 1597--1613,
  \href{http://arxiv.org/abs/1606.02169}{{\ttfamily arXiv:1606.02169}}.

\bibitem{bayer_spacea}
Arend Bayer and Emanuele Macr\`\i, \emph{The space of stability conditions on
  the local projective plane}, Duke Math. J. \textbf{160} (2011), no.~2,
  263--322, \href{http://arxiv.org/abs/0912.0043}{{\ttfamily arXiv:0912.0043}}.

\bibitem{bayer_space}
Arend Bayer, Emanuele Macr\`\i, and Paolo Stellari, \emph{The space of
  stability conditions on abelian threefolds, and on some {C}alabi-{Y}au
  threefolds}, Invent. Math. \textbf{206} (2016), no.~3, 869--933,
  \href{http://arxiv.org/abs/1410.1585}{{\ttfamily arXiv:1410.1585}}.

\bibitem{bayer_bridgeland}
Arend Bayer, Emanuele Macr\`\i, and Yukinobu Toda, \emph{Bridgeland stability
  conditions on threefolds {I}: {B}ogomolov-{G}ieseker type inequalities}, J.
  Algebraic Geom. \textbf{23} (2014), no.~1, 117--163,
  \href{http://arxiv.org/abs/1103.5010}{{\ttfamily arXiv:1103.5010}}.

\bibitem{beaujard_vafa}
Guillaume Beaujard, Jan Manschot, and Boris Pioline, \emph{Vafa-{W}itten
  invariants from exceptional collections}, Comm. Math. Phys. \textbf{385}
  (2021), no.~1, 101--226, \href{http://arxiv.org/abs/2004.14466}{{\ttfamily
  arXiv:2004.14466}}.

\bibitem{BBD}
A.~A. Beilinson, J.~Bernstein, and P.~Deligne, \emph{{F}aisceaux pervers},
  Analysis and topology on singular spaces, I (Luminy, 1981), Ast\'erisque,
  vol. 100, Soc. Math. France, Paris, 1982, pp.~5--171.

\bibitem{bondal_representations}
A.~I. Bondal, \emph{Representations of associative algebras and coherent
  sheaves}, Izv. Akad. Nauk SSSR Ser. Mat. \textbf{53} (1989), no.~1, 25--44.

\bibitem{borisov_class}
Lev~A. Borisov, \emph{The class of the affine line is a zero divisor in the
  {G}rothendieck ring}, J. Algebraic Geom. \textbf{27} (2018), no.~2, 203--209,
  \href{http://arxiv.org/abs/1412.6194}{{\ttfamily arXiv:1412.6194}}.

\bibitem{bousseau_scattering}
Pierrick Bousseau, \emph{{S}cattering diagrams, stability conditions, and
  coherent sheaves on $\mathbb{P}^2$}, 2019,
  \href{http://arxiv.org/abs/1909.02985}{{\ttfamily arXiv:1909.02985}}.

\bibitem{bridgeland_stability}
Tom Bridgeland, \emph{{S}tability conditions on triangulated categories}, Ann.
  of Math. (2) \textbf{166} (2007), no.~2, 317--345,
  \href{http://arxiv.org/abs/math/0212237}{{\ttfamily arXiv:math/0212237}}.

\bibitem{bridgeland_stabilityb}
\bysame, \emph{Stability conditions on {$K3$} surfaces}, Duke Math. J.
  \textbf{141} (2008), no.~2, 241--291,
  \href{http://arxiv.org/abs/math/0307164}{{\ttfamily arXiv:math/0307164}}.

\bibitem{bridgeland_introduction}
\bysame, \emph{{A}n introduction to motivic {H}all algebras}, Adv. Math.
  \textbf{229} (2012), no.~1, 102--138,
  \href{http://arxiv.org/abs/1002.4372}{{\ttfamily arXiv:1002.4372}}.

\bibitem{bridgeland_scattering}
\bysame, \emph{Scattering diagrams, {H}all algebras and stability conditions},
  Algebraic Geometry (2017), 523--561,
  \href{http://arxiv.org/abs/1603.00416}{{\ttfamily arXiv:1603.00416}}.

\bibitem{bridgeland_riemann}
\bysame, \emph{Riemann-{H}ilbert problems from {D}onaldson-{T}homas theory},
  Invent. Math. \textbf{216} (2019), no.~1, 69--124,
  \href{http://arxiv.org/abs/1611.03697}{{\ttfamily arXiv:1611.03697}}.

\bibitem{bridgeland_geometry}
\bysame, \emph{Geometry from {D}onaldson-{T}homas invariants}, Integrability,
  quantization, and geometry. {II}, Proc. Sympos. Pure Math., vol. 103, Amer.
  Math. Soc., Providence, RI, 2021,
  \href{http://arxiv.org/abs/1912.06504}{{\ttfamily arXiv:1912.06504}},
  pp.~1--66.

\bibitem{bridgeland_mukai}
Tom Bridgeland, Alastair King, and Miles Reid, \emph{{T}he {M}c{K}ay
  correspondence as an equivalence of derived categories}, J. Amer. Math. Soc.
  \textbf{14} (2001), no.~3, 535--554 (electronic),
  \href{http://arxiv.org/abs/math/9908027}{{\ttfamily arXiv:math/9908027}}.

\bibitem{bridgeland_helices}
Tom Bridgeland and David Stern, \emph{Helices on del {P}ezzo surfaces and
  tilting {C}alabi-{Y}au algebras}, Adv. Math. \textbf{224} (2010), no.~4,
  1672--1716, \href{http://arxiv.org/abs/0909.1732}{{\ttfamily
  arXiv:0909.1732}}.

\bibitem{cao_genus}
Yalong Cao, Davesh Maulik, and Yukinobu Toda, \emph{Genus zero
  {G}opakumar-{V}afa type invariants for {C}alabi-{Y}au 4-folds}, Adv. Math.
  \textbf{338} (2018), 41--92,
  \href{http://arxiv.org/abs/1801.02513}{{\ttfamily arXiv:1801.02513}}.

\bibitem{deligne_theorie}
Pierre Deligne, \emph{{T}h\'eorie de {H}odge. {II}}, Inst. Hautes \'Etudes Sci.
  Publ. Math. \textbf{40} (1971), 5--57.

\bibitem{douglas_geometry}
Michael Douglas, Jerome Gauntlett, and Mark Gross (eds.), \emph{{T}he geometry
  of string theory}, Clay Mathematics Proceedings, vol.~3, Providence, RI,
  American Mathematical Society, 2004.

\bibitem{dyckerhoff_highera}
Tobias Dyckerhoff and Mikhail Kapranov, \emph{{H}igher {S}egal spaces}, Lecture
  Notes in Mathematics, vol. 2244, Springer, 2019,
  \href{http://arxiv.org/abs/1212.3563}{{\ttfamily arXiv:1212.3563}}.

\bibitem{fan_contractibility}
Yu-Wei Fan, Chunyi Li, Wanmin Liu, and Yu~Qiu, \emph{Contractibility of space
  of stability conditions on the projective plane via global dimension
  function}, \href{http://arxiv.org/abs/2001.11984}{{\ttfamily
  arXiv:2001.11984}}.

\bibitem{fulton_intersection}
William Fulton, \emph{{I}ntersection theory}, second ed., Ergebnisse der
  Mathematik und ihrer Grenzgebiete, Folge 3, vol.~2, Springer-Verlag, Berlin,
  1998.

\bibitem{ginzburg_calabi}
Victor Ginzburg, \emph{{C}alabi-{Y}au algebras}, 2006,
  \href{http://arxiv.org/abs/math/0612139}{{\ttfamily arXiv:math/0612139}}.

\bibitem{gross_canonical}
Mark Gross, Paul Hacking, Sean Keel, and Maxim Kontsevich, \emph{Canonical
  bases for cluster algebras}, Journal of the American Mathematical Society
  \textbf{31} (2017), no.~2, 497--608,
  \href{http://arxiv.org/abs/1411.1394}{{\ttfamily arXiv:1411.1394}}.

\bibitem{gross_real}
Mark Gross and Bernd Siebert, \emph{From real affine geometry to complex
  geometry}, Ann. of Math. (2) \textbf{174} (2011), no.~3, 1301--1428,
  \href{http://arxiv.org/abs/math/0703822}{{\ttfamily arXiv:math/0703822}}.

\bibitem{hille_fourier}
Lutz Hille and Michel Van~den Bergh, \emph{Fourier-{M}ukai transforms},
  Handbook of tilting theory, London Math. Soc. Lecture Note Ser., vol. 332,
  Cambridge Univ. Press, Cambridge, 2007, pp.~147--177.

\bibitem{ikeda_q}
Akishi Ikeda and Yu~Qiu, \emph{q-stability conditions on
  {C}alabi-{Y}au-$\mathbb{X}$ categories},
  \href{http://arxiv.org/abs/1807.00469}{{\ttfamily arXiv:1807.00469}}.

\bibitem{joyce_configurations1}
Dominic Joyce, \emph{Configurations in abelian categories. {I}. {B}asic
  properties and moduli stacks}, Adv. Math. \textbf{203} (2006), no.~1,
  194--255, \href{http://arxiv.org/abs/math/0312190}{{\ttfamily
  arXiv:math/0312190}}.

\bibitem{joyce_configurations2}
\bysame, \emph{{C}onfigurations in abelian categories. {II}. {R}ingel-{H}all
  algebras}, Adv. Math. \textbf{210} (2007), no.~2, 635--706,
  \href{http://arxiv.org/abs/math/0503029}{{\ttfamily arXiv:math/0503029}}.

\bibitem{joyce_configurations3}
\bysame, \emph{Configurations in abelian categories. {III}. {S}tability
  conditions and identities}, Adv. Math. \textbf{215} (2007), no.~1, 153--219,
  \href{http://arxiv.org/abs/math/0410267}{{\ttfamily arXiv:math/0410267}}.

\bibitem{joyce_motivic}
\bysame, \emph{{M}otivic invariants of {A}rtin stacks and `stack functions'},
  Q. J. Math. \textbf{58} (2007), no.~3, 345--392,
  \href{http://arxiv.org/abs/math/0509722}{{\ttfamily arXiv:math/0509722}}.

\bibitem{joyce_configurations4}
\bysame, \emph{Configurations in abelian categories. {IV}. {I}nvariants and
  changing stability conditions}, Adv. Math. \textbf{217} (2008), no.~1,
  125--204, \href{http://arxiv.org/abs/math/0410268}{{\ttfamily
  arXiv:math/0410268}}.

\bibitem{joyce_theory}
Dominic Joyce and Yinan Song, \emph{{A} theory of generalized
  {D}onaldson-{T}homas invariants}, Mem. Amer. Math. Soc. \textbf{217} (2012),
  no.~1020, iv+199, \href{http://arxiv.org/abs/0810.5645}{{\ttfamily
  arXiv:0810.5645}}.

\bibitem{keller_chain}
Bernhard Keller, \emph{{C}hain complexes and stable categories}, Manuscripta
  Math. \textbf{67} (1990), no.~4, 379--417.

\bibitem{keller_deriveda}
Bernhard Keller and Dong Yang, \emph{Derived equivalences from mutations of
  quivers with potential}, Adv. Math. \textbf{226} (2011), no.~3, 2118--2168,
  \href{http://arxiv.org/abs/0906.0761}{{\ttfamily arXiv:0906.0761}}.

\bibitem{kontsevich_affine}
Maxim Kontsevich and Yan Soibelman, \emph{{A}ffine structures and
  non-{A}rchimedean analytic spaces}, The unity of mathematics, Progr. Math.,
  vol. 244, Birkh\"auser Boston, Boston, MA, 2006,
  \href{http://arxiv.org/abs/math.AG/0406564}{{\ttfamily
  arXiv:math.AG/0406564}}, pp.~321--385.

\bibitem{kontsevich_stability}
\bysame, \emph{{S}tability structures, motivic {D}onaldson-{T}homas invariants
  and cluster transformations}, 2008,
  \href{http://arxiv.org/abs/0811.2435}{{\ttfamily arXiv:0811.2435}}.

\bibitem{kontsevich_cohomological}
\bysame, \emph{{C}ohomological {H}all algebra, exponential {H}odge structures
  and motivic {D}onaldson-{T}homas invariants}, Commun. Num. Theor. Phys.
  \textbf{5} (2011), 231--352, \href{http://arxiv.org/abs/1006.2706}{{\ttfamily
  arXiv:1006.2706}}.

\bibitem{kontsevich_wall}
\bysame, \emph{{W}all-crossing structures in {D}onaldson-{T}homas invariants,
  integrable systems and {M}irror {S}ymmetry}, Lecture Notes of the Unione
  Matematica Italiana, Springer International Publishing, 2014,
  \href{http://arxiv.org/abs/1303.3253}{{\ttfamily arXiv:1303.3253}},
  pp.~197--308.

\bibitem{kontsevich_analyticity}
\bysame, \emph{Analyticity and resurgence in wall-crossing formulas},  (2020),
  \href{http://arxiv.org/abs/2005.10651}{{\ttfamily arXiv:2005.10651}}.

\bibitem{lazarsfeld_positivity}
Robert Lazarsfeld, \emph{Positivity in algebraic geometry {I}}, Ergebnisse der
  Mathematik und ihrer Grenzgebiete. 3. Folge. A Series of Modern Surveys in
  Mathematics, vol.~48, Springer-Verlag, Berlin, 2004, Classical setting: line
  bundles and linear series.

\bibitem{li_smoothness}
Chunyi Li and Xiaolei Zhao, \emph{Smoothness and {P}oisson structures of
  {B}ridgeland moduli spaces on {P}oisson surfaces}, Math. Z. \textbf{291}
  (2019), no.~1-2, 437--447, \href{http://arxiv.org/abs/1612.02729}{{\ttfamily
  arXiv:1612.02729}}.

\bibitem{lieblich_moduli}
Max Lieblich, \emph{Moduli of complexes on a proper morphism}, J. Algebraic
  Geom. \textbf{15} (2006), no.~1, 175--206,
  \href{http://arxiv.org/abs/math/0502198}{{\ttfamily arXiv:math/0502198}}.

\bibitem{mozgovoy_intersection}
Jan Manschot and Sergey Mozgovoy, \emph{{I}ntersection cohomology of moduli
  spaces of sheaves on surfaces}, Selecta Math. (N.S.) \textbf{24} (2018),
  no.~5, 3889--3926, \href{http://arxiv.org/abs/1512.04076}{{\ttfamily
  arXiv:1512.04076}}.

\bibitem{martin_class}
Nicolas Martin, \emph{The class of the affine line is a zero divisor in the
  {G}rothendieck ring: an improvement}, C. R. Math. Acad. Sci. Paris
  \textbf{354} (2016), no.~9, 936--939,
  \href{http://arxiv.org/abs/1604.06703}{{\ttfamily arXiv:1604.06703}}.

\bibitem{meinhardt_donaldsona}
Sven Meinhardt, \emph{{D}onaldson-{T}homas invariants vs. intersection
  cohomology for categories of homological dimension one}, 2015,
  \href{http://arxiv.org/abs/1512.03343}{{\ttfamily arXiv:1512.03343}}.

\bibitem{meinhardt_donaldson}
Sven Meinhardt and Markus Reineke, \emph{Donaldson-{T}homas invariants versus
  intersection cohomology of quiver moduli}, J. Reine Angew. Math. \textbf{754}
  (2019), 143--178, \href{http://arxiv.org/abs/1411.4062}{{\ttfamily
  arXiv:1411.4062}}.

\bibitem{mou_scattering}
Lang Mou, \emph{Scattering diagrams of quivers with potentials and mutations},
  \href{http://arxiv.org/abs/1910.13714}{{\ttfamily arXiv:1910.13714}}.

\bibitem{mozgovoy_translation}
Sergey Mozgovoy, \emph{{T}ranslation quiver varieties}, 2019,
  \href{http://arxiv.org/abs/1911.01788}{{\ttfamily arXiv:1911.01788}}.

\bibitem{mozgovoy_quiver}
\bysame, \emph{Quiver representations in abelian categories}, J. Algebra
  \textbf{541} (2020), 35--50,
  \href{http://arxiv.org/abs/1811.12935}{{\ttfamily arXiv:1811.12935}}.

\bibitem{mozgovoy_attractor}
Sergey Mozgovoy and Boris Pioline, \emph{Attractor invariants, brane tilings
  and crystals}, 2021, \href{http://arxiv.org/abs/2012.14358}{{\ttfamily
  arXiv:2012.14358}}.

\bibitem{mozgovoy_intersectiona}
Sergey Mozgovoy and Markus Reineke, \emph{{I}ntersection cohomology of moduli
  spaces of vector bundles over curves}, 2015,
  \href{http://arxiv.org/abs/1512.04076}{{\ttfamily arXiv:1512.04076}}.

\bibitem{piyaratne_moduli}
Dulip Piyaratne and Yukinobu Toda, \emph{Moduli of {B}ridgeland semistable
  objects on 3-folds and {D}onaldson-{T}homas invariants}, J. Reine Angew.
  Math. \textbf{747} (2019), 175--219,
  \href{http://arxiv.org/abs/1504.01177}{{\ttfamily arXiv:1504.01177}}.

\bibitem{qiu_global}
Yu~Qiu, \emph{Global dimension function on stability conditions and {G}epner
  equations}, \href{http://arxiv.org/abs/1807.00010}{{\ttfamily
  arXiv:1807.00010}}.

\bibitem{quillen_higher}
Daniel Quillen, \emph{Higher algebraic {$K$}-theory. {I}}, Algebraic
  {$K$}-theory, {I}: {H}igher {$K$}-theories, vol. 341, 1973, pp.~85--147.
  Lecture Notes in Math.

\bibitem{reineke_harder-narasimhan}
Markus Reineke, \emph{{T}he {H}arder-{N}arasimhan system in quantum groups and
  cohomology of quiver moduli}, Invent. Math. \textbf{152} (2003), no.~2,
  349--368, \href{http://arxiv.org/abs/math/0204059}{{\ttfamily
  arXiv:math/0204059}}.

\bibitem{rickard_morita}
Jeremy Rickard, \emph{Morita theory for derived categories}, J. London Math.
  Soc. (2) \textbf{39} (1989), no.~3, 436--456.

\bibitem{saito_introduction}
Morihiko Saito, \emph{{I}ntroduction to mixed {H}odge modules}, Ast\'erisque
  \textbf{179-180} (1989), 145--162, Actes du Colloque de Th{\'e}orie de Hodge
  (Luminy, 1987).

\bibitem{segal_a}
Ed~Segal, \emph{{T}he ${A}_\infty$ deformation theory of a point and the
  derived categories of local {C}alabi-{Y}aus}, J. Algebra \textbf{320} (2008),
  no.~8, 3232--3268, \href{http://arxiv.org/abs/math/0702539}{{\ttfamily
  arXiv:math/0702539}}.

\bibitem{toda_moduli}
Yukinobu Toda, \emph{Moduli stacks and invariants of semistable objects on
  {$K3$} surfaces}, Adv. Math. \textbf{217} (2008), no.~6, 2736--2781,
  \href{http://arxiv.org/abs/math/0703590}{{\ttfamily arXiv:math/0703590}}.

\bibitem{toda_stabilityb}
\bysame, \emph{Stability conditions and extremal contractions}, Math. Ann.
  \textbf{357} (2013), no.~2, 631--685,
  \href{http://arxiv.org/abs/1204.0602}{{\ttfamily arXiv:1204.0602}}.

\bibitem{toen_grothendieck}
B.~To\"en, \emph{{G}rothendieck rings of {A}rtin n-stacks}, 2005,
  \href{http://arxiv.org/abs/math/0509098}{{\ttfamily arXiv:math/0509098}}.

\end{thebibliography}
\bibliographystyle{../../../tex/hamsplain}
\end{document}